\documentclass[a4paper]{amsart}

\usepackage[margin=3cm]{geometry}

\usepackage[british]{babel}

\usepackage{hyperref}
\hypersetup{
  colorlinks = true,
  allcolors = {ForestGreen},
}

\usepackage[svgnames]{xcolor}

\usepackage[nameinlink]{cleveref}

\usepackage{eucal}

\usepackage{tikz,tikz-cd,tikz-3dplot}
\usetikzlibrary{positioning,arrows,shapes,snakes,decorations.markings,calc,patterns}

\usepackage{enumitem}

\usepackage{xparse}

\usepackage{amsthm,amssymb,amsmath,mathtools}

\makeatletter \def\subsection{\@startsection{subsection}{3}%
  \z@{.5\linespacing\@plus.7\linespacing}{.7\linespacing}%
  {\normalfont\bfseries}} \def\subsubsection{\@startsection{subsubsection}{3}%
  \z@{.5\linespacing\@plus.7\linespacing}{.7\linespacing}%
  {\normalfont\itshape}} \def\paragraph{\@startsection{paragraph}{3}%
  \z@{.5\linespacing\@plus.7\linespacing}{.7\linespacing}%
  {\normalfont\itshape}} \makeatother

\DeclareSymbolFont{sfoperators}{OT1}{cmss}{m}{n}
\DeclareSymbolFontAlphabet{\mathsf}{sfoperators}%
\makeatletter%
\def\operator@font{\mathgroup\symsfoperators}%
\makeatother%

\usepackage{thmtools}

\numberwithin{equation}{subsection}

\declaretheorem[style=theorem,qed=\qedsymbol,numberwithin=subsection]{theorem}
\declaretheorem[style=theorem,qed=\qedsymbol,sibling=theorem]{proposition}
\declaretheorem[style=theorem,qed=\qedsymbol,sibling=theorem]{lemma}
\declaretheorem[style=theorem,qed=\qedsymbol,sibling=theorem]{corollary}

\declaretheorem[style=definition,sibling=theorem,qed=\qedsymbol]{definition}
\declaretheorem[style=definition,numbered=no,qed=\qedsymbol,name=Definition]{definition*}
\declaretheorem[style=definition,name=Definition-Proposition,sibling=theorem,qed=\qedsymbol]{defprop}

\declaretheorem[style=remark,qed=\qedsymbol,sibling=theorem]{notation}
\declaretheorem[style=remark,qed=\qedsymbol,sibling=theorem]{remark}
\declaretheorem[style=remark,numbered=no,qed=\qedsymbol,name=Remark]{remark*}
\declaretheorem[style=remark,qed=\qedsymbol,sibling=theorem]{example}
\declaretheorem[style=remark,numbered=no,qed=\qedsymbol,name=Example]{example*}

\declaretheorem[style=theorem,qed=\qedsymbol,name=Theorem]{theorem*}

\newcommand{\kk}{\mathbf{k}}

\renewcommand{\AA}{\mathbb{A}}

\newcommand{\ZZ}{\mathbb{Z}}

\newcommand{\category}[1]{\mathcal{#1}}

\newcommand{\A}{\category{A}}
\newcommand{\F}{\category{F}}
\newcommand{\M}{\category{M}}
\newcommand{\U}{\category{U}}
\newcommand{\W}{\category{W}}
\newcommand{\X}{\category{X}}

\NewDocumentCommand{\Ho}{o}{\operatorname{Ho}(#1)}
\NewDocumentCommand{\mmod}{m}{\operatorname{mod}(#1)}
\NewDocumentCommand{\perf}{m}{\operatorname{perf}(#1)}
\NewDocumentCommand{\Df}{m}{\operatorname{D}^{\mathrm{fd}}(#1)}

\NewDocumentCommand{\wS}{o o o}{%
  \operatorname{S}_{\IfValueTF{#2}{#2}{\bullet}}%
  \IfValueT{#1}{^{\langle{#1}\rangle}}%
  \IfValueT{#3}{\!\left(#3\right)}
}

\NewDocumentCommand{\Aus}{s m m}{A_{#2,#3}\IfBooleanT{#1}{^{\vee}}}
\NewDocumentCommand{\dgAus}{s m m}{\A_{#2,#3}\IfBooleanT{#1}{^{\vee}}}
\NewDocumentCommand{\dgBruhat}{s m m}{\mathcal{B}_{#2,#3}\IfBooleanT{#1}{^{\vee}}}

\newcommand{\theSurface}{\Sigma} \newcommand{\theDivisor}{D}
\newcommand{\theDisk}{\mathbb{D}}
\NewDocumentCommand{\Sym}{m m}{\operatorname{Sym}^{#1}(#2)}

\NewDocumentCommand{\WF}{s m m}{\IfBooleanTF{#1}{\overline{\W}}{\W}_{#2}^{(#3)}}
\NewDocumentCommand{\UF}{m m}{\U_{#1}^{(#2)}}

\NewDocumentCommand{\Fun}{o o o o}{%
  \operatorname{Fun}\IfValueT{#3}{^{#3}}\IfValueT{#4}{_{#4}}\IfValueT{#1}{(#1,#2)}%
}

\RenewDocumentCommand{\hom}{O{-} O{-} o o}{%
  \operatorname{hom}\IfValueT{#3}{_{#3}}\IfValueT{#4}{^{#4}}({#1},{#2})%
}

\NewDocumentCommand{\Hom}{O{-} O{-} o o}{%
  \operatorname{Hom}\IfValueT{#3}{_{#3}}\IfValueT{#4}{^{#4}}({#1},{#2})%
}

\NewDocumentCommand{\End}{O{-} o o}{%
  \operatorname{End}\IfValueT{#2}{_{#2}}\IfValueT{#3}{^{#3}}({#1})%
}

\NewDocumentCommand{\morphism}{s o m m o O{normal}}{%
  \begin{tikzcd}[ampersand replacement=\&, column sep=#6]
    \IfValueT{#2}{#2\colon}
    #3\IfBooleanTF{#1}{\IfValueTF{#5}{\rar[mapsto]{#5}}{\rar[mapsto]}}{\IfValueTF{#5}{\rar{#5}}{\rar}}\&
    #4
  \end{tikzcd}
}

\NewDocumentCommand{\functor}{s o m m o O{normal}}{%
  \begin{tikzcd}[ampersand replacement=\&, column sep=#6]
    \IfValueT{#2}{#2\colon}
    #3\IfBooleanTF{#1}{\IfValueTF{#5}{\rar[mapsto]{#5}}{\rar[mapsto]}}{\IfValueTF{#5}{\rar{#5}}{\rar}}\&
    #4
  \end{tikzcd}
}

\NewDocumentCommand{\SymGrp}{m}{\mathfrak{S}_{#1}}

\NewDocumentCommand{\set}{m o}{\{#1%
    \IfNoValueF{#2}{\,|\,#2}%
    \}%
}

\NewDocumentCommand{\Serre}{o}{\mathbb{S}\IfValueT{#1}{_{#1}}}

\makeatletter
\DeclareDocumentCommand\Stops{O{} O{}}{
  \Lambda\def\@tempa{#2}\ifx\@tempa\@empty\else_{#2}\fi\def\@tempa{#1}\ifx\@tempa\@empty\else^{(#1)}\fi
}
\makeatother

\newcommand{\qquand}{\qquad\text{and}\qquad}
\newcommand{\coloneqq}{:=}

\NewDocumentCommand{\mychoose}{s m m}{\left(\begin{smallmatrix}\mathbf{#2}\\#3\end{smallmatrix}\right)\IfBooleanT{#1}{_0}}
\NewDocumentCommand{\multichoose}{s m m}{\left\{\begin{smallmatrix}\mathbf{#2}\\#3\end{smallmatrix}\right\}}
\NewDocumentCommand{\degchoose}{s m m}{\multichoose{n}{d}^\flat}

\DeclareMathOperator{\domdim}{dom.dim}
\DeclareMathOperator{\gldim}{gl.dim}

\NewDocumentCommand{\inv}{m}{\operatorname{inv}(#1)}

\newcommand{\St}{\mathbf{St}_\infty}
\newcommand{\op}{{\mathrm{op}}}

\newcommand{\ParacyclicCat}{\mathbf{\Lambda{}}}

\RenewDocumentCommand{\L}{s m}{L_{#2}\IfBooleanT{#1}{^\vee}}

\NewDocumentCommand{\PL}{m o}{\L{\d{#1}[#2]}}
\RenewDocumentCommand{\d}{m o}{{#1}\IfValueT{#2}{_{#2}}}

\newcommand{\n}{\mathbf{n}}
\newcommand{\NN}{\mathbb{N}}

\NewDocumentCommand{\rk}{m}{\operatorname{rk}(#1)}

\NewDocumentCommand{\interval}{m m}{[e,\pi^{#2 #1}_0]}

\author[T. Dyckerhoff]{Tobias Dyckerhoff}
\address[Dyckerhoff]{Fachbereich Mathematik\\
  Universit\"{a}t Hamburg\\
  Bundesstra{\ss}e 55\\
  D-20146 Hamburg, Germany}%
\email{tobias.dyckerhoff@uni-hamburg.de}%
\urladdr{https://www.math.uni-hamburg.de/home/dyckerhoff/}

\author[G. Jasso]{Gustavo Jasso}
\address[Jasso]{Mathematisches Institut\\
  Rheinische Friedrich-Wilhelms-Universit\"{a}t Bonn\\
  Endenicher Allee 60\\
  D-53115 Bonn, Germany}%
\email{gjasso@math.uni-bonn.de}%
\urladdr{http://gustavo.jasso.info}

\author[Y. Lekili]{Yank{\i} Lekili}
\address[Lekili]{Department of Mathematics\\
  King's College London\\
  Strand \\
  London WC2R 2LS, United Kingdom}%
\email{yanki.lekili@kcl.ac.uk}%
\urladdr{https://lekili.duckdns.org/}

\title[The symplectic geometry of higher Auslander algebras]{The symplectic
  geometry of higher Auslander algebras: \\Symmetric products of disks}



\begin{document}

\maketitle

\begin{abstract}
  We show that the perfect derived categories of Iyama's $d$-dimensional
  Auslander algebras of type $\AA$ are equivalent to the partially wrapped Fukaya
  categories of the $d$-fold symmetric product of the $2$-dimensional unit disk
  with finitely many stops on its boundary. Furthermore, we observe that Koszul
  duality provides an equivalence between the partially wrapped Fukaya
  categories associated to the $d$-fold symmetric product of the disk and those
  of its $(n-d)$-fold symmetric product; this observation leads to a symplectic
  proof of a theorem of Beckert concerning the derived Morita equivalence
  between the corresponding higher Auslander algebras of type $\AA$.

  As a byproduct of our results, we deduce that the partially wrapped Fukaya
  categories associated to the $d$-fold symmetric product of the disk organise
  into a paracyclic object equivalent to the $d$-dimensional Waldhausen
  $\wS$-construction, a simplicial space whose geometric realisation provides
  the $d$-fold delooping of the connective algebraic $K$-theory space of the
  ring of coefficients.
\end{abstract}

\vspace{4.5em}

\setcounter{tocdepth}{2}
\tableofcontents

\newpage

\section*{Introduction}

Let $n$ and $d$ be natural numbers and consider the poset
\[
	\multichoose{n}{d}=\set{I\in\NN^d}[1\leq{i_1}\leq{i_2}\leq\cdots\leq{i_d}\leq{n}]
\]
of $d$-element multi-subsets of $\mathbf{n}=\set{1,\dots,n}$, where $I \le J$ if
for each $1 \leq a \leq d$ the inequality $i_a \leq j_a$ is satisfied. Further,
introduce the subset $\degchoose{n}{d}\subseteq\multichoose{n}{d}$ consisting of
those $I\in\multichoose{n}{d}$ such that there exists an index $1 \leq a < d$
with $i_a = i_{a+1}$. Thus, the complement
$\mychoose{n}{d}=\multichoose{n}{d}\setminus\multichoose{n}{d}^\flat$ can be
identified with the set of $d$-element subsets of $\mathbf{n}$ equipped with the
natural product order.

Let $\kk$ be a field and define the finite-dimensional $\kk$-algebra
\[
	\Aus{n}{d} := (\bigoplus_{I\leq J} \kk f_{JI})\big/ \langle f_{KK}\,|\,
  K\in\degchoose{n}{d} \rangle
\]
equipped with the multiplication law
\[
	f_{KJ'}\cdot f_{JI}=%
  \begin{cases}%
    f_{KI} & \text{if }J=J',\\%
    0 & \text{otherwise.}%
  \end{cases}
\]
Equivalently, $\Aus{n}{d}$ is the quotient of the incidence $\kk$-algebra of the
poset $\multichoose{n}{d}$ by its two-sided ideal generated by the idempotents
$f_{KK}$, $K \in\degchoose{n}{d}$ (note that the $\kk$-algebra $\Aus{n}{d}$
vanishes if $n<d$ and is isomorphic to the base field $\kk$ if $n=d$). For
example, an $\Aus{n}{1}$-module corresponds to a $\kk$-vector-space-valued
representation
\[
	\begin{tikzcd}[column sep=small]
		V_1 \ar{r} & V_2 \ar{r} & \cdots \ar{r} & V_n
	\end{tikzcd}
\]
of the linearly oriented $A_n$-quiver, while an $\Aus{n}{2}$-module amounts to a
commutative diagram
\[
	\begin{tikzcd}[column sep=small,row sep=small]
		V_{11}\ar{r}&V_{12}\ar{r}\ar{d}&V_{13}\ar{r}\ar{d}&\cdots\ar{r}&V_{1n}\ar{d}\\%
		&V_{22}\ar{r}&V_{23}\ar{r}\ar{d}&\cdots\ar{r}&V_{2n}\ar{d}\\%
		&&V_{33}\ar{r}&\cdots\ar{r}&V_{3n}\ar{d}\\%
    &&&\ddots&\vdots\ar{d}\\%
    &&&&V_{nn}%
	\end{tikzcd}
\]
of $\kk$-vector spaces with vanishing diagonal terms $V_{aa}$, $1\leq a\leq n$.\\

Remarkably, the algebras $\Aus{n}{d}$ arise naturally in three \emph{a priori}
unrelated contexts within representation theory, algebraic $K$-theory, and
symplectic topology:

\begin{enumerate}[label=(\Alph*)]
\item \label{item:A} {\bf Higher Auslander--Reiten theory.} The classical
  Auslander correspondence~\cite{Aus71}, one of the cornerstones of the
  representation theory of finite-dimensional algebras, establishes a bijection
  \begin{center}
    $\left\{%
      \parbox{14em}{\centering%
        finite-dimensional $\kk$-algebras\\ of finite representation type}%
    \right\}%
    \quad%
    \longleftrightarrow%
    \quad%
    \left\{%
      \parbox{16em}{\centering%
        finite-dimensional $\kk$-algebras $\Gamma$ with\\
        $\gldim\Gamma\leq2\leq\domdim\Gamma$}%
    \right\}$
  \end{center}
  where both classes of algebras are considered up to Morita equivalence; we
  remind the reader that $\domdim\Gamma$, the dominant dimension in the sense of
  \cite{Tac64}, is the largest number $d$ such that, in a minimal injective
  coresolution
  \[
    \begin{tikzcd}[column sep=small]
      0\rar&\Gamma_\Gamma\rar&I^0\rar&
      I^1\rar&\cdots\rar&I^{d-1}\rar&I^d\rar&\cdots
    \end{tikzcd}
  \]
  of the regular representation $\Gamma_\Gamma$, the injective $\Gamma$-modules
  $I^0,I^1,\dots,I^{d-1}$ are also projective. The correspondence is realised by
  associating to a $\kk$-algebra $A$ of finite representation type its {\em
    Auslander algebra}
  \[
    \Gamma_A = \End[\oplus_{[M]} M][A],
  \]
  where the sum ranges over the isomorphism classes of indecomposable
  $A$-modules $M$. This correspondence relates the representation-theory of $A$
  to the homological properties of $\Gamma$.

  For example, the $\kk$-algebra $\Aus{n+1}{2}$ is the
  Auslander algebra corresponding to the $\kk$-algebra $A_{n,1}$. As shown
  in~\cite{Iya11}, for $n\geq d>2$ the algebra $\Aus{n}{d}$ rather satisfies
  the inequalities
  \[
    \gldim\Gamma\leq d \leq \domdim\Gamma
  \]
  Thus, $\Aus{n}{d}$ belongs to the class of \emph{$d$-dimensional Auslander
    algebras} introduced by Iyama in \cite{Iya07a} as central objects of study
  in a higher-dimensional version of Auslander--Reiten theory. From now on we
  will refer to the $\kk$-algebras $\Aus{n}{d}$ as the {\em higher Auslander
    algebras of Dynkin type $\AA$}. Due to their rich combinatorial structure,
  this family of algebras has garnered quite some attention in representation
  theory, see for example \cite{HI11a,IO11,OT12,HIO14,GI19,DJW19b} (where they
  are mostly referred to as the `$d$-representation finite algebras of type
  $\AA$') as well as the closely related \cite{IO13,JK19a}.

\item \label{item:B} {\bf Waldhausen $K$-theory.} By a construction of
  Waldhausen~\cite{BGT13,Wal85}, the sequence
  \[
    \morphism*{n}{\perf{\Aus{n}{1}}}
  \]
  of perfect derived categories of the $\kk$-algebras $\Aus{n}{1}$, $n\geq0$
  organises into a simplicial differential graded $\kk$-category 
  which provides a model for the Waldhausen $K$-theory~\cite{Wal85} of the field
  $\kk$. More precisely, the Waldhausen $K$-theory space
  \[
    K(\kk)=\Omega|N_{\mathrm{dg}}(\perf{\Aus{\bullet}{1}})^{\simeq}|
  \]
  of $\kk$ is defined to be the $1$-fold loop space of the geometric realisation
  of the simplicial $\infty$-groupoid
  $N_{\mathrm{dg}}(\perf{\Aus{\bullet}{1}})^{\simeq}$ (obtained by passing to
  the largest Kan complex level-wise), leading to the formula
  $K_i(\kk)\cong\pi_i(K(\kk))$ for the higher algebraic $K$-groups of $\kk$
  previously defined by Quillen~\cite{Qui73}.

  For $d > 1$, as discussed for abelian categories in the work of
  Poguntke~\cite{Pog17} and in \cite{Dyc17,DJW19b} in the stable context, the
  sequence
  \[
    \morphism*{n}{\perf{\Aus{n}{d}}}
  \]
  of perfect derived (DG-)categories also forms a simplicial category. The relation to Waldhausen $K$-theory is then
  given by the formula
  \[
    K(\kk)\simeq\Omega^d|N_{\mathrm{dg}}(\perf{\Aus{\bullet}{d}})^{\simeq}|,
  \]
  so that, for a fixed natural number $d$, the simplicial relations among the
  algebras $\Aus{n}{d}$ encode the $d$-fold delooping of the $K$-theory space
  $K(\kk)$.

\item\label{item:C}{\bf Wrapped Floer theory.} Let $\theDisk \subset \mathbb{C}$
  be the closed unit disk. For
  definiteness, fix the subset $\Stops[][n] \subset \partial \theDisk$ of
  $(n+1)$-st roots of unity. It is well known that there is a
  quasi-equivalence of $A_\infty$-categories
  \[
    \functor{\perf{\Aus{n}{1}}}{\W(\theDisk, \Stops[][n])}[\simeq]
  \]
  where $\W(\theDisk,\Stops[][n])$ denotes the \emph{partially wrapped} Fukaya
  category~\cite{Aur10a,HKK17} of $\theDisk$ with stops in $\Stops[][n]$. The
  main result of the present article, stated below, establishes a
  higher-dimensional version of the above equivalence, providing a symplectic
  interpretation of all higher Auslander of Dynkin type $\AA$. 
\end{enumerate}

\subsection*{Our results}

Let $\kk$ be a commutative ring. The following is the main theorem in this
article.
  
\begin{theorem*}[\Cref{thm:WF-Auslander}]
  \label{introthm:1}
  Let $\theDisk \subset \mathbb{C}$ be the closed unit disk and let $\Stops[][n] \subset \partial \theDisk$ be
  the subset of $(n+1)$-st roots of unity. Then, there is a quasi-equivalence of
  triangulated $A_\infty$-categories
  \[
    \functor{\perf{\Aus{n}{d}}}{\W(\Sym{d}{\theDisk},\Stops[d][n])}[\simeq]
  \]
  between the derived $A_\infty$-category of the $\kk$-algebra $\Aus{n}{d}$ and
  the $\ZZ$-graded partially wrapped Fukaya category of the $d$-th symmetric
  power $\Sym{d}{\theDisk}$ of $\theDisk$ with stops in
  \[
    \Stops[d][n]=\bigcup_{p\in\Stops[][n]}\set{p}\times\Sym{d-1}{\theDisk}.\qedhere
  \]
\end{theorem*}

It is remarkable that, when working over a field, the global dimension of the
$\kk$-algebra $\Aus{n}{d}$, which is $d$ provided that $n>d$, is reflected in
the dimension of the symplectic manifold $\Sym{d}{\theDisk}$, which is a real
$2d$-dimensional ball.

The proof of the foregoing theorem relies on an explicit computation of the
derived endomorphism algebra $\dgAus{n}{d}$ of a specific set of generators of
the $A_\infty$-category $\W(\Sym{d}{\theDisk},\Stops[d][n])$. As it turns out,
$\dgAus{n}{d}$ is in fact a formal differential graded $\kk$-algebra whose
(degree $0$) cohomology is isomorphic to $\Aus{n}{d}$ as an ungraded
$\kk$-algebra (in the parlance of representation theory, we construct an
explicit tilting object in $\W(\Sym{d}{\theDisk},\Stops[d][n])$ whose
endomorphism algebra is isomorphic to $\Aus{n}{d}$). The differential graded
$\kk$-algebra $\dgAus{n}{d}$ was originally introduced in the context of
bordered Heegaard Floer homology \cite{LOT18} under the name `strands algebra
with $d$ strands and $n$ places'; its cohomology is computed in Section~4.1
in~\cite{LOT15}. However, in the aforementioned sources the authors work in
characteristic $2$. For this reason we follow an approach (hinted at in
\cite{LOT18,LOT15}) which relates the strands algebra to the Bruhat order on the
symmetric group on $d$ letters.

The connection between \ref{item:A}, \ref{item:B} and \ref{item:C} provided by
the family of higher Auslander algebras of type $\AA$ offers the opportunity to
use insights or techniques from one of these subjects to the benefit of another.
In this work we give a first illustration of the possibilities that arise from
the interaction between higher Auslander--Reiten theory \ref{item:A} and wrapped
Floer theory \ref{item:C}. Namely, we provide a symplectic proof of the derived
Morita equivalence between the $\kk$-algebras $\Aus{n}{d}$ and $\Aus{n}{n-d}$
obtained by Beckert in~\cite{Bec18} by means of a rather involved calculus of
derivators.
  
\begin{theorem*}[\Cref{thm:WF-Auslander:Koszul,thm:dgAus-sharp:AusShiftedKoszul}]
\label{thm:portmanteau}
Let $n\geq d\geq 1$.There is a commutative diagram of quasi-equivalences of
  triangulated $A_\infty$-categories
  \[
    \begin{tikzcd}[column sep=small]
      \perf{\Aus*{n}{d}}%
      \drar\ar[leftrightarrow]{rr}{\text{\tiny{Koszul Duality}}}%
      &&\perf{\Aus{n}{d}}\drar\dlar\\
      &\W(\Sym{d}{\theDisk},\Stops[d][n])&&\W(\Sym{n-d}{\theDisk},\Stops[n-d][n])\\
      &&\perf{\Aus{n}{n-d}}%
      \urar\ular\ar[leftrightarrow]{rr}[swap]{\text{\tiny{Koszul Duality}}}%
      &&\perf{\Aus*{n}{n-d}}\ular
    \end{tikzcd}
  \]
  where, by convention, $\W(\Sym{0}{\theDisk},\Stops[0][n])=\perf{\kk}$, and
  $\Aus*{n}{d}$ and $\Aus*{n}{n-d}$ denote the derived Koszul duals of the (augmented)
  $\kk$-algebras $\Aus{n}{d}$ and $\Aus{n}{n-d}$, respectively.
\end{theorem*}

The proofs of both of the above theorems exploit results of Auroux~\cite{Aur10a}
which allow us calculate the cochain complexes of morphisms between objects in a
distinguished subcategory of $\W(\Sym{d}{\theDisk},\Stops[d][n])$ and to
determine explicit classical generators of this partially wrapped Fukaya
category. While the triangles of quasi-equivalences in
\Cref{thm:portmanteau} are direct consequences of derived Koszul duality for
homologically smooth and proper algebras, the diamond of quasi-equivalences in
the middle is due to an additional symmetry: First, note that the binomial
symmetry
\[ 	
	\textstyle{n \choose d} = {n \choose n-d}	
\]
can be concretely realised by the bijection
\[
  \begin{tikzcd}[column sep=large]
    \mychoose{n}{d}\rar[leftrightarrow]{\cong}&\mychoose{n}{n-d},
  \end{tikzcd}\qquad\morphism*{I}{\n\setminus I}
\]
obtained by associating to a subset $I \subseteq \mathbf{n}$ of cardinality $d$
its complement. This bijection further mediates a correspondence of objects
\[
  \begin{tikzcd}[row sep=tiny]
    \W(\Sym{d}{\theDisk},\Stops[d][n]) & \mychoose{n}{d} \ar{l}
    \ar[<->]{r}{\cong} &
		\mychoose{n}{n-d}	\ar{r} & \W(\Sym{n-d}{\theDisk},\Stops[n-d][n])\\
    \prod_{i \in I} L_{0,i} & \ar[|->]{l} I & J \ar[|->]{r} & \prod_{j \in J}
    L_{j-1,j}
  \end{tikzcd}
\]
where the Lagrangians $L_{ij}\subset\theDisk$ are the ones depicted in
\Cref{fig:Sigma5_Iyama,fig:Sigma5_Koszul}, respectively. The proof of
\Cref{thm:portmanteau} simply amounts to verifying that the above collections of
Lagrangians generate the corresponding partially wrapped Fukaya categories and,
for a suitable choice of grading structures on them, the derived endomorphism
algebras of their direct sums are quasi-isomorphic.

\begin{remark*}
  Let $\overline{\mathbb{H}}\subset\mathbb{C}$ be the closed upper half-plane
  and $\Stops_\ZZ\subset\partial\overline{\mathbb{H}}$ the set of integer
  numbers. The results in this article can be easily extended to prove that
  there is a quasi-equivalence of triangulated $A_\infty$-categories
  \[
    \functor{\Df{\Aus{\infty}{d}}}{\W(\Sym{d}{\overline{\mathbb{H}}},\Stops[d]_\ZZ),}[\simeq]
  \]
  where $\Aus{\infty}{d}$ is the (infinite-dimensional, locally unital)
  $\kk$-algebra defined in complete analogy to the higher Auslander algebras of
  type $\AA$, considering instead the poset
  \[
    \left\{\begin{smallmatrix} \ZZ\\d
      \end{smallmatrix}\right\}=\set{I\in\mathbb{Z}^d}[i_1\leq\cdots\leq i_d]
  \]
  of $d$-element multi-subsets of $\ZZ$. When working over a field, the
  $\kk$-algebra $\Aus{\infty}{d}$ was introduced originally in \cite{JK19a} as a
  higher-dimensional analogue of the mesh category of type $\ZZ A_\infty$ which
  appears prominently in classical Auslander--Reiten theory. Similarly, the
  partially wrapped Fukaya category
  $\W(\Sym{d}{S^1\times\mathbb{R}_{\geq0}},\Stops[d][n])$ associated to the
  $d$-fold symmetric product of a marked half-infinite cylinder is related to
  the higher-dimensional analogues of the tubes introduced in \emph{loc.~cit.}
\end{remark*}

\subsection*{Future work}

It is known \cite{Boc16, HKK17,LP17a} that the partially wrapped and the
infinitesimally wrapped Fukaya categories of a punctured surface are equivalent
to (suitable versions of) the derived categories of graded gentle algebras
associated to a polygonal decomposition of the surface dual to a ribbon graph
(Lagrangian skeleton) onto which the surface retracts. Conversely, the derived
category of every graded gentle algebra arises this way. Therefore, the Fukaya
categories of punctured Riemann surfaces provide a geometric way of studying the
combinatorial structure of derived categories of graded gentle
algebras. For example, in \cite{LP18a} new derived equivalences between gentle
algebras were discovered using this perspective, which subsequently led to a
classification of all (ungraded) gentle algebras up to derived equivalence
\cite{APS19,Opp19}.

An important feature of partially wrapped Fukaya categories of Weinstein
manifolds is that, as envisioned by Kontsevich \cite{Kon09} and further explored
in \cite{GPS18}, they are expected to satisfy cosheaf-like descent conditions.
Namely, there should be a constructible cosheaf of $A_\infty$-categories defined
on every Lagrangian skeleton of the Weinstein manifold whose $A_\infty$-category
of global sections is equivalent to the partially wrapped Fukaya category of the
Weinstein manifold. In the case of surfaces, this amounts to the fact that
derived categories of gentle algebras can be exhibited as a colimit of a
suitable diagram formed by perfect derived categories of $A_n$-quivers, which
are to be interpreted as the stalks of the corresponding (constructible) cosheaf
defined on a ribbon graph of the surface \cite{HKK17,DK18,Dyc17a}.

In subsequent work we exploit the interplay between \ref{item:B} and
\ref{item:C} to further advance this programme: The simplicial combinatorics
intrinsically present in \ref{item:B} provide a systematic approach to
local-to-global phenomena for the partially wrapped Fukaya categories of
symmetric powers of Riemann surfaces. As a result, we expect to compute these
partially wrapped Fukaya categories by gluing together perfect derived
categories of higher-dimensional Auslander algebras of type $\mathbb{A}$. The
present article thus focuses on the local aspects that enter into this picture.
For completeness, we mention that the case of \emph{punctured} disks is
investigated in detail in \cite{LP18} from the point of view of homological
mirror symmetry.

\subsection*{Conventions}

We fix an arbitrary commutative ring $\kk$. In what concerns
$A_\infty$-categories, we follow the notational conventions in Seidel's
book~\cite{Sei08}. Given objects $x$ and $y$ in an $A_\infty$-category $\A$, we
denote the cochain complex of morphisms from $x$ to $y$ by $\hom[x][y][\A]$ and
write $\hom[x][y]=\hom[x][y][\A]$ if the ambient $A_\infty$-category is clear
from the context. We also write
\[
  \Hom[x][y][\A][\ast]=H^*(\hom[x][y][\A])
\]
and $\Hom[x][y][\A]=\Hom[x][y][\A][0]$ to denote the graded $\kk$-module and the
plain $\kk$-module of morphisms from $x$ to $y$ in the graded $\kk$-category
$H^*(\A)$ and the ungraded $\kk$-category $H^0(\A)$, respectively. Finally, for
differential graded $\kk$-categories we follow the conventions in \cite{Kel06};
in particular, the graded Leibniz rule takes the form
\[
  \partial(\psi\circ\varphi)=\partial(\psi)\circ\varphi+(-1)^{|\psi|}\psi\circ\partial(\varphi)
\]
whenever $\psi$ is a homogeneous morphism (we compose morphisms from right to
left).

\subsection*{Acknowledgements}

T.~D. acknowledges the support by the VolkswagenStiftung for his Lichtenberg Professorship at the
University of Hamburg and is further supported by the Deutsche Forschungsgemeinschaft under
Germany‘s Excellence Strategy – EXC 2121 „Quantum Universe“ – 390833306. G.~J. was funded by the
Deutsche Forschungsgemeinschaft (German Research Foundation) under Germany’s Excellence Strategy -
GZ 2047/1, Projekt-ID 390685813. Y.~L. was partially funded by the Royal Society
URF$\setminus$R$\setminus$180024; he would like to thank D.~Nadler for motivating conversations
around this topic. 

\section{Preliminaries}

In this preliminary section we recall general aspects of the theory of Fukaya
categories which are relevant to our investigation of the partially wrapped
Fukaya categories associated to the symmetric products of the $2$-dimensional
unit disk.

\subsection{Fukaya categories of symmetric products of Riemann surfaces}

We begin with a brief overview of the different variants of the Fukaya category
associated to the symmetric products of an arbitrary Riemann surface.

\subsubsection{Fukaya categories and wrapped Fukaya categories}

Let $\Sigma$ be a compact Riemann surface and $d \geq 1$ an integer. The complex
structure $j$ on $\Sigma$ induces a complex structure $\Sym{d}{j}$ on the
$d$-fold symmetric product
\[
  \Sym{d}{\theSurface}\coloneqq\underbrace{\theSurface\times\cdots\times\theSurface}_{d\text{
      times}}/ \SymGrp{d}
\]
of $\theSurface$. Take $z \in \Sigma$ and let $\theDivisor$ be the image of the
embedding
\[
  \morphism{\Sym{d-1}{\theSurface}}{\Sym{d}{\theSurface},}\qquad
  \morphism*{\mathbf{x}}{z+\mathbf{x}.}
\]
Recall that there is an isomorphism $H^2(\Sym{d}{\theSurface})\cong\bigwedge^2
H^1(\theSurface) \oplus \mathbb{Z}$, where the direct summand $\mathbb{Z}$ is
represented by the Poincar{\'e} dual of $\theDivisor$. We are chiefly interested
in open symplectic manifolds of the form
\[
  M\coloneqq\Sym{d}{\theSurface}\setminus
  \theDivisor=\Sym{d}{\theSurface\setminus \set{z}}.
\]
The divisor $\theDivisor$ is ample, thus the complex manifold
$\Sym{d}{\theSurface}$ admits a symplectic form ${\omega=\frac{1}{2\pi}i F_A}$,
where $A$ is a connection on the line bundle $\mathcal{O}(k\theDivisor)$ for
some sufficiently large $k$. The symplectic form $\omega$ tames the complex
structure on $\Sym{d}{\theSurface}$ and lies in the cohomology class
Poincar{\'e} dual to $k\theDivisor$. Moreover, the restriction of $\omega$ to
$M$ gives an exact symplectic form which is convex at infinity.

Let $p\colon\theSurface^{\times d}\to\Sym{d}{\theSurface}$ be the branched
covering map, which is ramified along the big diagonal in
$\Sym{d}{\theSurface}$. Let $\omega_\theSurface$ be a
choice of a symplectic structure on $\theSurface$. In Corollary~7.2
in~\cite{Per08} Perutz constructs an exact K{\"a}hler form on $M$ of the above
type but with the additional property that
$\omega=p_*(\omega_\theSurface^{\times d})$ outside a small neighbourhood of the
big diagonal. This additional property allows us to consider exact Lagrangians
in $\Sym{d}{\theSurface}$ of the form $L_1 \times L_2\times\cdots\times L_d$,
where $\set{L_i}$ is a collection of pairwise disjoint exact Lagrangians on
$\Sigma$ (with respect to $\theta_\theSurface$).

\begin{remark}
  The fundamental theorem of algebra, responsible for the
  identification ${\Sym{d}{\mathbb{C}}\cong\mathbb{C}^d}$, implies that
  $\Sym{d}{\theDisk}$ is a ball of (real) dimension $2d$. Moreover, the
  symplectic structure on $\Sym{d}{\theDisk}$ is equivalent to the standard
  symplectic structure on the ball since both tame the standard complex
  structure. We prefer to use the symplectic structure constructed via the
  symmetric product construction in order to provide a diagrammatic description
  of the Lagrangians.
\end{remark}

For an exact symplectic manifold with convexity at infinity, such as $M$ above,
a rigorous construction of a Fukaya category of compact exact Lagrangians
$\F(M)$ appears in Seidel's book~\cite{Sei08}. By construction, $\F(M)$ is an
idempotent complete triangulated $A_\infty$-category which is linear over $\kk$.
There is also a rigorous construction of a larger $A_\infty$-category---due to
Abouzaid and Seidel~\cite{AS10}---that allows for non-compact exact Lagrangians
with controlled behaviour at infinity, namely the \emph{wrapped} Fukaya category
$\mathcal{W}(M)$ in which $\mathcal{F}(M)$ embeds as a full subcategory.

\subsubsection{Partially wrapped Fukaya categories}

A fruitful extension of the above considerations consists in equipping the
symplectic manifold $M$ with \emph{stops} and considering the resulting
\emph{partially wrapped} Fukaya categories. We review the technical setup of
this construction. Let $\theSurface$ be a surface with non-empty boundary
equipped with an exact area form $\omega_\theSurface$ and
$\Stops\subset\partial\theSurface$ a non-empty finite subset. Consider the
complex manifold $M=\Sym{d}{\theSurface}$ equipped with a symplectic form as
outlined above. Note that $M$ is a symplectic manifold with corners, considered
as a Liouville sector as in~\cite{GPS19}. We let
\[
  \Stops[d]= \bigcup_{p\in\Stops}\set{p}\times\Sym{d-1}{\theSurface}
\]
be a union of symplectic hypersurfaces in $M$ which we call \emph{stops}. To
these data Auroux~\cite{Aur10} associates a partially wrapped Fukaya category
$\W(M,\Stops[d])$; the general theory of partially wrapped Fukaya categories
has been further developed by Ganatra, Pardon and Shende~\cite{GPS18} and
Sylvan~\cite{Syl19}. Auroux's construction of partially wrapped Fukaya
categories was inspired by the bordered Heegaard Floer theory of Lipshitz,
Ozsv{\'a}th and Thurston \cite{LOT18}, as well as by the more familiar special
case of the Fukaya--Seidel category~\cite{Sei08} associated to a Lefschetz
fibration on an exact symplectic manifold.

\subsubsection{Stop-removal sequences}

The three flavours of Fukaya categories discussed above are related by a
commutative diagram of functors
\[
  \begin{tikzcd}[column sep=small]
    &\W(M,\Stops[d])\drar\\
    \mathcal{F}(M)\urar[hookrightarrow]\ar[hookrightarrow]{rr}&&\W(M)
  \end{tikzcd}
\]
where the functor $\F(M)\hookrightarrow\W(M,\Stops[d])$ is fully faithful, and
so is its composite with the functor $\W(M,\Stops[d])\to\W(M)$. Moreover, the
functor $\W(M,\Stops[d])\to\W(M)$ is the localisation at the union of the images
of the so-called \emph{Orlov functors}
\[
  \functor[\iota_p]{\W(\Sym{d-1}{\theSurface},\Stops_p^{(d-1)})}{\W(\Sym{d}{\theSurface},\Stops[d])}
\]
for $p\in\Stops$, where $\Stops_p=\Stops\setminus\set{p}$. In the above setting
of symmetric products of Riemann surfaces, Auroux~\cite{Aur10} proved that the
image of the Orlov functors is generated by the objects supported near the
stops (see also~\cite{Syl19} and~\cite{GPS18} for results in a much more general
setting and~\cite{Syl19a} for further applications of these functors). For a
fixed stop $p\in\Stops$, the Orlov functor $\iota_p$ is part of a Drinfeld--Verdier
localisation sequence
\begin{equation}
  \label{locseq}
  \begin{tikzcd}
    \W(\Sym{d-1}{\theSurface},\Stops_p^{(d-1)})\rar[hookrightarrow]&\W(\Sym{d}{\theSurface},\Stops[d])\rar[two
    heads]&\mathcal{W}(\Sym{d}{\theSurface}, \Stops[d][p])
  \end{tikzcd}
\end{equation}
which we refer to as the \emph{stop-removal sequence} at $p$.

\subsubsection{Grading structures}
\label{subsec:global_gradings}

In general, the (cochain complexes of morphisms in the) partially wrapped Fukaya
category $\W(\Sym{d}{\theSurface},\Stops[d])$ can only be $\ZZ/2$-graded. If the
first Chern class of $\Sym{d}{\theSurface}$ satisfies $2c_1=0$ (that is, if
$\theSurface$ has genus $0$ or if $d=1$), results in~\cite{Sei00} imply that
$\W(\Sym{d}{\theSurface},\Stops[d])$ admits a $\ZZ$-grading; the possible
$\ZZ$-gradings on $\W(\Sym{d}{\theSurface},\Stops[d])$ form a torsor over the
first cohomology group $H^1(\Sym{d}{\theSurface})\cong H^1(\theSurface)$. Since
in this article we only study the case when $\theSurface$ is a disk (which has
vanishing first cohomology), the partially wrapped Fukaya categories we consider
admit a unique $\ZZ$-grading.

\subsubsection{Generators of the partially wrapped Fukaya category}

The following result~\cite{Aur10a} describes sets of generators of the partially
wrapped Fukaya category $\W(\Sym{d}{\theSurface},\Stops[d])$ in terms of arcs in
$\theSurface\setminus\Stops[d]$.

\begin{theorem}[Auroux]
  \label{thm:Auroux}
  Let $\theSurface$ be a compact Riemann surface with non-empty boundary and
  $\Stops$ a finite set of points on its boundary. Let $L_1,\dots,L_n$ be a
  collection of disjoint properly embedded arcs in $\theSurface\setminus\Stops$
  with endpoints in the boundary of $\theSurface$. Assume that
  $\theSurface\setminus\set{L_1,\dots,L_n}$ is a disjoint union of disks, each
  of which contains at most one point of $\Stops$. Then, the partially wrapped
  Fukaya category $\W(\Sym{d}{\theSurface},\Stops[d])$ is generated, as an
  idempotent complete triangulated $A_\infty$-category, by the $n\choose d$
  Lagrangian submanifolds $L_I=\prod_{i\in I}L_i$, where $I$ ranges over the
  $d$-element subsets of $\set{1,\dots,n}$.
\end{theorem}

Let us review Auroux's proof of \Cref{thm:Auroux}. Fix $2g + 1$ points $p_0,
\ldots , p_{2g} \in \mathbb{C}$ and consider the compact surface
$\overline{\theSurface}$ given by a branched double cover of $\mathbb{C}P^1$
branched over the $p_i$'s and $\infty$. Let $\theSurface$ be the surface
obtained from $\overline{\theSurface}$ by puncturing it at the preimage of
$\infty$. The surface $\theSurface$ comes with a $2$-fold branched covering
$\pi\colon\theSurface \to \mathbb{C}$. Consider the map
\[
  \morphism[f]{\Sym{d}{\theSurface}}{\mathbb{C},}\qquad\morphism*{[z_1,\dots,z_d]}{\pi(z_1)+\cdots+\pi(z_d).}
\]
The map $f$ is a Lefschetz fibration whose critical points are tuples
$[q_{s_1},\dots,q_{s_d}]$ where $s = \set{s_1,\dots,s_d}$ is a $d$-element
subset of $\set{p_0,\dots , p_{2g}}$ and $\pi(q_{s_i}) = p_{s_i}$ . That is, the
critical points consist of $d$-tuples of distinct critical points of $\pi$. By
adding more stops if necessary, the generators specified in \Cref{thm:Auroux}
can be identified with thimbles of the Lefschetz fibration $f$ and, moreover,
the resulting partially wrapped Fukaya category is quasi-equivalent to the
Fukaya--Seidel category of $f$. A celebrated result of Seidel~\cite{Sei08} shows
that these thimbles generate the Fukaya--Seidel category of $f$. Removing the
additional stops, if there are any, corresponds to a localisation of the
Fukaya--Seidel category of $f$. Auroux leverages these considerations to prove
that the product Lagrangians listed in \Cref{thm:Auroux} are indeed sufficient
to generate the partially wrapped Fukaya category
$\W(\Sym{d}{\theSurface},\Stops[d])$ as an idempotent complete triangulated
$A_\infty$-category.

In the particular case of the pair $(\Sym{d}{\theDisk},\Stops[][n])$, Auroux's
Lefschetz fibration can be described concretely as follows: First, consider the
$(n+1)$-fold branched covering map
\[
  \morphism{\mathbb{C}}{\mathbb{C},}\qquad\morphism*{z}{z^{n+1}.}
\]
This is chosen so that the preimage of $1\in\mathbb{C}$ is the set $\Stops[][n]
= \set{p_0,p_1,\dots, p_n}$ of stops, which is contained in the boundary of the
unit disk. Therefore, the Fukaya--Seidel category of a morsification of this map
is quasi-equivalent to the partially wrapped Fukaya category
$\WF{n}{1}=\W(\theDisk,\Stops[][n])$. A concrete Morsification is given by
\[
  \morphism[f_{n,1}]{\mathbb{C}}{\mathbb{C},}\qquad\morphism*{z}{z^{n+1} -
    \varepsilon (n+1) z}
\]
for some $\varepsilon>0$, so that the $n$ distinct critical points of $f_{n,1}$
lie on the circle of radius $\sqrt[n]{\varepsilon}$. Consider now the Lefschetz
fibration
\[
  \morphism[f_{n,d}]{\Sym{d}{\mathbb{C}}}{\mathbb{C},}\qquad\morphism*{[z_1,\dots,z_d]}{\sum_{i=1}^df_{n,1}(z_i)=\sum_{i=1}^d
    z_i^{n+1} - \varepsilon (n+1) z_i.}
\]
Essentially by construction, the Fukaya--Seidel category of $f_{n,d}$ is
quasi-equivalent to the partially wrapped Fukaya category
$\WF{n}{d}=\W(\Sym{d}{\theDisk},\Stops[d][n])$. Under the identification
\[
  \morphism{\Sym{d}{\mathbb{C}}}{\mathbb{C}^d,}[\cong]\qquad\morphism*{[z_1,\dots,z_d]}{(e_1(z_1,\dots,z_d),\dots,e_d(z_1,\dots,z_d))}
\]
where $e_i$ is the $i$-th elementary symmetric polynomial (the sum of all
distinct products of $i$ variables), we can express the function $f_{n,d}$ (and
its Morsification) in terms of the coordinates of $\mathbb{C}^d$ using Newton's
identities. For example, for $n=d=2$ we obtain the map
\[
  \morphism{\mathbb{C}^2}{\mathbb{C},}\qquad\morphism*{(u,v)}{u^3-3uv-3\varepsilon{u}.}
\]

\subsubsection{Derived endomorphism algebras of generators}

Let $L_1,\dots,L_n$ be arcs in $\theSurface\setminus\Stops[][n]$ which satisfy
the assumptions in \Cref{thm:Auroux}, with the additional condition that each
disk in the decomposition $\theSurface\setminus(\cup L_i)$ contains
\emph{exactly} one stop. Suppose for a moment that the ground ring $\kk$ has
characteristic $2$, so that the results in \cite{Aur10} apply verbatim. In this
context, Auroux provides a differential $\ZZ/2$-graded model for the derived
endomorphism algebra $\mathcal{A}=\bigoplus_{I,J}\hom[L_I][L_J]$ of the
associated generator of the partially wrapped Fukaya category
$\W(\Sym{d}{\theSurface},\Stops[d])$. By picking a particularly nice
perturbation scheme adapted to this set-up, he shows that there are no
holomorphic $n$-gons for $n \geq 3$ with boundary on these generators and
therefore the higher products in the derived endomorphism algebra $\mathcal{A}$
vanish. Moreover, he determines all the holomorphic bigons and triangles that
contribute to the product and the differential. In the end, Auroux establishes a
quasi-isomorphism
\[
  \mathcal{A}=\bigoplus_{I,J}\hom[\L{I}][\L{J}]\simeq\bigoplus_{I,J}\bigoplus_{\pi\in\SymGrp{d}}\hom[\L{I}][\L{J}]^\pi,
\]
where
\[
  \hom[\L{I}][\L{J}]^\pi=\hom[\L{i_1}][\L{j_1}]\otimes\hom[\L{i_2}][\L{j_{\pi(2)}}]\otimes\cdots\otimes\hom[\L{i_d}][\L{j_{\pi(d)}}]
\]
and $\hom[L_i][L_j]$ is the free $\kk$-module generated by all Reeb chords
$L_i\to L_j$ induced by the Reeb flow along the various boundary components of
$\theSurface\setminus\Stops[][n]$. The differential $\ZZ/2$-graded $\kk$-algebra
on the right-hand side can be identified with the strands algebra arising in the
context of the bordered Heegaard--Floer homology of Lipshitz, Ozsv\'ath and
Thurston \cite{LOT18}. We emphasise that in general the above direct sum
decomposition of the strands algebra is not preserved neither by the
differential nor the multiplication operation, which is a source of
difficulties.

In arbitrary characteristic, Auroux's description is still relevant. Indeed, the
above counting of holomorphic polygons does not depend on the ground ring nor on
a possible choice of grading structure on $\Sym{d}{\theSurface}$. However, there are \emph{signs} involved in the differential and the
composition of morphisms, which arise from the orientations of the various
moduli spaces of holomorphic disks. In the case $\theSurface=\theDisk$, which is
the only concern of this article, we approach this problem of determining signs by providing an
explicit characteristic-free lift of the strands algebra from \cite{LOT18} by
means of a simple modification of a construction of Khovanov \cite{Kho14}, and showing that any other lift gives rise to an isomorphic algebra.

\subsection{The partially wrapped Fukaya categories $\WF{n}{d}$}

Let $\theDisk$ be the $2$-dimensional unit disk and
$\Stops[][n]=\set{p_0,p_1,\dots,p_n}$ a set of stops on its boundary, where
$n\geq0$. While from the point of view of symplectic topology the precise
position of the stops does not matter, for definiteness we often let
$\Stops[][n]$ be the set of $(n+1)$-st roots of unity. In this section we begin
our analysis of the partially wrapped Fukaya categories
\[
  \WF{n}{d}\coloneqq\W(\Sym{d}{\theDisk},\Stops[d][n]),\qquad d\geq1.
\]
Note that $\Sym{d}{\theDisk}$ has Liouville completion
$\Sym{d}{\mathbb{C}}\cong\mathbb{C}^n$, which is an exact symplectic manifold
with vanishing first Chern class. As explained in
\S\ref{subsec:global_gradings}, the partially wrapped Fukaya category
$\WF{n}{d}$ admits a $\ZZ$-grading and the vanishing of the first cohomology
group $H^1(\Sym{d}{\theDisk})\cong H^1(\theDisk)$ implies that this
$\ZZ$-grading is unique. We equip $\WF{n}{d}$ with this canonical $\ZZ$-grading
throughout the article. In $\Sym{d}{\theDisk}$, we will consider Lagrangians $L$
of the form $L = \prod_{i=1}^d L_i$ where $L_i$ are pairwise disjoint arcs in
$\theDisk \setminus \Stops[][n]$. These are all contractible, hence the choice
of a grading structure on such a Lagrangian is unique up to the natural action
of $\mathbb{Z}$ which corresponds to the shift functor in $\WF{n}{d}$. Finally,
if $n<d$ it can be shown---by applying the stop-removal sequence \eqref{locseq}
recursively and keeping in mind that the (fully) wrapped Fukaya category
$\mathcal{W}^{(d)}= \mathcal{W}(\mathbb{C}^n)$ is trivial---that the partially
wrapped Fukaya category $\WF{n}{d}$ is trivial. Thus, our interest mostly lies
in the case $n\geq d\geq1$.

\subsection{The partially wrapped Fukaya category $\WF{n}{1}$}


\subsubsection{Combinatorial coordinates}

We assume that the orientation of $\theDisk$ is such that, if embedded in the
plane, its boundary is oriented counter-clockwise. For combinatorial reasons, we
assume that the cyclic order on the labels of the stops $p_i\in\Stops[][n]$
induced by the orientation of the disk's boundary agrees with the
\emph{opposite} of the natural cyclic order on the set $\set{0,1,\dots,n}$. We
denote by $i$ the boundary component of $\mathbb{D}\setminus\Stops[][n]$ which
lies between the points $p_i$ and $p_{i+1}$. In particular, if $n=0$ then there
is a single boundary component labelled $0$.

For $0\leq i<j\leq n$ we let $L_{ij}$ be a properly embedded arc in
$\theDisk\setminus\Stops[][n]$ whose endpoints lie on the components of
$\partial\mathbb{D}\setminus\Stops[][n]$ labelled $i$ and $j$. With some abuse
of notation, we identify these Lagrangians with objects of the partially wrapped
Fukaya category $\WF{n}{1}$, keeping in mind that one has to also choose grading
structures on them. Objects with the same underlying Lagrangians with
different grading structures are related by the shift functor. The morphisms
between disjoint Lagrangians are given by (composites of) Reeb chords
corresponding to the Reeb flow along the boundary of the disk which is the
rotational flow in the counterclockwise direction. More generally, there are
additional morphisms induced by intersection points.

The Lagrangians $L_{01},\dots,L_{0n}$ play a particularly important role in this
article. We choose isotopy classes of these Lagrangians so that they are
mutually disjoint and partition $\theDisk$ into a disjoint union of
(topological) disks, each of which contains exactly one stop in $\Stops[][n]$.
By \Cref{thm:Auroux}, these Lagrangians generate the partially wrapped
Fukaya category $\WF{n}{1}$ as an idempotent complete triangulated
$A_\infty$-category (with respect to arbitrarily chosen grading structures on
them). See \Cref{fig:Sigma5_Iyama} for an example in the case $n=5$.

\begin{figure}[!h]
  \includegraphics{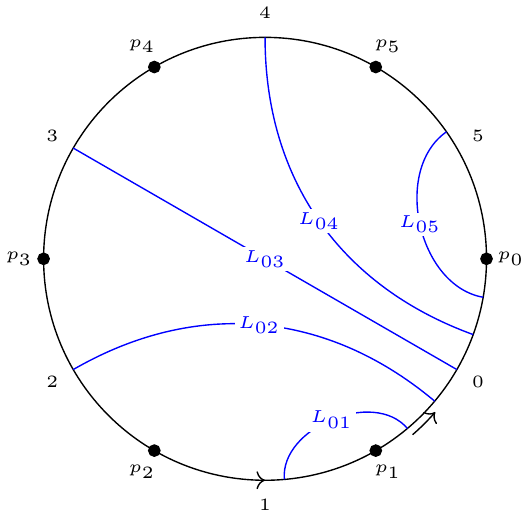}
  \caption{The generators $L_{01},\dots,L_{0n}$ of
    $\WF{n}{1}=\W(\theDisk,\Stops[][n])$ in the case $n=5$.}
  \label{fig:Sigma5_Iyama}
\end{figure}

\subsubsection{The quasi-equivalence $\WF{n}{1}\simeq\perf{\Aus{n}{1}}$}

Let $n\geq 1$. Recall that $\Aus{n}{1}$ denotes the path $\kk$-algebra of the
quiver
\[
  \begin{tikzcd}[column sep=small]
    1\rar&2\rar&\cdots\rar&n.
  \end{tikzcd}
\] 
We view $\Aus{n}{1}$ as an $A_\infty$-algebra concentrated in degree $0$ and
vanishing higher products. We claim that there is an \emph{isomorphism} of
$A_\infty$-algebras
\[
  \begin{tikzcd}
    \bigoplus_{i,j=1}^n\hom[L_{0i}][L_{0j}]\rar{\simeq}&\Aus{n}{1}
  \end{tikzcd}
\]
and, consequently, a quasi-isomorphism of triangulated $A_\infty$-categories
\[
  \functor{\perf{\Aus{n}{1}}}{\WF{n}{1}.}[\simeq]
\]
Indeed, since the Lagrangians $L_{01},\dots,L_{0n}$ do not intersect in the
interior of the disk, the only morphisms between them in $\WF{n}{1}$ are given
by Reeb chords corresponding to the Reeb flow along the boundary of the disk:
\[
  \begin{tikzcd}[column sep=small]
    L_{01}\rar&L_{02}\rar&\cdots\rar&L_{0n}.
  \end{tikzcd}
\]
Note that there are no relations between these morphisms. Moreover, since the
above quiver is a tree, we can choose grading structures on the above
Lagrangians to ensure that all of the morphisms above have degree $0$ (for
example, by choosing arbitrary grading structures and then shifting them as
necessary). Therefore, with respect to the aforementioned choice of grading
structures, the natural map
\[
  \morphism{\bigoplus_{i,j=1}^n\hom[L_{0i}][L_{0j}]}{\bigoplus_{i,j}^nH^0(\hom[L_{0i}][L_{0j}])}[\cong]
\]
is an isomorphism of $A_\infty$-algebras and, moreover, that there is an
isomorphism of (ungraded) $\kk$-algebras
\[
  \bigoplus_{i,j}^nH^0(\hom[L_{0i}][L_{0j}])\cong\Aus{n}{1};
\]
in particular, the $A_\infty$-algebra $\bigoplus_{i,j=1}^n\hom[L_{0i}][L_{0j}]$
is in fact an (ungraded) $\kk$-algebra. This proves the claim.

For later use we note that, with respect to appropriately chosen grading
structures on the Lagrangians $L_{0i}$, $1\leq i\leq n$, the above argument
establishes the existence of isomorphisms of \emph{graded} $\kk$-modules
\begin{equation}
  \label{eq:LiLj}
  \hom[L_{0i}][L_{0j}]\cong\begin{cases}
    \kk(0)&i\leq j;\\
    0&\text{otherwise}.
  \end{cases}
\end{equation}
In particular, the differential graded $\kk$-algebra
$\bigoplus_{i,j=1}^n\hom[L_{0,i}][L_{0,j}]$ is in fact an (ungraded)
$\kk$-algebra.

\subsubsection{The quasi-equivalence $\WF{n}{1}\simeq\perf{\Aus{n}{n-1}}$}

Following Section 3.3 in~\cite{HKK17} and
Section 2.1 in~\cite{LP18}, consider the Lagrangians $L_{i-1,i}$, $1\leq i\leq
n$ in $\theDisk\setminus\Stops[][n]$, embedded in such a way that they are
mutually disjoint and partition $\theDisk$ into a disjoint union of disks, each
of which containing exactly one stop in $\Stops[][n]$; for reference, the
situation is depicted in \Cref{fig:Sigma5_Koszul} in the case $n=5$.
\begin{figure}
  \includegraphics{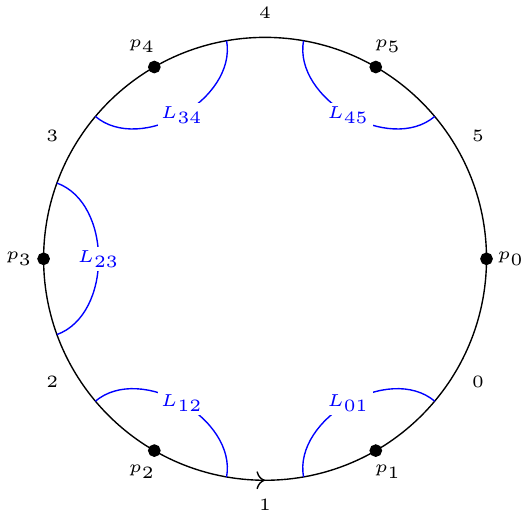}
  \caption{The generators $L_{01},L_{12},\dots,L_{n-1,n}$ of
    $\W(\theDisk,\Stops[][n])$ in the case $n=5$.}
  \label{fig:Sigma5_Koszul}
\end{figure}
By \Cref{thm:Auroux}, these Lagrangians generate the partially wrapped Fukaya
category $\WF{n}{1}=\W(\theDisk,\Stops[][n])$ as an idempotent complete
triangulated $A_\infty$-category. Moreover, since the Lagrangians $L_{i-1,i}$,
$1\leq i\leq n$, do not intersect in the interior of the disk, the only
morphisms in $\WF{n}{1}$ between them are given by Reeb chords corresponding to
the Reeb flow along the boundary of the disk:
\[
  \begin{tikzcd}[column sep=small]
    L_{n-1,n}\rar&\cdots\rar&L_{12}\rar&L_{01}.
  \end{tikzcd}
\]
Note that the composite of any two consecutive of these morphisms vanishes for
the corresponding Reeb chords cannot be composed. Also, since the above quiver
is a tree, we can choose grading structures on the Lagrangians $L_{i-1,i}$,
$1\leq i\leq n$, to ensure that all the morphisms between them have degree $0$.
The resulting $\kk$-algebra is isomorphic to the higher Auslander algebra
$\Aus{n-1}{n}$. This establishes the existence of the required
quasi-equivalences. For later use, we note that the above argument shows that there is an isomorphism of
graded $\kk$-modules
\begin{equation}
  \label{eq:hom_vee-alt}
  \hom[L_{j-1,j}][L_{i-1,i}]\cong\begin{cases}
    \kk(0)&0\leq j-i\leq1;\\
    0&\text{otherwise,}
  \end{cases}
\end{equation}
with respect to appropriate grading structures on the corresponding Lagrangians.
In particular, the differential graded $\kk$-algebra
$\bigoplus_{i,j=1}^n\hom[L_{j-1,j}][L_{i-1,i}]$ is in fact and (ungraded)
$\kk$-algebra.

\section{Fukaya categories of symmetric products of disks}

We fix natural numbers $n\geq d\geq 1$. We
recall from~\cite{OT12} an equivalent definition of Iyama's higher Auslander
algebras of type $\AA$ which is the most convenient for our present purposes. It
is straightforward to verify that this definition agrees with that given in the
introduction to this article.

\begin{definition*}
  The \emph{$d$-dimensional Auslander algebra of type $\AA_{n-d+1}$} is the
  (ungraded) $\kk$-algebra with underlying $\kk$-module
  \[
    \Aus{n}{d}\coloneqq%
    \bigoplus_{I\leq J}\kk\cdot f_{JI}/\langle f_{JI}\ |\ \exists\ a:j_a\geq i_{a+1}\rangle,
  \]
  where $I$ and $J$ range over the poset $\mychoose{n}{d}$, equipped with the
  apparent multiplication law
  \[
    f_{KJ'}\cdot f_{JI}=\begin{cases}
      f_{KI}&\text{if }J=J',\\
      0&\text{otherwise}.
    \end{cases}
  \]
  In particular, $\Aus{n}{d}$ is a monomial quotient of the incidence
  $\kk$-algebra of the poset $\mychoose{n}{d}$.
\end{definition*}

\begin{remark*}
  If $n=d$, then the corresponding higher Auslander algebra $\Aus{n}{n}$ is
  isomorphic to the ground commutative ring $\kk$ since the poset
  $\mychoose{n}{n}$ consists of a single element.
\end{remark*}

This section is devoted to the proofs of our main results, which explicitly determine the derived
endomorphism algebras
\[
  \dgAus{n}{d}\coloneqq\bigoplus_{I,J}\hom[\L{I}][\L{J}],\quad%
  \dgAus*{n}{d}\coloneqq\bigoplus_{J,I}\hom[\L*{J}][\L*{I}]\quad\text{and}\quad\dgAus{n}{d}^\sharp\coloneqq\bigoplus_{J,I}\hom[\L{J}^\sharp][\L{I}^\sharp]
\]
associated to three different sets of generators $\set{\L{I}}$, $\set{\L*{I}}$
and $\set{\L{J}^\sharp}$ of the partially wrapped Fukaya category $\WF{n}{d}$.
The collections $\set{\L*{I}}$ and $\set{\L{I}^\sharp}$
agree up to shift:
\[
  \L{I}^\sharp=\L*{I}[-\rk{I}]
\]
where $\rk{I}$ is the (normalised) rank of $I\in\mychoose{n}{d}$; in particular,
their underlying Lagrangians in $\Sym{d}{\theDisk\setminus\Stops[d][n]}$ are
equal. More precisely, we prove the following statements:
\begin{enumerate}
\item In \Cref{thm:WF-Auslander} we prove that there is quasi-isomorphisms of
  differential graded $\kk$-algebras
  \[
    \morphism{\dgAus{n}{d}}{\Aus{n}{d}}[\simeq]
  \]
  which results in a quasi-equivalence of triangulated $A_\infty$-categories
  \[
    \functor{\perf{\Aus{n}{d}}}{\WF{n}{d},}[\simeq]\qquad\functor*{\Aus{n}{d}}{\bigoplus
   \L{I}.}
\]
\item In \Cref{thm:WF-Auslander:Koszul,thm:dgAus-sharp:AusShiftedKoszul}
  respectively, we prove that there are \emph{isomorphisms} of differential
  graded $\kk$-algebras
  \[
    \dgAus*{n}{d}\cong\Aus*{n}{d}\qquand\dgAus{n}{d}^{\sharp}\cong\Aus{n}{n-d},
  \]
  where we remind the reader that the Koszul dual algebra $\Aus*{n}{d}$ is
  generated in cohomological degrees $0$ and $1$. These result in
  quasi-equivalences of triangulated $A_\infty$-categories
  \[
    \functor{\perf{\Aus*{n}{d}}}{\WF{n}{d},}[\simeq]\qquad\functor*{\Aus*{n}{d}}{\bigoplus
      \L*{I},}
  \]
  and
  \[
    \functor{\perf{\Aus{n}{n-d}}}{\WF{n}{d},}[\simeq]\qquad\functor*{\Aus{n}{n-d}}{\bigoplus
      \L{I}^\sharp,}
  \]
  respectively. Moreover, the collections $\set{\L{I}}$ and $\set{\L*{I}}$ are
  Koszul dual to one another and hence the associated quasi-equivalences
  $\perf{\Aus{n}{d}}\xrightarrow{\simeq}\WF{n}{d}$ and
  $\perf{\Aus*{n}{d}}\xrightarrow{\simeq}\WF{n}{d}$ are intertwined by Koszul
  duality.
\end{enumerate}

\subsection{The strands algebra $\dgBruhat{n}{d}$}

For the sake of completeness, we provide a lift of the `strands algebra with $d$
strands $n$ places' introduced in \cite{LOT18} which takes into account the fact
that we work in arbitrary characteristic. Our construction can be interpreted as
a decorated version of a construction of Khovanov \cite{Kho14}. As we explain
below, the strands algebra is a combinatorially-defined differential graded
$\kk$-algebra which we eventually prove to be quasi-isomorphic to the derived
endomorphism algebra $\dgAus{n}{d}$ of an appropriate set $\set{\L{I}}$ of
generators of the partially wrapped Fukaya category $\WF{n}{d}$.

\subsubsection{The Bruhat order on $\SymGrp{d}$}

As originally observed in \cite{LOT18}, the strands algebra is closely related
to the Bruhat order on the symmetric group $\SymGrp{d}$. We recall from
\cite{BB05} the relevant terminology. By a \emph{word} we mean a (finite) word
in the alphabet $\set{s_1,\dots,s_{d-1}}\subset\SymGrp{d}$ given by the simple
transpositions $s_a=(a+1,a)$ (note that we read words from right to left since
we want to relate these to the composition of morphisms in the Fukaya category).
A word $w'$ is a \emph{subword} of a word $w$ if $w'$ is obtained from $w$ by
deleting some of its letters; in particular, the empty word is a subword of
every word. A \emph{reduced expression} for a permutation $\pi\in\SymGrp{d}$ is
factorisation of $\pi$ as a word of minimal length; a word is \emph{reduced} if
it is a reduced expression for some permutation. The \emph{length} of a
permutation is the length of any of its reduced expressions. Equivalently, the
length of $\pi\in\SymGrp{d}$ is the number
\[
  \inv{\pi}=\#\set{1\leq a<b\leq d}[\pi(b)<\pi(a)]
\]
of \emph{inversions} in $\pi$, see Proposition~1.5.2 in \cite{BB05}. The
relation $\pi'\leq\pi$ holds in the Bruhat order on $\SymGrp{d}$ if and only if
there exists a reduced expression of $\pi$ which contains a reduced expression
of $\pi'$ as a subword, see Theorem~2.2.2 in \cite{BB05}. Equivalently,
$\pi'\leq\pi$ if and only if every reduced expression of $\pi$ contains a
reduced expression of $\pi'$ as a subword, see Corollary~2.2.3 in \cite{BB05}.
The function $\pi\mapsto\inv{\pi}$ endows the Bruhat order on $\SymGrp{d}$ with
the structure of a graded poset, see Theorem~2.2.6 in \cite{BB05}; the elements
of degree $1$ are precisely the simple transpositions. As is customary, we write
$\pi'\triangleleft\pi$ to denote a covering relation in the Bruhat order on
$\SymGrp{d}$; equivalently, we write $\pi'\triangleleft\pi$ if $\pi'<\pi$ and
$\inv{\pi}=\inv{\pi'}+1$.

The Bruhat order on $\SymGrp{d}$ can be visualised diagrammatically as follows:
A simple transposition can be depicted as a planar strand diagram; for
example, the simple transpositions $(21)$ and $(32)$ in $\SymGrp{d}$ correspond
to the diagram
\begin{center}
\begin{tikzpicture}[scale=0.5]
  \pgfmathsetmacro{\r}{0.1}
  \node[label=left:{$1$}] (01) at (0,1) {};%
  \node[label=left:{$2$}] (02) at (0,2) {};%
  \node[label=left:{$3$}] (03) at (0,3) {};%

  \node[label=right:{$1$}] (11) at (1,1) {};%
  \node[label=right:{$2$}] (12) at (1,2) {};%
  \node[label=right:{$3$}] (13) at (1,3) {};%

  \node at (4.5,2) {$\text{and}$};%

  \draw[fill=black] (01) circle (\r);%
  \draw[fill=black] (02) circle (\r);%
  \draw[fill=black] (03) circle (\r);%

  \draw[fill=black] (11) circle (\r);%
  \draw[fill=black] (12) circle (\r);%
  \draw[fill=black] (13) circle (\r);%

  \draw (01.center) -- (12.center);%
  \draw (02.center) -- (11.center);%
  \draw (03.center) -- (13.center);%
\end{tikzpicture}\qquad
\begin{tikzpicture}[scale=0.5]
  \pgfmathsetmacro{\r}{0.1}
  \node[label=left:{$1$}] (01) at (0,1) {};%
  \node[label=left:{$2$}] (02) at (0,2) {};%
  \node[label=left:{$3$}] (03) at (0,3) {};%

  \node[label=right:{$1$}] (11) at (1,1) {};%
  \node[label=right:{$2$}] (12) at (1,2) {};%
  \node[label=right:{$3$}] (13) at (1,3) {};%

  \draw[fill=black] (01) circle (\r);%
  \draw[fill=black] (02) circle (\r);%
  \draw[fill=black] (03) circle (\r);%

  \draw[fill=black] (11) circle (\r);%
  \draw[fill=black] (12) circle (\r);%
  \draw[fill=black] (13) circle (\r);%

  \draw (01.center) -- (11.center);%
  \draw (02.center) -- (13.center);%
  \draw (03.center) -- (12.center);%
\end{tikzpicture}
\end{center}
respectively. Each reduced word
can be visualised as a strand diagram (obtained by concatenating the strands associated to
simple transpositions) in which the crossings that appear are linearly
ordered; the fact that the word is reduced corresponds to the fact that no two
strands cross more than once. Deleting a letter from a word corresponds to
resolving a crossing; note, however, that resolving a crossing may result in a
strand diagram in which two strands cross more than once, see
\Cref{fig:strand_resolution_with_two_crossings} for an example.
\begin{figure}[!h]
  \begin{tikzpicture}[scale=0.5]
  \pgfmathsetmacro{\r}{0.1}
  \node[label=left:{$1$}] (01) at (0,1) {};%
  \node[label=left:{$2$}] (02) at (0,2) {};%
  \node[label=left:{$3$}] (03) at (0,3) {};%

  \node (11) at (1,1) {};%
  \node (12) at (1,2) {};%
  \node (13) at (1,3) {};%

  \draw[fill=black] (01) circle (\r);%
  \draw[fill=black] (02) circle (\r);%
  \draw[fill=black] (03) circle (\r);%

  \draw (01.center) -- (12.center);%
  \draw (02.center) -- (11.center);%
  \draw (03.center) -- (13.center);%

  \begin{scope}[xshift=1cm]
  \node (01) at (0,1) {};%
  \node (02) at (0,2) {};%
  \node (03) at (0,3) {};%

  \node (11) at (1,1) {};%
  \node (12) at (1,2) {};%
  \node (13) at (1,3) {};%

  \draw (01.center) -- (11.center);%
  \draw (02.center) -- (13.center);%
  \draw (03.center) -- (12.center);%
  \begin{scope}[xshift=1cm]
  \node (01) at (0,1) {};%
  \node (02) at (0,2) {};%
  \node (03) at (0,3) {};%

  \node[label=right:{$1$}] (11) at (1,1) {};%
  \node[label=right:{$2$}] (12) at (1,2) {};%
  \node[label=right:{$3$}] (13) at (1,3) {};%

  \node at (4,2) {$\xmapsto{\text{resolve}}$};%

  \draw[fill=black] (11) circle (\r);%
  \draw[fill=black] (12) circle (\r);%
  \draw[fill=black] (13) circle (\r);%

  \draw (01.center) -- (12.center);%
  \draw (02.center) -- (11.center);%
  \draw (03.center) -- (13.center);%
\end{scope}
\end{scope}
\end{tikzpicture}\quad
\begin{tikzpicture}[scale=0.5]
  \pgfmathsetmacro{\r}{0.1} \node[label=left:{$1$}] (01) at (0,1) {};%
  \node[label=left:{$2$}] (02) at (0,2) {};%
  \node[label=left:{$3$}] (03) at (0,3) {};%

  \node (11) at (1,1) {};%
  \node (12) at (1,2) {};%
  \node (13) at (1,3) {};%

  \draw[fill=black] (01) circle (\r);%
  \draw[fill=black] (02) circle (\r);%
  \draw[fill=black] (03) circle (\r);%

  \draw (01.center) -- (12.center);%
  \draw (02.center) -- (11.center);%
  \draw (03.center) -- (13.center);%

  \begin{scope}[xshift=1cm]
    \node (01) at (0,1) {};%
    \node (02) at (0,2) {};%
    \node (03) at (0,3) {};%

    \node (11) at (1,1) {};%
    \node (12) at (1,2) {};%
    \node (13) at (1,3) {};%

    \draw (01.center) -- (11.center);%
    \draw (02.center) -- (12.center);%
    \draw (03.center) -- (13.center);%
    \begin{scope}[xshift=1cm]
      \node (01) at (0,1) {};%
      \node (02) at (0,2) {};%
      \node (03) at (0,3) {};%

      \node[label=right:{$1$}] (11) at (1,1) {};%
      \node[label=right:{$2$}] (12) at (1,2) {};%
      \node[label=right:{$3$}] (13) at (1,3) {};%

      \draw[fill=black] (11) circle (\r);%
      \draw[fill=black] (12) circle (\r);%
      \draw[fill=black] (13) circle (\r);%

      \draw (01.center) -- (12.center);%
      \draw (02.center) -- (11.center);%
      \draw (03.center) -- (13.center);%
    \end{scope}
  \end{scope}

\end{tikzpicture}
\caption{Example of a resolution which results in a strand diagram in which two strands
  cross more than once. On the left, the reduced word $(21)(32)(21)$; on the
  right, the non-reduced word $(21)e(21)$.}
\label{fig:strand_resolution_with_two_crossings}
\end{figure}
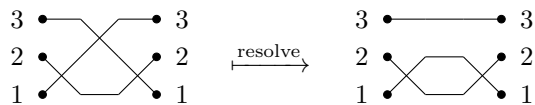
In these terms, the covering relation $\pi'\triangleleft\pi$ holds in the Bruhat
order on $\SymGrp{d}$ if and only if, in any strand diagram representation of $\pi$ in
which no two strands cross more than once, we may resolve a single crossing
(which is in fact uniquely determined, see Corollary~1.4.4 in \cite{BB05}) to
obtain a strand diagram representation of $\pi$ with the same property.

The Bruhat order on $\SymGrp{d}$ is an interval: the minimum element is the
trivial permutation $e\in\SymGrp{d}$ and the maximum element is the permutation
$a\mapsto d+1-a$. Proposition~2.3.4 in \cite{BB05} shows that multiplication by
the longest element induces an anti-automorphism of the Bruhat order on
$\SymGrp{d}$; in particular, the duality principle applies in this context. More
generally, a \emph{Bruhat interval} is an interval
\[
  [\pi',\pi]=\set{\pi''\in\SymGrp{d}}[\pi'\leq\pi''\leq\pi]
\]
in the Bruhat order on $\SymGrp{d}$. It is a remarkable fact that, provided that
$\inv{\pi}-\inv{\pi'}\geq2$, there exists a regular cell
complex, uniquely determined up to cellular homeomorphism, whose poset of cells is isomorphic to $[\pi',\pi]$ (by convention,
$\pi'$ corresponds to an empty cell); moreover, this regular cell complex is
homeomorphic to a ball of dimension $\inv{\pi}-1$, see Theorem~2.7.12 in
\cite{BB05} for details.

Let $\pi\in\SymGrp{d}$ be a non-trivial permutation; following Section~2.7 in
\cite{BB05}, we associate a cochain complex to the Bruhat interval $[e,\pi]$ as
follows. A \emph{balanced signature} on the Hasse diagram of the Bruhat interval
$[e,\pi]$ is an assignment
$\pi'\triangleleft\pi\mapsto\varepsilon(\pi'\triangleleft\pi)$ of signs $\pm1$
to the edges of the Hasse diagram of $[e,\pi]$ such that for every square
\[
  \begin{tikzpicture}
    \node (a) at (90:2em) {$\pi_a$}; \node (b) at (180:2em) {$\pi_b$}; \node (c)
    at (0:2em) {$\pi_c$}; \node (d) at (270:2em) {$\pi_d$}; \draw
    (a)--(b)--(d)--(c)--(a);
  \end{tikzpicture}
\]
the equality
\[
  \varepsilon(\pi_b\triangleleft\pi_a)\varepsilon(\pi_d\triangleleft\pi_b)+\varepsilon(\pi_c\triangleleft\pi_a)\varepsilon(\pi_d\triangleleft\pi_c)=0
\]
is satisfied (every Bruhat interval of length $2$ is of this form,
see Lemma~2.7.3 in \cite{BB05}).
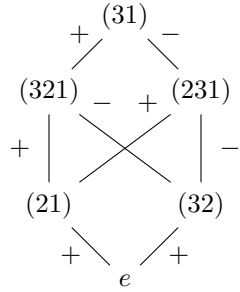
\begin{figure}
  \begin{tikzpicture}
    \node (e) at (0,0) {$e$};%
    \node (21) at (-1,1) {$(21)$};%
    \node (32) at (1,1) {$(32)$};%
    \node (321) at (-1,2.5) {$(321)$};%
    \node (231) at (1,2.5) {$(231)$};%
    \node (31) at (0,3.5) {$(31)$};%
    \draw (e)--node[near start,label=left:{$+$}] {}%
    (21)--node[label=left:{$+$}]{}%
    (321)--node[near start,label={$+$}]{}%
    (31)--node[near end,label={$-$}]{}%
    (231)--node[label=right:{$-$}]{}%
    (32)--node[near end,label=right:{$+$}]{}(e);%
    \draw (21)--node[near end,label={$+$}]{}(231);%
    \draw (32)--node[near end,label={$-$}]{}(321);%
  \end{tikzpicture}
  \caption{The Hasse diagram of Bruhat order on $\SymGrp{3}$ together with a
    choice of balanced signature.}
  \label{fig:Bruhat_S3}
\end{figure}
Given a balanced signature $\varepsilon$ on the Hasse quiver of a Bruhat
interval $[e,\pi]$, we define the free graded $\kk$-module
$C[e,\pi]=\kk\cdot[e,\pi]$, where $\pi\in[e,\pi]$ lies in degree $-\inv{\pi}$,
and endow it with the differential
\[
  \partial_\varepsilon(\pi')=\sum_{\pi''\triangleleft\pi'}\varepsilon(\pi''\triangleleft\pi')\cdot\pi''
\]
for $\pi'\in[e,\pi]$. It is straightforward to verify that
$\partial_\varepsilon^2=0$. In terms of strand diagrams, the above differential is given
by the signed sum of all the permutations obtained from $\pi'$ by resolving a
single crossing in any of its strand diagram representations, excluding resolutions which
result in a diagram in which two strands cross more than once. Note that we can
obtain a new balanced signature $\varepsilon'$ from $\varepsilon$ by choosing an
arbitrary permutation $\pi'\in[e,\pi]$ and multiplying by $-1$ the signs of all
the edges incident to $\pi'$ in the Hasse quiver of $[e,\pi]$. In terms of the
associated cochain complexes, this corresponds to the isomorphism
\[
  \morphism{(C[e,\pi],\partial_{\varepsilon})}{(C[e,\pi],\partial_{\varepsilon'})}[\cong],\qquad\morphism*{\pi'}{-\pi'}
\]
which fixes all the other generators of $C[e,\pi]$. In particular, the above
operation does not change the isomorphism type of the associated chain complex.

Suppose that $\inv{\pi}\geq 2$ and let $X=X[e,\pi]$ be the regular cell complex
whose poset of cells is isomorphic to the Bruhat interval $[e,\pi]$. An
arbitrary choice of orientation of the cells of $X$ allows us to identify
$C[e,\pi]$ with the cellular chain complex $C(X;\kk)$. In particular, the
cochain complex $C[e,\pi]$ (which computes the reduced cellular homology of the
topological ball $X$) is acyclic. The signs involved in the differential on
$C[e,\pi]$ which arise via this procedure yield a balanced signature on the
Hasse diagram of the Bruhat interval $[e,\pi]$. The above is essentially the
given proof of Corollary~2.7.14 in \cite{BB05}, where general Bruhat intervals
are considered. We need the following slight refinement of this result.

\begin{proposition}
  \label{prop:Bruhat_acyclic_intervals}
  Let $\pi\in\SymGrp{d}$ be a non-trivial permutation. The following statements
  hold:
  \begin{enumerate}
  \item There exists a balanced signature $\varepsilon$ on the Hasse quiver of
    the Bruhat interval $[e,\pi]$ such that the corresponding cochain complex
    $(C[e,\pi],\partial_\varepsilon)$ is acyclic.
  \item Let $\varepsilon$ and $\varepsilon'$ be balanced signatures on the Hasse
    quiver of the Bruhat interval $[e,\pi]$. Then, the corresponding chain
    complexes $(C[e,\pi],\partial_\varepsilon)$ and
    $(C[e,\pi],\partial_{\varepsilon'})$ are isomorphic.
  \end{enumerate}
  In particular, for any choice of balanced signature $\varepsilon$ on the Hasse
  quiver of the Bruhat interval $[e,\pi]$, the cochain complex
  $(C[e,\pi],\partial_{\varepsilon})$ is acyclic.
\end{proposition}
\begin{proof}
  The first statement is an immediate consequence Corollary~2.7.14 in
  \cite{BB05} which establishes the result in the case of integer coefficients
  (the general case follows by extending the coefficients to $\kk$).

  To prove the second statement, we argue inductively as follows. Let
  $\varepsilon$ and $\varepsilon'$ be two different balanced signatures on the
  Hasse diagram of the Bruhat interval $[e,\pi]$. Recall that multiplying by
  $-1$ the signs associated to all the edges in the Hasse diagram which are
  incident to a given permutation does not change the isomorphism type of the
  corresponding cochain complex. We consider first the following special case:
  Suppose that $\varepsilon$ and $\varepsilon'$ agree on the half-open interval
  $[e,\pi)$. We claim that $\varepsilon'$ is obtained from $\varepsilon$ by
  multiplying by $-1$ all the signs associated to all the edges incident to
  $\pi$ in $[e,\pi]$. If the Bruhat interval $[e,\pi]$ has length $1$, that is
  if $\pi$ is a simple transposition, the claim is obvious. Suppose then that
  $\pi$ has length at least $2$. The dual of Corollary 2.2.8~in \cite{BB05}
  implies that every pair of distinct edges $\pi_a\triangleleft\pi$ and
  $\pi_b\triangleleft\pi$ is part of some square
  \[
    \begin{tikzpicture}
      \node (a) at (90:2em) {$\pi$}; \node (b) at (180:2em) {$\pi_a$}; \node (c)
      at (0:2em) {$\pi_b$}; \node (d) at (270:2em) {$\pi_{ab}$}; \draw
      (a)--(b)--(d)--(c)--(a);
    \end{tikzpicture}
  \]
  in the poset $[e,\pi]$. Clearly, once the signs at the edges
  $\pi_{ab}\triangleleft \pi_a$ and $\pi_{ab}\triangleleft \pi_b$ are fixed,
  there are exactly two balanced signatures on the above Hasse diagram and these
  are related to each other by multiplying by $-1$ the signs at the edges
  $\pi_a\triangleleft\pi$ and $\pi_b\triangleleft\pi$. Thus, it is enough to
  show that the value of a balanced signature on any such square determines the
  value on all other squares. If two squares share an edge incident to $\pi$,
  then it is clear that the restriction of the signature to either square
  determines the value on both squares, keeping in mind that the value at the
  bottom edges is already fixed. If this is not the case, then the two squares
  are arranged as follows (depicted by solid edges):
  \[
    \begin{tikzpicture}
      \node (pi) at (0,3) {$\pi$};%
      \node (a) at (-2,2) {$\pi_a$};%
      \node (b) at (-0.75,2) {$\pi_b$};%
      \node (c) at (0.75,2) {$\pi_c$};%
      \node (d) at (2,2) {$\pi_d$};%
      \node (ab) at (-2,1) {$\pi_{ab}$};%
      \node (bc) at (0,1) {$\pi_{bc}$};%
      \node (cd) at (2,1) {$\pi_{cd}$};%
      \draw (pi)--(a);%
      \draw (pi)--(b);%
      \draw (pi)--(c);%
      \draw (pi)--(d);%
      \draw (a)--(ab)--(b); \draw (c)--(cd)--(d);%
      \draw[dashed] (b)--(bc)--(c);%
    \end{tikzpicture}
  \]
  By the above observation regarding pairs of edges incident to $\pi$, there are
  (at least) further covering relations in $[e,\pi]$ as indicated by the dashed
  edges. It is now clear that a balanced signature on the above Hasse diagram is
  completely determined by its values on the edges which are not incident to
  $\pi$ together with its values at any square with maximum element $\pi$.

  We now return to the general case. Inductively, let $1<k\leq\inv{\pi}$ and
  suppose that $\varepsilon$ and $\varepsilon'$ are balanced signatures which
  agree on all edges in the Hasse diagram in $[e,\pi]$ between permutations of
  length at most $k-1$. Let $\pi'\in[e,\pi]$ be a permutation of length $k$ and
  consider the Bruhat interval $[e,\pi']$. The above argument shows that we can
  replace $\varepsilon'$ by a new balanced signature which agrees with
  $\varepsilon$ on $[e,\pi']$ without changing the isomorphism type of the
  associated cochain complex. We can then repeat the argument for all remaining
  permutations $\pi'\leq\pi$ of length $k$. By induction, the claim follows.
\end{proof}

\begin{remark}
  Let $\pi\in\SymGrp{d}$ be a non-trivial permutation. Combined with the proof
  of Corollary~2.7.4 in \cite{BB05}, the proof of
  \Cref{prop:Bruhat_acyclic_intervals} shows that all possible balanced
  signatures on the Hasse diagram of the Bruhat order $\SymGrp{d}$ are obtained
  by choosing arbitrary orientations for the cells of the regular cell complex
  with incidence poset $[e,\pi]$.
\end{remark}

\begin{example}
  The Bruhat order on $\SymGrp{3}$ is isomorphic to the poset of cells
  associated to the following cell decomposition of the disk:
  \[
    \begin{tikzpicture}
      \pgfmathsetmacro{\r}{1}%
      \node (21) at (180:\r+\r*0.5) {$(21)$};%
      \node (21) at (0:\r+\r*0.5) {$(32)$};%
      \node (321) at (90:\r+\r*0.33) {$(321)$};%
      \node (321) at (-90:\r+\r*0.33) {$(231)$};%
      \draw[fill=black] (180:\r) circle (0.05);%
      \draw[fill=black] (0:\r) circle (0.05);%
      \draw[pattern=north west lines] (0,0) circle (\r);%
      \node[fill=white,circle,inner sep=0] (321) at (0:0) {$(31)$};%
    \end{tikzpicture}
  \]
  With respect to the choice of balanced signature in \Cref{fig:Bruhat_S3}, the
  cochain complex $C[e,(31)]$ is given by
  \[
    \begin{tikzcd}[ampersand replacement=\&,column sep=large]
      0\rar\&
      \kk\cdot(31)\rar{\left(\begin{smallmatrix}1\\-1\end{smallmatrix}\right)}\&\kk\cdot(321)\oplus\kk\cdot(231)\rar\rar{\left(\begin{smallmatrix}1&1\\-1&-1\end{smallmatrix}\right)}\&\kk\cdot(21)\oplus\kk\cdot(32)\rar{\left(\begin{smallmatrix}1&1\end{smallmatrix}\right)}\&\kk\cdot
      e
    \end{tikzcd}
  \]
  The reader can easily verify that the above cochain complex is acyclic.
\end{example}

\subsubsection{The definition of the strands algebra $\dgBruhat{n}{d}$}

In \cite{Kho14}, Khovanov introduced a characteristic-free simplified variant of
the strands algebra of \cite{LOT18} which turns out to be contractible (in his
definition, the unit is a coboundary). We introduce a decorated version of
Khovanov's construction which provides a cohomologically-non-trivial lift of
the strands algebra of \cite{LOT18} to arbitrary characteristic. We begin with a
technical lemma.

\begin{lemma}
  \label{lemma:intervalIJ}
  Let $I,J\in\mychoose{n}{d}$. If non-empty, the subset
  \[
    \set{\pi\in\SymGrp{d}}[\forall1\leq a\leq d:i_a\leq j_{\pi(a)}]
  \]
  is a Bruhat interval of the form $\interval{I}{J}$, where $e\in\SymGrp{d}$ is
  the trivial permutation. In particular, the above subset is non-empty if and
  only if $I\leq J$.
\end{lemma}
\begin{proof}
  Let $X=X_{JI}=\set{\pi\in\SymGrp{d}}[\forall a:i_a\leq j_{\pi(a)}]$. Suppose
  that $X$ is non-empty. We show first that $X$ is downwards-closed. Since the
  symmetric group is finite, it suffices to show that if $\pi'\triangleleft\pi$
  is a covering relation in the Bruhat order on $\SymGrp{d}$ and $\pi\in X$,
  then $\pi'\in X$. By Lemma~2.1.4 in \cite{BB05}, we can write $\pi(b,a)=\pi'$
  for some $a<b$ such that $\pi(b)<\pi(a)$ and there does not exist $a<c<b$ such
  that $\pi(b)<\pi(c)<\pi(a)$. In particular
  \[
    i_a<i_b\leq j_{\pi(b)}=j_{\pi'(a)}<j_{\pi(a)}=j_{\pi(b)}.
  \]
  Therefore $\pi'\in X$, as required. Finally, $X$ contains a maximum
  $\pi^{JI}_0$, constructed inductively as follows: If $d=1$ then $I=\set{i}$
  and $J=\set{j}$ for some $i\leq j$ and $\pi^{JI}_0=e$ is the trivial
  permutation. For $d>1$ we let $\pi^{JI}_0(d)=\min\set{a}[i_d\leq j_a]$ and
  define $I'=I\setminus\set{i_d}$ and $J'=J\setminus\set{j_{\pi^{JI}_0(d)}}$.
  Note that $I'\leq J'$. Inductively, we define $\pi^{JI}_0$ on
  $\set{1,\dots,d-1}$ to be the unique maximal element of the poset
  \[
    \set{\pi\in\SymGrp{d-1}}[\forall 1\leq a<d:i_a\leq j_{\pi(a)}],
  \]
  which is non-empty since it contains the trivial permutation.
\end{proof}

\begin{notation}
  Let $I,J\in\mychoose{n}{d}$. With some abuse of notation, we let
  $\interval{I}{J}\coloneqq\emptyset$ whenever $I\not\leq J$.
\end{notation}

Before giving the definition of the strands algebra we introduce a convenient
diagrammatic language for visualising its generators.

\begin{definition}
  Let $I,J\in\mychoose{n}{d}$ be such that $I\leq J$. We view $I$ and $J$ as
  subsets of the real interval $[0,n+1]$. A \emph{strand diagram from $I$ to
    $J$} consists of a collection $\varphi\colon I\to J$ of $d$ properly
  embedded arcs in the rectangle $[0,1]\times[0,n+1]$, called \emph{strands},
  subject to the following restrictions:
  \begin{enumerate}
  \item Each strand is the graph of a monotonically increasing function
    ${[0,1]\to[0,n+1]}$ whose value at time $0$ lies in $I$ and whose value at
    time $1$ lies in $J$.
  \item No two of strands intersect at their endpoints. 
  \item No two strands intersect more than once. Moreover, all intersections are
    transversal.
  \item At most two strands cross at any given time $t\in[0,1]$.
  \end{enumerate}
  We write $\varphi(i)=j$ to indicate that $\varphi$ contains a strand
  connecting $i\in\set{0}\times I$ with $j\in\set{1}\times J$.
\end{definition}

\begin{remark}
  Let $I,J\in\mychoose{n}{d}$ be such that $I\leq J$. More informally, a strand
  diagram from $I$ to $J$ consists of $d$ distinct curves in the rectangle
  $[0,1]\times[0,n+1]$ connecting exactly one point of $\set{0}\times I$ with
  exactly one point of $\set{1}\times J$. Moreover, strands are required to move
  strictly to the left and weakly upwards as time flows from $0$ to $1$.
  Finally, we require that the number of crossings between any two strands is
  minimal, and that the crossings are positioned in such a way that the
  projection onto the time coordinate induces a linear order on set of crossings
  in the strand diagram. When working in characteristic $2$, the last condition
  on the linear order of the crossings does not play any role in the definition
  of the strands algebra; we may then drop it and recover the strand diagrams
  from \cite{LOT18}.
\end{remark}

\begin{notation}
  Let $I,J\in\mychoose{n}{d}$ be such that $I\leq J$. A strand diagram
  $\varphi\colon I\to J$ defines a permutation $\pi_\varphi\in\interval{I}{J}$
  given by $a\mapsto b$ if $\varphi(i_a)=j_b$; the linear order on the crossings
  of $\varphi$ (which correspond to inversions in $\pi_\varphi$) encodes the
  datum of a reduced expression of $\pi_\varphi$. In particular,
  $\inv{\pi_\varphi}$ is precisely the number of crossings between strands in
  $\varphi$. Conversely, for every reduced expression of a permutation
  $\pi\in\interval{I}{J}$ there exists a (non-unique) strand diagram
  $\varphi=\varphi_\pi$ from $I$ to $J$ such that $\pi_\varphi=\pi$ and which
  encodes the given reduced expression of $\pi$.
\end{notation}

\begin{notation}
  Let $I,J\in\mychoose{n}{d}$ and $\varphi\colon I\to J$ a strand diagram. We
  write $\inv{\varphi}=\inv{\pi_\varphi}$ for the number of crossings in
  $\varphi$.
\end{notation}

\begin{example}
  Let $n=5$ and $d=3$. Let $I=\set{1,2,3}$ and $J=\set{3,4,5}$. Consider the
  strand diagrams
  \begin{center}
    \begin{tikzpicture}[scale=0.5]
      \pgfmathsetmacro{\r}{0.1} \node[label=left:{$1$}] (01) at (0,1) {};%
      \node[label=left:{$2$}] (02) at (0,2) {};%
      \node[label=left:{$3$}] (03) at (0,3) {};%
      \node[label=left:{$4$}] (04) at (0,4) {};%
      \node[label=left:{$5$}] (05) at (0,5) {};%

      \node (11) at (1,1) {};%
      \node (12) at (1,2) {};%
      \node (13) at (1,3) {};%
      \node (14) at (1,4) {};%
      \node (15) at (1,5) {};%

      \draw[fill=black] (01) circle (\r);%
      \draw[fill=black] (02) circle (\r);%
      \draw[fill=black] (03) circle (\r);%

      \draw (01.center) -- (12.center);%
      \draw (02.center) -- (12.center);%
      \draw (03.center) -- (13.center);%

      \begin{scope}[xshift=1cm]
        \node (01) at (0,1) {};%
        \node (02) at (0,2) {};%
        \node (03) at (0,3) {};%
        \node (04) at (0,4) {};%
        \node (05) at (0,5) {};%

        \node (11) at (1,1) {};%
        \node (12) at (1,2) {};%
        \node (13) at (1,3) {};%
        \node (14) at (1,4) {};%
        \node (15) at (1,5) {};%

        \draw (02.center) -- (12.center);%
        \draw (02.center) -- (13.center);%
        \draw (03.center) -- (13.center);%
        \begin{scope}[xshift=1cm]
          \node (01) at (0,1) {};%
          \node (02) at (0,2) {};%
          \node (03) at (0,3) {};%
          \node (04) at (0,4) {};%
          \node (05) at (0,5) {};%

          \node (11) at (1,1) {};%
          \node (12) at (1,2) {};%
          \node (13) at (1,3) {};%
          \node (14) at (1,4) {};%
          \node (15) at (1,5) {};%

          \draw (02.center) -- (13.center);%
          \draw (03.center) -- (13.center);%
          \draw (03.center) -- (14.center);%
          \begin{scope}[xshift=1cm]
            \node (01) at (0,1) {};%
            \node (02) at (0,2) {};%
            \node (03) at (0,3) {};%
            \node (04) at (0,4) {};%
            \node (05) at (0,5) {};%

            \node[label=right:{$1$}] (11) at (1,1) {};%
            \node[label=right:{$2$}] (12) at (1,2) {};%
            \node[label=right:{$3$}] (13) at (1,3) {};%
            \node[label=right:{$4$}] (14) at (1,4) {};%
            \node[label=right:{$5$}] (15) at (1,5) {};%

            \node at (4.5,3) {and};%

            \draw[fill=black] (13) circle (\r);%
            \draw[fill=black] (14) circle (\r);%
            \draw[fill=black] (15) circle (\r);%

            \draw (03.center) -- (13.center);%
            \draw (03.center) -- (14.center);%
            \draw (04.center) -- (15.center);%
          \end{scope}
        \end{scope}
      \end{scope}
    \end{tikzpicture}
    \qquad
    \begin{tikzpicture}[scale=0.5]
      \pgfmathsetmacro{\r}{0.1} \node[label=left:{$1$}] (01) at (0,1) {};%
      \node[label=left:{$2$}] (02) at (0,2) {};%
      \node[label=left:{$3$}] (03) at (0,3) {};%
      \node[label=left:{$4$}] (04) at (0,4) {};%
      \node[label=left:{$5$}] (05) at (0,5) {};%

      \node (11) at (1,1) {};%
      \node (12) at (1,2) {};%
      \node (13) at (1,3) {};%
      \node (14) at (1,4) {};%
      \node (15) at (1,5) {};%

      \draw[fill=black] (01) circle (\r);%
      \draw[fill=black] (02) circle (\r);%
      \draw[fill=black] (03) circle (\r);%

      \draw (01.center) -- (12.center);%
      \draw (02.center) -- (13.center);%
      \draw (03.center) -- (13.center);%
      \begin{scope}[xshift=1cm]
        \node (01) at (0,1) {};%
        \node (02) at (0,2) {};%
        \node (03) at (0,3) {};%
        \node (04) at (0,4) {};%
        \node (05) at (0,5) {};%

        \node (11) at (1,1) {};%
        \node (12) at (1,2) {};%
        \node (13) at (1,3) {};%
        \node (14) at (1,4) {};%
        \node (15) at (1,5) {};%

        \draw (03.center) -- (14.center);%
        \draw (02.center) -- (13.center);%
        \draw (03.center) -- (13.center);%
        \begin{scope}[xshift=1cm]
          \node (01) at (0,1) {};%
          \node (02) at (0,2) {};%
          \node (03) at (0,3) {};%
          \node (04) at (0,4) {};%
          \node (05) at (0,5) {};%

          \node (11) at (1,1) {};%
          \node (12) at (1,2) {};%
          \node (13) at (1,3) {};%
          \node (14) at (1,4) {};%
          \node (15) at (1,5) {};%

          \draw (04.center) -- (14.center);%
          \draw (03.center) -- (13.center);%
          \draw (03.center) -- (14.center);%
          \begin{scope}[xshift=1cm]
            \node (01) at (0,1) {};%
            \node (02) at (0,2) {};%
            \node (03) at (0,3) {};%
            \node (04) at (0,4) {};%
            \node (05) at (0,5) {};%

            \node[label=right:{$1$}] (11) at (1,1) {};%
            \node[label=right:{$2$}] (12) at (1,2) {};%
            \node[label=right:{$3$}] (13) at (1,3) {};%
            \node[label=right:{$4$}] (14) at (1,4) {};%
            \node[label=right:{$5$}] (15) at (1,5) {};%

            \draw[fill=black] (13) circle (\r);%
            \draw[fill=black] (14) circle (\r);%
            \draw[fill=black] (15) circle (\r);%

            \draw (03.center) -- (13.center);%
            \draw (04.center) -- (14.center);%
            \draw (04.center) -- (15.center);%
          \end{scope}
        \end{scope}
      \end{scope}
    \end{tikzpicture}
  \end{center}
  whose associated permutation is $(31)$. From left to right, these strand
  diagrams encode the reduced expressions $(21)(32)(21)$ and $(32)(21)(32)$,
  respectively.
\end{example}

\begin{definition}
  Let $I,J\in\mychoose{n}{d}$ such that $I\leq J$. Two strand diagrams
  $\varphi,\psi\colon I\to J$ are \emph{equivalent} if $\pi_\varphi=\pi_\psi$
  and the reduced expressions of this permutation induced by $\varphi$ and
  $\psi$ are equal.
\end{definition}

\begin{remark}
  The equivalence relation between strands diagrams can be phrased in terms of a
  natural notion of isotopy. Namely, two strand diagrams are equivalent if they
  are isotopic via an isotopy through strand diagrams in which the linear order
  induced on the crossings by their time coordinates remains unchanged.
\end{remark}

\begin{remark}
  Let $I,J\in\mychoose{n}{d}$. By construction, equivalence classes of strand
  diagrams $I\to J$ are in bijection with reduced words in $\interval{I}{J}$.
\end{remark}

\begin{notation}
  Let $I,J\in\mychoose{n}{d}$ be such that $I\leq J$. We denote the equivalence
  class of strand diagrams $\varphi\colon I\to J$ such that $\pi_\varphi=e$ by
  $e_{JI}$.
\end{notation}

We define a differential graded $\kk$-category with set of objects
$\mychoose{n}{d}$ as follows:
\begin{itemize}
\item For $I,J\in\mychoose{n}{d}$ we let $\hom[I][J]$ be the $\kk$-module
  generated by equivalence classes of strand diagrams from $I$ to $J$, subject
  to the following relation: If $\varphi,\psi\colon I\to J$ are strand diagrams
  related by moving a single crossing pass another, then $\psi=-\varphi$.
\item Let $\varphi\colon I\to J$ and $\psi\colon J\to K$ be strand diagrams. We
  define
  \[
    \psi\circ\varphi=\begin{cases}
      \psi\varphi&\text{if }\inv{\pi_\psi\pi_\varphi}=\inv{\pi_\psi}+\inv{\pi_\varphi},\\
      0&\text{otherwise}.
    \end{cases}
  \]
  where the strand diagram $\psi\varphi$ is obtained by horizontal concatenation
  with the implicit re-parametrisation of the time coordinate (recall that
  strands flow from left to right). More informally, the natural composition law
  has the following important caveat: If after concatenating two strand diagrams
  the resulting strand diagram contains two strands which cross more than once,
  their composite vanishes.
\item We endow $\hom[I][J]$ with a grading by declaring the degree $|\varphi|$
  of a strand diagram $\varphi$ to be $-\inv{\varphi}=-\inv{\pi_\varphi}$, that
  is minus its number of crossings. Since the composition operation is by design
  compatible with the grading, this yields a graded $\kk$-category.
\item Finally, we endow the graded $\kk$-module $\hom[I][J]$ with the
  differential
  \[
    \partial(\varphi)=\sum_{\psi\triangleleft
      \varphi}\varepsilon(\psi\triangleleft\varphi)\psi
  \]
  where the sum ranges over all strand diagrams $\psi$ obtained from $\varphi$
  by resolving a single crossing, excluding resolutions which result in a strand
  diagram in which two strands cross more than once, and
  $\varepsilon(\psi\triangleleft\varphi)$ is given by $-1$ to the number of
  crossings in $\varphi$ which happen \emph{after} the crossing being resolved.
\end{itemize}

\begin{defprop}
  The \emph{strands algebra with $d$ strands and $n$ places} is the differential
  graded $\kk$-algebra
  \[
    \dgBruhat{n}{d}\coloneqq\bigoplus_{I,J}\hom[I][J].\qedhere
  \]
\end{defprop}
\begin{proof}
  It is straightforward to verify the differential on $\hom[I][J]$ squares to
  $0$. To prove that the differential satisfies the graded Leibniz rule we may
  proceed exactly as in the proof of Lemma~3.1 in \cite{LOT18}, with a small
  modification. We include the argument for the sake of completeness. We define
  a larger graded $\kk$-algebra $\widetilde{\mathcal{B}}_{n,d}$ where we allow all
  generalised strand diagrams obtained by concatenating strand diagrams in which
  no two strands cross more than once. The degree of a generalised strand
  diagram is defined to be minus the number of crossings in the diagram (which
  may be larger than the number of inversions in the corresponding permutation).
  The product is given by horizontal concatenation of generalised strand
  diagrams (with no caveats) and the differential is defined exactly as for the
  strands algebra $\dgBruhat{n}{d}$. Clearly, as a graded $\kk$-algebra,
  $\dgBruhat{n}{d}$ is a quotient of $\widetilde{\mathcal{B}}_{n,d}$. We verify
  that the differential on $\widetilde{\mathcal{B}}_{n,d}$ satisfies the graded
  Leibniz rule. Let $\varphi\colon I\to J$ and $\psi\colon J\to K$ be
  generalised strand diagrams. We analyse the expression
  \[
    \partial(\psi)\varphi+(-1)^{|\psi|}\psi\partial(\varphi).
  \]
  The term $\partial(\psi)\varphi$ corresponds precisely to the summand of
  $\partial(\psi\varphi)$ obtain be resolving crossings in $\psi\varphi$ which
  are crossings in $\psi$. Similarly, the remaining term
  $(-1)^{|\psi|}\psi\partial(\varphi)$ corresponds precisely to the summand of
  $\partial(\psi\varphi)$ obtain be resolving crossings in $\psi\varphi$ which
  are crossings in $\psi$; the sign $(-1)^{|\psi|}$ appears since
  $\partial(\varphi)$ does not take into account the number $-|\psi|$ of
  crossings in $\psi$, and all of these crossings happen after every crossing in
  $\varphi$. This shows that $\widetilde{\mathcal{B}}_{n,d}$ is a differential
  graded $\kk$-algebra. To finish the proof, it remains to show that the graded
  submodule of $\widetilde{\mathcal{B}}_{n,d}$ generated by the generalised strand
  diagrams with at least one double crossing is in fact a differential graded
  ideal. It is clear that this graded $\kk$-submodule is a graded ideal. Also,
  if a strand diagram $\varphi$ has at least two double crossings then
  $\partial(\varphi)$ has at least one double crossing and therefore lies in the
  ideal. Finally, if $\varphi$ has a single double crossing then there are
  exactly two strand diagrams $\psi$ and $\psi'$ which appear as summands of
  $\partial(\varphi)$ which do not lie in the ideal; we claim that these terms
  cancel each other. Indeed, since the permutations $\pi_{\psi}$ and
  $\pi_{\psi'}$ clearly agree, $\psi$ and $\psi'$ differ from each other by a
  sign as generators of $\widetilde{\mathcal{B}}_{n,d}$; this sign is given by
  $-1$ to the number of crossings that appear strictly between the two crossings
  which are part of the unique double crossing in $\varphi$. The claim follows.
\end{proof}

\begin{example}
  Let $n=5$ and $d=4$. Let $\set{1,2,3,4}$ and $J=\set{2,3,4,5}$. As generators
  of the strands algebra $\dgBruhat{n}{d}$, the strand diagrams
  \begin{center}
    \begin{tikzpicture}[scale=0.5]
      \pgfmathsetmacro{\r}{0.1}%
      \node[label=left:{$1$}] (01) at (0,1) {};%
      \node[label=left:{$2$}] (02) at (0,2) {};%
      \node[label=left:{$3$}] (03) at (0,3) {};%
      \node[label=left:{$4$}] (04) at (0,4) {};%
      \node[label=left:{$5$}] (05) at (0,5) {};%

      \node (11) at (1,1) {};%
      \node (12) at (1,2) {};%
      \node (13) at (1,3) {};%
      \node (14) at (1,4) {};%
      \node (15) at (1,5) {};%

      \draw[fill=black] (01) circle (\r);%
      \draw[fill=black] (02) circle (\r);%
      \draw[fill=black] (03) circle (\r);%
      \draw[fill=black] (04) circle (\r);%

      \draw (01.center) -- (12.center);%
      \draw (02.center) -- (12.center);%
      \draw (03.center) -- (13.center);%
      \draw (04.center) -- (14.center);%

      \begin{scope}[xshift=1cm]
        \node (01) at (0,1) {};%
        \node (02) at (0,2) {};%
        \node (03) at (0,3) {};%
        \node (04) at (0,4) {};%
        \node (05) at (0,5) {};%

        \node (11) at (1,1) {};%
        \node (12) at (1,2) {};%
        \node (13) at (1,3) {};%
        \node (14) at (1,4) {};%
        \node (15) at (1,5) {};%


        \draw (02.center) -- (12.center);%
        \draw (03.center) -- (14.center);%
        \draw (02.center) -- (13.center);%
        \draw (04.center) -- (14.center);%
        \begin{scope}[xshift=1cm]
          \node (01) at (0,1) {};%
          \node (02) at (0,2) {};%
          \node (03) at (0,3) {};%
          \node (04) at (0,4) {};%
          \node (05) at (0,5) {};%

          \node[label=right:{$1$}] (11) at (1,1) {};%
          \node[label=right:{$2$}] (12) at (1,2) {};%
          \node[label=right:{$3$}] (13) at (1,3) {};%
          \node[label=right:{$4$}] (14) at (1,4) {};%
          \node[label=right:{$5$}] (15) at (1,5) {};%

          \draw[fill=black] (12) circle (\r);%
          \draw[fill=black] (13) circle (\r);%
          \draw[fill=black] (14) circle (\r);%
          \draw[fill=black] (15) circle (\r);%

          \node (11) at (1,1) {};%
          \node (12) at (1,2) {};%
          \node (13) at (1,3) {};%
          \node (14) at (1,4) {};%
          \node (15) at (1,5) {};%
          
          \node at (4.5,3) {and};%

          \draw (02.center) -- (12.center);%
          \draw (03.center) -- (13.center);%
          \draw (04.center) -- (14.center);%
          \draw (04.center) -- (15.center);%
        \end{scope}
      \end{scope}
    \end{tikzpicture}\qquad
    \begin{tikzpicture}[scale=0.5]
      \pgfmathsetmacro{\r}{0.1}%
      \node[label=left:{$1$}] (01) at (0,1) {};%
      \node[label=left:{$2$}] (02) at (0,2) {};%
      \node[label=left:{$3$}] (03) at (0,3) {};%
      \node[label=left:{$4$}] (04) at (0,4) {};%
      \node[label=left:{$5$}] (05) at (0,5) {};%

      \node (11) at (1,1) {};%
      \node (12) at (1,2) {};%
      \node (13) at (1,3) {};%
      \node (14) at (1,4) {};%
      \node (15) at (1,5) {};%

      \draw[fill=black] (01) circle (\r);%
      \draw[fill=black] (02) circle (\r);%
      \draw[fill=black] (03) circle (\r);%
      \draw[fill=black] (04) circle (\r);%

      \draw (01.center) -- (11.center);%
      \draw (02.center) -- (12.center);%
      \draw (03.center) -- (14.center);%
      \draw (04.center) -- (14.center);%

      \begin{scope}[xshift=1cm]
        \node (01) at (0,1) {};%
        \node (02) at (0,2) {};%
        \node (03) at (0,3) {};%
        \node (04) at (0,4) {};%
        \node (05) at (0,5) {};%

        \node (11) at (1,1) {};%
        \node (12) at (1,2) {};%
        \node (13) at (1,3) {};%
        \node (14) at (1,4) {};%
        \node (15) at (1,5) {};%


        \draw (01.center) -- (12.center);%
        \draw (04.center) -- (15.center);%
        \draw (02.center) -- (12.center);%
        \draw (04.center) -- (14.center);%
        \begin{scope}[xshift=1cm]
          \node (01) at (0,1) {};%
          \node (02) at (0,2) {};%
          \node (03) at (0,3) {};%
          \node (04) at (0,4) {};%
          \node (05) at (0,5) {};%

          \node[label=right:{$1$}] (11) at (1,1) {};%
          \node[label=right:{$2$}] (12) at (1,2) {};%
          \node[label=right:{$3$}] (13) at (1,3) {};%
          \node[label=right:{$4$}] (14) at (1,4) {};%
          \node[label=right:{$5$}] (15) at (1,5) {};%

          \draw[fill=black] (12) circle (\r);%
          \draw[fill=black] (13) circle (\r);%
          \draw[fill=black] (14) circle (\r);%
          \draw[fill=black] (15) circle (\r);%

          \node (11) at (1,1) {};%
          \node (12) at (1,2) {};%
          \node (13) at (1,3) {};%
          \node (14) at (1,4) {};%
          \node (15) at (1,5) {};%

          \draw (02.center) -- (12.center);%
          \draw (02.center) -- (13.center);%
          \draw (04.center) -- (14.center);%
          \draw (05.center) -- (15.center);%
        \end{scope}
      \end{scope}
    \end{tikzpicture}    
  \end{center}
  are the negative of each other.
\end{example}

\begin{example}
  Let $n=5$ and $d=3$. Let $I=\set{1,2,3}$, $J=\set{2,3,4}$ and $K=\set{3,4,5}$.
  Consider the following composite of strand diagrams $I\to J$ and $J\to K$:
    \begin{center}
    \begin{tikzpicture}[scale=0.5]
      \pgfmathsetmacro{\r}{0.1}%
      \node[label=left:{$1$}] (01) at (0,1) {};%
      \node[label=left:{$2$}] (02) at (0,2) {};%
      \node[label=left:{$3$}] (03) at (0,3) {};%
      \node[label=left:{$4$}] (04) at (0,4) {};%
      \node[label=left:{$5$}] (05) at (0,5) {};%

      \node (11) at (1,1) {};%
      \node (12) at (1,2) {};%
      \node (13) at (1,3) {};%
      \node (14) at (1,4) {};%
      \node (15) at (1,5) {};%

      \draw[fill=black] (01) circle (\r);%
      \draw[fill=black] (02) circle (\r);%
      \draw[fill=black] (03) circle (\r);%

      \draw (01.center) -- (12.center);%
      \draw (02.center) -- (12.center);%
      \draw (03.center) -- (13.center);%

      \begin{scope}[xshift=1cm]
        \node (01) at (0,1) {};%
        \node (02) at (0,2) {};%
        \node (03) at (0,3) {};%
        \node (04) at (0,4) {};%
        \node (05) at (0,5) {};%

        \node[label=right:{$1$}] (11) at (1,1) {};%
        \node[label=right:{$2$}] (12) at (1,2) {};%
        \node[label=right:{$3$}] (13) at (1,3) {};%
        \node[label=right:{$4$}] (14) at (1,4) {};%
        \node[label=right:{$5$}] (15) at (1,5) {};%

        \node at (3,3) {\LARGE$\bullet$};%

        \draw[fill=black] (12) circle (\r);%
        \draw[fill=black] (13) circle (\r);%
        \draw[fill=black] (14) circle (\r);%

        \draw (02.center) -- (12.center);%
        \draw (02.center) -- (13.center);%
        \draw (03.center) -- (14.center);%
        \begin{scope}[xshift=5cm]
          \node[label=left:{$1$}] (01) at (0,1) {};%
          \node[label=left:{$2$}] (02) at (0,2) {};%
          \node[label=left:{$3$}] (03) at (0,3) {};%
          \node[label=left:{$4$}] (04) at (0,4) {};%
          \node[label=left:{$5$}] (05) at (0,5) {};%

          \node (11) at (1,1) {};%
          \node (12) at (1,2) {};%
          \node (13) at (1,3) {};%
          \node (14) at (1,4) {};%
          \node (15) at (1,5) {};%

          \draw[fill=black] (02) circle (\r);%
          \draw[fill=black] (03) circle (\r);%
          \draw[fill=black] (04) circle (\r);%

          \draw (02.center) -- (13.center);%
          \draw (03.center) -- (13.center);%
          \draw (04.center) -- (14.center);%
          \begin{scope}[xshift=1cm]
            \node (01) at (0,1) {};%
            \node (02) at (0,2) {};%
            \node (03) at (0,3) {};%
            \node (04) at (0,4) {};%
            \node (05) at (0,5) {};%

            \node[label=right:{$1$}] (11) at (1,1) {};%
            \node[label=right:{$2$}] (12) at (1,2) {};%
            \node[label=right:{$3$}] (13) at (1,3) {};%
            \node[label=right:{$4$}] (14) at (1,4) {};%
            \node[label=right:{$5$}] (15) at (1,5) {};%

            \node at (3.5,3) {\LARGE$=$};%
            
            \draw[fill=black] (13) circle (\r);%
            \draw[fill=black] (14) circle (\r);%
            \draw[fill=black] (15) circle (\r);%

            \draw (03.center) -- (13.center);%
            \draw (03.center) -- (14.center);%
            \draw (04.center) -- (15.center);%
          \end{scope}
        \end{scope}
      \end{scope}
    \end{tikzpicture}\quad
    \begin{tikzpicture}[scale=0.5]
      \pgfmathsetmacro{\r}{0.1} \node[label=left:{$1$}] (01) at (0,1) {};%
      \node[label=left:{$2$}] (02) at (0,2) {};%
      \node[label=left:{$3$}] (03) at (0,3) {};%
      \node[label=left:{$4$}] (04) at (0,4) {};%
      \node[label=left:{$5$}] (05) at (0,5) {};%

      \node (11) at (1,1) {};%
      \node (12) at (1,2) {};%
      \node (13) at (1,3) {};%
      \node (14) at (1,4) {};%
      \node (15) at (1,5) {};%

      \draw[fill=black] (01) circle (\r);%
      \draw[fill=black] (02) circle (\r);%
      \draw[fill=black] (03) circle (\r);%

      \draw (01.center) -- (12.center);%
      \draw (02.center) -- (12.center);%
      \draw (03.center) -- (13.center);%

      \begin{scope}[xshift=1cm]
        \node (01) at (0,1) {};%
        \node (02) at (0,2) {};%
        \node (03) at (0,3) {};%
        \node (04) at (0,4) {};%
        \node (05) at (0,5) {};%

        \node (11) at (1,1) {};%
        \node (12) at (1,2) {};%
        \node (13) at (1,3) {};%
        \node (14) at (1,4) {};%
        \node (15) at (1,5) {};%


        \draw (02.center) -- (12.center);%
        \draw (02.center) -- (13.center);%
        \draw (03.center) -- (14.center);%
        \begin{scope}[xshift=1cm]
          \node (01) at (0,1) {};%
          \node (02) at (0,2) {};%
          \node (03) at (0,3) {};%
          \node (04) at (0,4) {};%
          \node (05) at (0,5) {};%

          \node (11) at (1,1) {};%
          \node (12) at (1,2) {};%
          \node (13) at (1,3) {};%
          \node (14) at (1,4) {};%
          \node (15) at (1,5) {};%

          \draw (02.center) -- (13.center);%
          \draw (03.center) -- (13.center);%
          \draw (04.center) -- (14.center);%
          \begin{scope}[xshift=1cm]
            \node (01) at (0,1) {};%
            \node (02) at (0,2) {};%
            \node (03) at (0,3) {};%
            \node (04) at (0,4) {};%
            \node (05) at (0,5) {};%

            \node[label=right:{$1$}] (11) at (1,1) {};%
            \node[label=right:{$2$}] (12) at (1,2) {};%
            \node[label=right:{$3$}] (13) at (1,3) {};%
            \node[label=right:{$4$}] (14) at (1,4) {};%
            \node[label=right:{$5$}] (15) at (1,5) {};%

            \node at (4,3) {\LARGE$=\ 0$};%
            
            \draw[fill=black] (13) circle (\r);%
            \draw[fill=black] (14) circle (\r);%
            \draw[fill=black] (15) circle (\r);%

            \draw (03.center) -- (13.center);%
            \draw (03.center) -- (14.center);%
            \draw (04.center) -- (15.center);%
          \end{scope}
        \end{scope}
      \end{scope}
    \end{tikzpicture}
    \end{center}
    Since concatenating the two strand diagrams results in a strand diagram in
    which two strands cross more than once, their product vanishes in the
    strands algebra $\dgBruhat{n}{d}$.
  \end{example}

\begin{example}
  Let $n=5$ and $d=3$. Let $I=\set{1,2,3}$ and $J=\set{3,4,5}$. Below we
  exemplify the action of the differential on a strand diagram from $I$ to $J$:
  \begin{center}
    \begin{tikzpicture}[scale=0.5]
      \pgfmathsetmacro{\r}{0.1} \node[label=left:{$1$}] (01) at (0,1) {};%
      \node[label=left:{$2$}] (02) at (0,2) {};%
      \node[label=left:{$3$}] (03) at (0,3) {};%
      \node[label=left:{$4$}] (04) at (0,4) {};%
      \node[label=left:{$5$}] (05) at (0,5) {};%

      \node (11) at (1,1) {};%
      \node (12) at (1,2) {};%
      \node (13) at (1,3) {};%
      \node (14) at (1,4) {};%
      \node (15) at (1,5) {};%

      \draw[fill=black] (01) circle (\r);%
      \draw[fill=black] (02) circle (\r);%
      \draw[fill=black] (03) circle (\r);%

      \draw (01.center) -- (12.center);%
      \draw (02.center) -- (12.center);%
      \draw (03.center) -- (13.center);%

      \begin{scope}[xshift=1cm]
        \node (01) at (0,1) {};%
        \node (02) at (0,2) {};%
        \node (03) at (0,3) {};%
        \node (04) at (0,4) {};%
        \node (05) at (0,5) {};%

        \node (11) at (1,1) {};%
        \node (12) at (1,2) {};%
        \node (13) at (1,3) {};%
        \node (14) at (1,4) {};%
        \node (15) at (1,5) {};%

        \draw (02.center) -- (12.center);%
        \draw (02.center) -- (13.center);%
        \draw (03.center) -- (13.center);%
        \begin{scope}[xshift=1cm]
          \node (01) at (0,1) {};%
          \node (02) at (0,2) {};%
          \node (03) at (0,3) {};%
          \node (04) at (0,4) {};%
          \node (05) at (0,5) {};%

          \node (11) at (1,1) {};%
          \node (12) at (1,2) {};%
          \node (13) at (1,3) {};%
          \node (14) at (1,4) {};%
          \node (15) at (1,5) {};%

          \draw (02.center) -- (13.center);%
          \draw (03.center) -- (13.center);%
          \draw (03.center) -- (14.center);%
          \begin{scope}[xshift=1cm]
            \node (01) at (0,1) {};%
            \node (02) at (0,2) {};%
            \node (03) at (0,3) {};%
            \node (04) at (0,4) {};%
            \node (05) at (0,5) {};%

            \node[label=right:{$1$}] (11) at (1,1) {};%
            \node[label=right:{$2$}] (12) at (1,2) {};%
            \node[label=right:{$3$}] (13) at (1,3) {};%
            \node[label=right:{$4$}] (14) at (1,4) {};%
            \node[label=right:{$5$}] (15) at (1,5) {};%

            \node at (4.5,3) {\LARGE$\xmapsto{\partial}$};
          
            \draw[fill=black] (13) circle (\r);%
            \draw[fill=black] (14) circle (\r);%
            \draw[fill=black] (15) circle (\r);%

            \draw (03.center) -- (13.center);%
            \draw (03.center) -- (14.center);%
            \draw (04.center) -- (15.center);%
          \end{scope}
        \end{scope}
      \end{scope}
    \end{tikzpicture}\qquad
    \begin{tikzpicture}[scale=0.5]
      \pgfmathsetmacro{\r}{0.1} \node[label=left:{$1$}] (01) at (0,1) {};%
      \node[label=left:{$2$}] (02) at (0,2) {};%
      \node[label=left:{$3$}] (03) at (0,3) {};%
      \node[label=left:{$4$}] (04) at (0,4) {};%
      \node[label=left:{$5$}] (05) at (0,5) {};%

      \node (11) at (1,1) {};%
      \node (12) at (1,2) {};%
      \node (13) at (1,3) {};%
      \node (14) at (1,4) {};%
      \node (15) at (1,5) {};%

      \draw[fill=black] (01) circle (\r);%
      \draw[fill=black] (02) circle (\r);%
      \draw[fill=black] (03) circle (\r);%

      \draw (01.center) -- (11.center);%
      \draw (02.center) -- (12.center);%
      \draw (03.center) -- (13.center);%

      \begin{scope}[xshift=1cm]
        \node (01) at (0,1) {};%
        \node (02) at (0,2) {};%
        \node (03) at (0,3) {};%
        \node (04) at (0,4) {};%
        \node (05) at (0,5) {};%

        \node (11) at (1,1) {};%
        \node (12) at (1,2) {};%
        \node (13) at (1,3) {};%
        \node (14) at (1,4) {};%
        \node (15) at (1,5) {};%

        \draw (01.center) -- (12.center);%
        \draw (02.center) -- (13.center);%
        \draw (03.center) -- (13.center);%
        \begin{scope}[xshift=1cm]
          \node (01) at (0,1) {};%
          \node (02) at (0,2) {};%
          \node (03) at (0,3) {};%
          \node (04) at (0,4) {};%
          \node (05) at (0,5) {};%

          \node (11) at (1,1) {};%
          \node (12) at (1,2) {};%
          \node (13) at (1,3) {};%
          \node (14) at (1,4) {};%
          \node (15) at (1,5) {};%

          \draw (02.center) -- (13.center);%
          \draw (03.center) -- (13.center);%
          \draw (03.center) -- (14.center);%
          \begin{scope}[xshift=1cm]
            \node (01) at (0,1) {};%
            \node (02) at (0,2) {};%
            \node (03) at (0,3) {};%
            \node (04) at (0,4) {};%
            \node (05) at (0,5) {};%

            \node[label=right:{$1$}] (11) at (1,1) {};%
            \node[label=right:{$2$}] (12) at (1,2) {};%
            \node[label=right:{$3$}] (13) at (1,3) {};%
            \node[label=right:{$4$}] (14) at (1,4) {};%
            \node[label=right:{$5$}] (15) at (1,5) {};%

            \node at (3.5,3) {\LARGE$+$};

            \draw[fill=black] (13) circle (\r);%
            \draw[fill=black] (14) circle (\r);%
            \draw[fill=black] (15) circle (\r);%

            \draw (03.center) -- (13.center);%
            \draw (03.center) -- (14.center);%
            \draw (04.center) -- (15.center);%
          \end{scope}
        \end{scope}
      \end{scope}
    \end{tikzpicture}\quad
    \begin{tikzpicture}[scale=0.5]
      \pgfmathsetmacro{\r}{0.1} \node[label=left:{$1$}] (01) at (0,1) {};%
      \node[label=left:{$2$}] (02) at (0,2) {};%
      \node[label=left:{$3$}] (03) at (0,3) {};%
      \node[label=left:{$4$}] (04) at (0,4) {};%
      \node[label=left:{$5$}] (05) at (0,5) {};%

      \node (11) at (1,1) {};%
      \node (12) at (1,2) {};%
      \node (13) at (1,3) {};%
      \node (14) at (1,4) {};%
      \node (15) at (1,5) {};%

      \draw[fill=black] (01) circle (\r);%
      \draw[fill=black] (02) circle (\r);%
      \draw[fill=black] (03) circle (\r);%

      \draw (01.center) -- (12.center);%
      \draw (02.center) -- (12.center);%
      \draw (03.center) -- (13.center);%

      \begin{scope}[xshift=1cm]
        \node (01) at (0,1) {};%
        \node (02) at (0,2) {};%
        \node (03) at (0,3) {};%
        \node (04) at (0,4) {};%
        \node (05) at (0,5) {};%

        \node (11) at (1,1) {};%
        \node (12) at (1,2) {};%
        \node (13) at (1,3) {};%
        \node (14) at (1,4) {};%
        \node (15) at (1,5) {};%

        \draw (02.center) -- (12.center);%
        \draw (02.center) -- (13.center);%
        \draw (03.center) -- (13.center);%
        \begin{scope}[xshift=1cm]
          \node (01) at (0,1) {};%
          \node (02) at (0,2) {};%
          \node (03) at (0,3) {};%
          \node (04) at (0,4) {};%
          \node (05) at (0,5) {};%

          \node (11) at (1,1) {};%
          \node (12) at (1,2) {};%
          \node (13) at (1,3) {};%
          \node (14) at (1,4) {};%
          \node (15) at (1,5) {};%

          \draw (02.center) -- (12.center);%
          \draw (03.center) -- (13.center);%
          \draw (03.center) -- (14.center);%
          \begin{scope}[xshift=1cm]
            \node (01) at (0,1) {};%
            \node (02) at (0,2) {};%
            \node (03) at (0,3) {};%
            \node (04) at (0,4) {};%
            \node (05) at (0,5) {};%

            \node[label=right:{$1$}] (11) at (1,1) {};%
            \node[label=right:{$2$}] (12) at (1,2) {};%
            \node[label=right:{$3$}] (13) at (1,3) {};%
            \node[label=right:{$4$}] (14) at (1,4) {};%
            \node[label=right:{$5$}] (15) at (1,5) {};%

            \draw[fill=black] (13) circle (\r);%
            \draw[fill=black] (14) circle (\r);%
            \draw[fill=black] (15) circle (\r);%

            \draw (02.center) -- (13.center);%
            \draw (03.center) -- (14.center);%
            \draw (04.center) -- (15.center);%
          \end{scope}
        \end{scope}
      \end{scope}
    \end{tikzpicture}
  \end{center}
  Notice that resolving the second crossing would result in a strand diagram in
  which two strands cross more than once; therefore, this resolution does not
  contribute to the differential.
\end{example}

The following observation is crucial for our purposes.

\begin{proposition}
  \label{lemma:Khovanov}
  Let $I,J\in\mychoose{n}{d}$ be such that $I\leq J$. For each permutation
  $\pi\in\interval{I}{J}$ choose a strand diagram $\varphi=\varphi_\pi\colon
  I\to J$ such that $\pi_\varphi=\pi$. The set
  $\set{\varphi_\pi}[\pi\in\interval{I}{J}]$ is a free $\kk$-basis of the graded
  $\kk$-module $\hom[I][J]$. In particular, $\hom[I][J]$ is isomorphic to the
  cochain complex $C\interval{I}{J}$.
\end{proposition}
\begin{proof}
  The proof of Lemma~1 in \cite{Kho14} applies verbatim. The second claim
  follows from \Cref{prop:dgBruhat:acyclic}, taking into account that the signs
  in the definition of the differential on $\hom[I][J]$ yield a balanced
  signature on the Hasse quiver of the Bruhat interval $\interval{I}{J}$.
\end{proof}

We also record the following elementary observation.

\begin{lemma}
  \label{lemma:dgBruhat:unique_differential}
  There is a unique differential on the underlying graded $\kk$-algebra of the
  strands algebra $\dgBruhat{n}{d}$ which endows it with the structure of a
  differential graded $\kk$-algebra, subject to the additional restriction that
  $\partial(\varphi)=e_{JI}$ for every strand diagram $\varphi\colon I\to J$
  with a single crossing.
\end{lemma}
\begin{proof}
  By construction, the strands algebra $\dgBruhat{n}{d}$ is generated as a
  graded $\kk$-algebra in cohomological degrees $0$ and $-1$ (since every
  reduced word is, by definition, a product of simple transpositions). The
  graded Leibniz rule implies that the differential on $\dgBruhat{n}{d}$ is
  completely determined by its action on the generators of these cohomological
  degrees. The claim follows.
\end{proof}

\subsubsection{The quasi-isomorphism $\dgBruhat{n}{d}\simeq\Aus{n}{d}$}

The following theorem is the first step in establishing the quasi-equivalence
between the partially wrapped Fukaya category $\WF{n}{d}$ and the perfect
derived category $\perf{\Aus{n}{d}}$.

\begin{theorem}
  \label{thm:Bruhat_qiso_Auslander}
  There is a quasi-isomorphism of differential graded $\kk$-algebras
  \[
    \morphism{\dgBruhat{n}{d}}{H^0(\dgBruhat{n}{d})}[\simeq]
  \]
  and an isomorphism of (ungraded) $\kk$-algebras
  $H^0(\dgBruhat{n}{d})\cong\Aus{n}{d}$.
\end{theorem}

\begin{remark}
  When working in characteristic $2$, the cohomology of the strands algebra
  $\dgBruhat{n}{d}$ is computed additively in Section~4.1 in \cite{LOT15}. However, the
  authors do not relate the cohomology algebra of the strands algebra with the
  higher Auslander algebra $\Aus{n}{d}$.
\end{remark}

\Cref{thm:Bruhat_qiso_Auslander} is a consequence of the following results.

\begin{proposition}
  \label{prop:dgBruhat:acyclic}
  Let $I,J\in\mychoose{n}{d}$ be such that $I\leq J$ and $\interval{I}{J}$ the
  corresponding Bruhat interval. There are isomorphisms of
  graded $\kk$-modules
  \[
    H^*(\hom[I][J])\cong\begin{cases}
      \kk(0)&\text{if }\pi^{JI}_0=e,\\
      0&\text{otherwise}.
      \end{cases}
  \]
  In particular, the differential graded $\kk$-algebra $\dgBruhat{n}{d}$ has its
  cohomology concentrated in degree $0$:
  \[
    H^*(\dgBruhat{n}{d})=H^0(\dgBruhat{n}{d}).\qedhere
  \]
\end{proposition}
\begin{proof}
  Let $I,J\in\mychoose{n}{d}$. By \Cref{lemma:Khovanov} the cochain complex
  $\hom[I][J]$ is isomorphic to $C\interval{I}{J}$. By
  \Cref{prop:dgBruhat:acyclic} the latter complex is acyclic if the permutation
  $\pi^{JI}_0$ is non-trivial and is otherwise isomorphic to the ground ring $\kk$
  placed in cohomological degree $0$. The claim follows.
\end{proof}

\begin{remark}
  The analogue in characteristic $2$ of \Cref{prop:dgBruhat:acyclic} is proven
  in Proposition~4.2 in \cite{LOT15} where, in the case $\pi^{JI}_0\neq e$, the
  authors construct an explicit null-homotopy of the identity morphism of the
  cochain complex $\hom[I][J]$. In order to avoid the sign considerations
  involved when constructing an explicit null-homotopy in arbitrary
  characteristic, we have opted for the alternative approach hinted at in
  Remark~4.4 in \cite{LOT15}.
\end{remark}

\begin{proposition}
  \label{prop:dgBruhat:Auslander}
  There is an isomorphism of (ungraded) $\kk$-algebras $H^0(\dgBruhat{n}{d})\cong\Aus{n}{d}$.
\end{proposition}
\begin{proof}
  Let $I,J\in\mychoose{n}{d}$. By construction, there is an
  isomorphism of graded $\kk$-modules
  \[
    \dgBruhat{n}{d}^0=\bigoplus_{I,J}\hom[I][J][][0]\cong\bigoplus_{I\leq
      J}\kk\cdot e_{JI}.
  \]
  where $e_{JI}\in\interval{I}{J}$ is the trivial permutation. Notice that the
  right-hand side can be identified with the underlying $\kk$-module of the
  incidence $\kk$-algebra of the poset $\mychoose{n}{d}$. Comparing the
  multiplication laws on both sides, we conclude that the above isomorphism is
  in fact an isomorphism of (ungraded) $\kk$-algebras.

  Let $I,J\in\mychoose{n}{d}$ be such that $I\leq J$ and
  $e_{JI}\in\interval{I}{J}$ the trivial permutation (which we identify with the
  unique equivalence class of strand diagrams $I\to J$ with no crossings).
  Suppose that there exists an index $1\leq a<d$ such that $j_a\geq i_{a+1}$. We
  claim that $e_{JI}=0$ in $H^0(\dgBruhat{n}{d})$ in this case. Indeed, the
  simple transposition $s_a=(a+1,a)$ lies in $\interval{I}{J}$ since
  \[
    i_a\leq j_a<j_{a+1}=j_{s_a(a)}\qquand i_{a+1}\leq j_a=j_{s_a(a+1)},
  \]
  where inequalities on the left-hand side hold by definition and the inequality
  on the right-hand side holds by assumption. Clearly,
  $\partial(\varphi)=e_{JI}$ for any strand diagram $\varphi$ such that
  $\pi_\varphi=s_a$. The claim follows.

  The above argument shows that the above isomorphism of (ungraded)
  $\kk$-algebras between $\dgBruhat{n}{d}^0$ and $\bigoplus_{I\leq J}\kk\cdot e_{JI}$ maps the image
  of the differential on the left-hand side to the submodule
  \[
    \langle e_{JI}\ |\ \exists a:j_a\geq i_{a+1} \rangle
  \]
  on the right hand side. Comparing with the definition of
  the higher Auslander algebra $\Aus{n}{d}$, the claim follows.
\end{proof}

We now give the proof of \Cref{thm:Bruhat_qiso_Auslander}

\begin{proof}[Proof of \Cref{thm:Bruhat_qiso_Auslander}]
  According to \Cref{prop:dgBruhat:acyclic}, the differential graded algebra
  $\dgBruhat{n}{d}$ has its cohomology concentrated in degree $0$. This
  immediately implies that the canonical morphism
  \[
    \morphism{\dgBruhat{n}{d}}{H^0(\dgBruhat{n}{d})}[\simeq]
  \]
  is a quasi-isomorphism. Finally, \Cref{prop:dgBruhat:Auslander} shows that there is an
  isomorphism of (ungraded) $\kk$-algebras
  $H^0(\dgBruhat{n}{d})\cong\Aus{n}{d}$.
\end{proof}

\subsection{The quasi-equivalence $\WF{n}{d}\simeq\perf{\Aus{n}{d}}$}

\begin{notation}
  For $1\leq i\leq n$ we introduce the notation
  \[
    \L{i}\coloneqq L_{0i}.
  \]
  More generally, for $I\in\mychoose{n}{d}$ we introduce the Lagrangian
  \[
    \L{I}\coloneqq\prod_{a=1}^d\L{i_a}=\prod_{a=1}^dL_{0i_a}%
  \]
  in $\Sym{d}{\theDisk\setminus\Stops[][n]}$ and define 
  \[
    \dgAus{n}{d}\coloneqq\bigoplus_{I,J}\hom[\L{I}][\L{J}]
  \]
  (the precise grading structures on these Lagrangians are determined in
  \Cref{prop:grading_structures:LI}).
\end{notation}

The following proposition is an immediate consequence of \Cref{thm:Auroux}.

\begin{proposition}
  \label{prop:LI:generate}
  The collection $\set{\L{I}}[I\in\mychoose{n}{d}]$ generates the partially
  wrapped Fukaya category $\WF{n}{d}$ as an idempotent-complete triangulated
  $A_\infty$-category. In particular, there is a quasi-equivalence of
  triangulated $A_\infty$-categories
  \[
    \functor{\perf{\dgAus{n}{d}}}{\WF{n}{d},}[\simeq]\qquad\functor*{\Aus{n}{d}}{\bigoplus_I\L{I}.}\qedhere
  \]
\end{proposition}

In this section we establish the following quasi-equivalence.

\begin{theorem}
  \label{thm:WF-Auslander}
  Let $n\geq d\geq 1$. There is a quasi-isomorphism of differential graded
  $\kk$-algebras
  \[
    \morphism{\dgAus{n}{d}}{\Aus{n}{d}.}[\simeq]
  \]
  Thus, there exists a quasi-equivalence of triangulated $A_\infty$-categories
  \[
    \functor{\perf{\Aus{n}{d}}}{\WF{n}{d},}[\simeq]\qquad\functor*{\Aus{n}{d}}{\bigoplus_I\L{I}.}\qedhere
\]
\end{theorem}

\begin{corollary}
  The differential graded $\kk$-algebras and $\dgAus{n}{d}$ and
  $\dgBruhat{n}{d}$ are related by a zig-zag of quasi-isomorphisms
  \[
    \begin{tikzcd}
      \dgAus{n}{d}\rar{\simeq}&\Aus{n}{d}&\dgBruhat{n}{d}\lar[swap]{\simeq}.
    \end{tikzcd}\qedhere
  \]
\end{corollary}
\begin{proof}
  This is an immediate consequence of \Cref{thm:Bruhat_qiso_Auslander} and
  \Cref{thm:WF-Auslander}.
\end{proof}

\subsubsection{Grading structures}
\label{subsubsec:grading_structures}

Our first task towards the proof of \Cref{thm:WF-Auslander} consist on
constructing suitable grading structures on the Lagrangians
$\set{\L{I}}[I\in\mychoose{n}{d}]$.

\begin{proposition}
  \label{prop::grading_structures:e_poset}
  Let $e\in\SymGrp{d}$ be the trivial permutation. The following statements
  hold:
  \begin{enumerate}
  \item Let $I,J\in\mychoose{n}{d}$. There are isomorphisms of (ungraded)
    $\kk$-modules
    \[
      \hom[\L{I}][\L{J}]^e\cong\begin{cases}
        \kk\cdot f_{JI}&\text{if }I\leq J;\\
        0&\text{otherwise}.
      \end{cases}
    \]
  \item There is an isomorphism of (ungraded) $\kk$-algebras between
    $\dgAus{n}{d}^0=\bigoplus_{I,J}\hom[\L{I}][\L{J}]^e$ and the incidence
    $\kk$-algebra of the poset $\mychoose{n}{d}$.\qedhere
  \end{enumerate}
\end{proposition}
\begin{proof}
  Recall that
  \[
    \hom[\L{I}][\L{J}]^e=\hom[L_{i_1}][L_{j_1}]\otimes\cdots\otimes\hom[L_{i_d}][L_{j_d}].
  \]
  The first claim follows immediately from \eqref{eq:LiLj} since the condition
  that $i_a\leq j_a$ for all $1\leq a\leq d$ is precisely the condition that
  $I\leq J$ in the poset $\mychoose{n}{d}$.

  We now prove the second claim. The isomorphism of (ungraded) $\kk$-modules
  \[
    \dgAus{n}{d}^0\cong\bigoplus_{I\leq J}\kk\cdot f_{JI}
  \]
  induces a product operation on the right-hand side of the form
  \[
    f_{KJ}\circ f_{JI}=\varepsilon_{KI}^Jf_{KI},
  \]
  where the sign $\varepsilon_{KI}^J=\pm1$ is in general difficult to determine
  as it is induced by the orientation of certain moduli spaces of holomorphic
  disks. We claim that, by possibly replacing some of the generators by their
  negatives, we can assume that all these signs are positive (note that this
  immediately proves the second statement in the proposition). We proceed
  inductively as follows. For $d=1$ and $n\geq d$ the poset $\mychoose{n}{d}$ is
  isomorphic to a chain with $n$ elements. It is clear that we can inductively
  choose signs on the generators $f_{i+1,i}\colon \set{i}\to \set{i+1}$ to
  ensure that all signs are positive in this case. Let now $d>1$ and suppose
  that we have proven the claim for all $d-1$ and $n\geq d-1$. Observe first
  that $\mychoose{n-1}{d}$ is a convex subset of $\mychoose{n}{d}$ (that is, a
  subset closed under the passage to closed intervals). Similarly, the map
  $I'\mapsto I'\cup\set{n}$ induces an isomorphism between $\mychoose{n-1}{d-1}$
  and the convex subset
  \[
    \mychoose{n}{d}_n=\set{I\in\mychoose{n}{d}}[n\in I]
  \]
  of $\mychoose{n}{d}$. Thus, by induction, we can assume that the product of
  generators indexed by elements of either of these subsets involves only
  positive signs (where we only allow for products of generators indexed by
  elements of the same convex subset). Let
  \[
    I_{min}=\set{1,\dots,d}\qquand K_{max}=\set{n-d+1,\dots,n-1,n}
  \]
  be the minimal and the maximal elements in $\mychoose{n}{d}$, respectively.
  For each pair $J<K$ such that $J\in\mychoose{n-1}{d}$ and
  $K\in\mychoose{n}{d}_n$, replacing the generator $f_{KJ}$ by its negative if
  necessary, we can assume that
  \[
    f_{KJ}\circ f_{JI_{min}}=f_{KI_{min}}
  \]
  (note that if $J=I_{min}$ this imposes no condition on the generator
  $f_{KJ}=f_{KI_{min}}$). Let $J< K$ be as above and $I\leq J$ (so in
  particular $I\in\mychoose{n-1}{d}$); we claim that
  \[
    f_{KJ}\circ f_{JI}=f_{KI}.
  \]
  Indeed,
  \begin{align*}
    f_{KI_{min}}&=f_{KJ}\circ f_{JI_{min}}\\
                    &=f_{KJ}\circ (f_{JI}\circ f_{II_{min}})\\
                    &=(f_{KJ}\circ f_{JI})\circ f_{II_{min}}\\
                    &=\varepsilon_{KI}^Jf_{KI}\circ f_{II_{min}}=\varepsilon_{KI}^Jf_{KI},
  \end{align*}
  and therefore $\varepsilon_{KI}^J=1$. Now, for each $K\in\mychoose{n}{d}_n$,
  replacing the generator $f_{KI_{min}}$ by its negative if necessary, we
  can assume that
  \[
    f_{K_{max} K}\circ f_{KI_{min}}=f_{K_{max}I_{min}}.
  \]
  Let $K\leq K'$ (so in particular $K'\in\mychoose{n}{d}_n$); we claim that
  \[
    f_{K'K}\circ f_{KI_{min}}=f_{K'I_{min}}.
  \]
  Indeed,
  \begin{align*}
    f_{K_{max}I_{min}}&=f_{K_{max}K}\circ f_{KI_{min}}\\
                      &=(f_{K_{max}K'}\circ f_{K'K})\circ f_{KI_{min}},\\    
                      &=f_{K_{max}K'}\circ (f_{K'K}\circ f_{KI_{min}}),\\
                      &=\varepsilon_{K'I_{min}}^Kf_{K_{max}K'}\circ f_{K'I_{min}}=\varepsilon_{K'I_{min}}^Kf_{K_{max}I_{min}},
  \end{align*}
  and therefore $\varepsilon_{K'I_{min}}^K=1$. We can then proceed similarly as
  above to show that
  \[
    f_{K'K}\circ f_{KJ}=f_{K'J}.
  \]
  for all $J<K\leq K'$ such that $J\in\mychoose{n-1}{d-1}$ and
  $K,K'\in\mychoose{n}{d}_n$, which is what we needed to show. This proves the
  claim for all $n\geq d$ and, by induction on $d$, the general case follows.
\end{proof}

\begin{proposition}
  \label{prop:grading_structures:LI}
  Let $e\in\SymGrp{d}$ be the trivial permutation. Up to simultaneous shift,
  there exist unique grading structures on the Lagrangians
  $\set{\L{I}}[I\in\mychoose{n}{d}]$ such that the graded $\kk$-algebra
  \[
    \bigoplus_{I,J}\hom[\L{I}][\L{J}]^e=\bigoplus_{I\leq J}\hom[\L{I}][\L{J}]^e
  \]
  is concentrated in degree $0$.
\end{proposition}
\begin{proof}
  According to \Cref{prop::grading_structures:e_poset} there is an isomorphism
  of (ungraded) $\kk$-algebras between $\bigoplus_{I,J}\hom[\L{I}][\L{J}]^e$ and
  the incidence $\kk$-algebra of the poset $\mychoose{n}{d}$. In particular, for
  subsets ${I\leq J\leq K}$ in $\mychoose{n}{d}$, the composition map
  \[
    \morphism{\hom[\L{K}][\L{J}]^e\otimes\hom[\L{I}][\L{J}]^e}{\hom[\L{I}][\L{K}]^e}[\simeq]
  \]
  is an isomorphism of graded $\kk$-modules (after choosing \emph{arbitrary}
  grading structures on the corresponding Lagrangians).
  
  Fix an arbitrary grading structure on the Lagrangian $\L{1\cdots d}$. For
  $I\in\mychoose{n}{d}$ equip the Lagrangian $\L{I}$ with the unique grading
  structure such that the graded $\kk$-module
  \[
    \hom[\L{1\cdots d}][\L{I}]\cong\kk
  \]
  is concentrated in degree $0$ (notice that $\set{1,\dots,d}$ is the smallest
  element in the poset $\mychoose{n}{d}$). We claim that this choice of grading
  structures has the desired property. Indeed, if $J\in\mychoose{n}{d}$ is such
  that $J\leq I$, then the composition map
  \[
    \morphism{\hom[\L{J}][\L{I}]^e\otimes\hom[\L{01\cdots
        d}][\L{J}]^e}{\hom[\L{1\cdots d}][\L{I}]^e}[\simeq]
  \]
  is an isomorphism of graded $\kk$-modules. Since, by construction, the graded
  $\kk$-modules
  \[
    \hom[\L{1\cdots d}][\L{J}]^e\qquand\hom[\L{1\cdots d}][\L{I}]^e
  \]
  are concentrated in degree $0$, the graded $\kk$-module $\hom[\L{J}][\L{I}]^e$
  must be concentrated in degree $0$ as well. Since the shift functor on
  $\WF{n}{d}$ induces a free and transitive action on the set of grading
  structures on the Lagrangian $\L{1\cdots d}$, the above argument also shows
  that, up to global shift, there is a unique choice of grading structures with
  the required property.
\end{proof}

\subsubsection{The proof of \Cref{thm:WF-Auslander}}

\begin{notation}
  Once and for all, we fix grading structures on the Lagrangians
  $\set{\L{I}}[I\in\mychoose{n}{d}]$ as in \Cref{prop:grading_structures:LI}.
\end{notation}

\begin{proposition}
  \label{prop:dgAus=dgBruhat:graded}
  Let $I,J\in\mychoose{n}{d}$. There are isomorphisms of graded $\kk$-modules
  \[
    \hom[\L{I}][\L{J}]\cong\begin{cases}
      \hom[I][J]&\text{if }I\leq J,\\
      0&\text{otherwise}.
    \end{cases}
  \]
  In particular, the differential graded $\kk$-algebra $\dgAus{n}{d}$ has its
  cohomology concentrated in degree $0$:
  \[
    H^*(\dgAus{n}{d})=H^0(\dgAus{n}{d}).\qedhere
  \]
\end{proposition}
\begin{proof}
  By \eqref{eq:LiLj} and \Cref{lemma:Khovanov} there is an
  isomorphism of (ungraded) $\kk$-modules
  \[
    \bigoplus_{\pi\in\SymGrp{d}}\hom[\L{I}][\L{J}]^\pi=\bigoplus_{\pi\in\interval{I}{J}}\hom[\L{I}][\L{J}]^\pi\cong\bigoplus_{\pi\in\interval{I}{J}}\kk\cdot\pi.
  \]
  Moreover, since the Lagrangians $\set{\L{I}}[I\in\mychoose{n}{d}]$ all have an
  endpoint in the same boundary component, the results in \cite{Aur10} show that
  the differential on a morphism $f\in\hom[\L{I}][\L{J}]$ lies in the
  $\kk$-module
  \[
    \bigoplus_{\inv{\pi}=\inv{\pi'}+1}\hom[\L{I}][\L{J}]^{\pi'}.
  \]
  Since, by \Cref{prop:grading_structures:LI}, the graded $\kk$-module
  $\hom[\L{I}][\L{J}]^e$ is concentrated in degree $0$, arguing by induction on
  the number of inversions of $\pi$ (and taking into account that the
  differential is a morphism of degree $-1$) we conclude that
  $\hom[\L{I}][\L{J}]^\pi$ is concentrated in degree $-\inv{\pi}$. Thus, the
  above isomorphism can be promoted to an isomorphism of graded $\kk$-modules.
  The claim then follows since the Bruhat interval $\interval{I}{J}$ is
  non-empty if and only if $I\leq J$, see \Cref{lemma:intervalIJ}. The proof of
  \Cref{prop:dgBruhat:acyclic} applies verbatim to show that $\dgAus{n}{d}$ has
  its cohomology concentrated in degree $0$.
\end{proof}

We now give the proof of \Cref{thm:WF-Auslander}

\begin{proof}[Proof of \Cref{thm:WF-Auslander}.]  
  By \Cref{prop:dgAus=dgBruhat:graded} the differential graded $\kk$-algebra
  $\dgAus{n}{d}$ has its cohomology concentrated in degree $0$:
  \[
    H^*(\dgAus{n}{d})\cong H^0(\dgAus{n}{d}).
  \]
  In particular, the canonical map
  \[
    \morphism{\dgAus{n}{d}}{H^0(\dgAus{n}{d})}[\simeq]
  \]
  is a quasi-isomorphism. By \Cref{prop::grading_structures:e_poset} there is
  an isomorphism of ungraded $\kk$-algebras between $\dgAus{n}{d}^0$ and the
  incidence $\kk$-algebra of the poset $\mychoose{n}{d}$. The same argument used
  in the proof of \Cref{prop:dgBruhat:Auslander} then shows that there is an
  isomorphisms of ungraded $\kk$-algebras between $H^0(\dgAus{n}{d})$ and the
  higher Auslander algebra $\Aus{n}{d}$. The existence of the required
  quasi-equivalence
  \[
    \functor{\perf{\Aus{n}{d}}}{\WF{n}{d},}[\simeq]\qquad\functor*{\Aus{n}{n-d}}{\bigoplus_I\L{I}}
  \]
  follows from \Cref{prop:LI:generate}.
\end{proof}

\subsection{The quasi-equivalence $\WF{n}{d}\simeq\perf{\Aus*{n}{d}}$}
\label{subsec:WF:KD}

\begin{notation}
  For $1\leq i\leq n$ we introduce the notation
  \[
    \L*{i}\coloneqq L_{i-1,i}.
  \]
  More generally, for $I\in\mychoose{n}{d}$ we introduce the Lagrangian
  \[
    \L*{I}\coloneqq\prod_{a=1}^d\L*{i_a}=\prod_{a=1}^dL_{i_a-1,i_a}%
  \]
  in $\Sym{d}{\theDisk\setminus\Stops[][n]}$. We equip this Lagrangian with
  the unique grading structure such that the apparent morphism
  \[
    \morphism{\displaystyle\L{I}=\prod_{a=1}^dL_{0,i_a}}{\displaystyle\prod_{a=1}^dL_{i_a-1,i_a}=\L*{I}},
  \]
  given by the Reeb chords $L_{0,i_a}\to L_{i_a-1,i_a}$ induced by the Reeb
  flow along the boundary component of $\theDisk\setminus\Stops[][n]$ labelled
  $i_a$, is a morphism of degree $0$. Finally, we define the differential
  graded $\kk$-algebra
  \[
    \dgAus*{n}{d}\coloneqq\bigoplus_{J,I}\hom[\L*{J}][\L*{I}].\qedhere
  \]
\end{notation}

The following proposition is an immediate consequence of \Cref{thm:Auroux}.

\begin{proposition}
  \label{prop:LIv:generate}
  The collection $\set{\L*{I}}[I\in\mychoose{n}{d}]$ generates the partially
  wrapped Fukaya category $\WF{n}{d}$ as an idempotent-complete triangulated
  $A_\infty$-category. Thus, there exists a quasi-equivalence of triangulated
  $A_\infty$-categories
  \[
    \functor{\perf{\dgAus*{n}{d}}}{\WF{n}{d},}[\simeq]\qquad\functor*{\dgAus*{n}{d}}{\bigoplus_I\L*{I}.}\qedhere
  \]
\end{proposition}

In this section we establish the following equivalences.

\begin{theorem}
  \label{thm:WF-Auslander:Koszul}
  Let $n\geq d\geq 1$. There are an isomorphism of $A_\infty$-algebras
  $\dgAus*{n}{d}\cong\Aus*{n}{d}$ and, consequently, a quasi-equivalence of
  triangulated $A_\infty$-categories
  \[
    \functor{\perf{\Aus*{n}{d}}}{\WF{n}{d},}[\simeq]\qquad\functor*{\Aus*{n}{d}}{\bigoplus_I\L*{I},}
  \]
  where $\Aus*{n}{d}$ denotes the Koszul dual of the (augmented) $\kk$-algebra
  $\Aus{n}{d}$. Moreover, there is a commutative diagram
  \[
    \begin{tikzcd}
      \perf{\Aus*{n}{d}}\drar\ar[leftrightarrow]{rr}{\text{\tiny{Koszul Duality}}}&&\perf{\Aus{n}{d}}\dlar\\
      &\WF{n}{d}
    \end{tikzcd}
  \]
  of quasi-equivalences between triangulated $A_\infty$-categories.
\end{theorem}
\begin{proof}
  By \Cref{thm:WF-Auslander} and \Cref{prop:LIv:generate,prop:LI:generate}, it
  is enough to show that the collections $\set{\L{I}}[I\in\mychoose{n}{d}]$ and
  $\set{\L*{I}}[I\in\mychoose{n}{d}]$ are Koszul dual of each other; This is the
  content of \Cref{prop:Koszul} below. Indeed a theorem of Keller concerning the
  derived Morita equivalence between Koszul dual differential graded
  $\kk$-algebras which are homologically smooth and proper over $\kk$ yields the
  desired quasi-equivalences, see Section~10.5 in \cite{Kel94}.
\end{proof}

\subsubsection{Standard resolutions in $\WF{n}{d}$}
\label{sec:resolutions}

We recall the existence of certain exact triangles in $\WF{n}{d}$, see Lemma~5.2
in \cite{Aur10} for details.

\begin{lemma}[Auroux]
  \label{lemma:Auroux}
  Let $0\leq i<j<k\leq n$. The following statements hold (with respect to a
  suitable choice of grading structures on the corresponding Lagrangians):
  \begin{enumerate}
  \item\label{it:Auroux:triangle:1} There is an exact triangle
    \[
      \begin{tikzcd}[column sep=small]
        L_{ij}\rar&L_{ik}\rar&L_{jk}\rar&L_{ij}[1]
      \end{tikzcd}
    \]
    in the partially wrapped Fukaya category $\WF{n}{1}$.
  \item\label{it:Auroux:triangle} More generally, let $L$ be the product of $d-1$ pairwise disjoint
    properly embedded arcs in $\theDisk\setminus\Stops[][n]$ which are not
    homotopic to any of the arcs $L_{ij}$, $L_{ik}$, $L_{jk}$. There is a
    non-split exact triangle
    \[
      \begin{tikzcd}[column sep=small]
        L\times L_{ij}\rar&L\times L_{ik}\rar&L\times L_{jk}\rar&(L\times
        L_{ij})[1]
      \end{tikzcd}
    \]
    in the partially wrapped Fukaya category $\WF{n}{d}$.
  \item\label{it:Auroux:iso} Suppose that $d\geq 2$. Let $L$ be the product of $d-2$ pairwise
    disjoint properly embedded arcs in $\theDisk\setminus\Stops[][n]$ which are
    not homotopic to any of the arcs $L_{ij}$, $L_{ik}$, $L_{jk}$. There is a
    quasi-isomorphism
    \[
      \morphism{L\times L_{ij}\times L_{ik}}{L\times L_{ij}\times
        L_{jk}}[\simeq]
    \]
    in the partially wrapped Fukaya category $\WF{n}{d}$.\qedhere
  \end{enumerate}
\end{lemma}

\Cref{lemma:Auroux} has several useful corollaries. To state them, it is convenient
to introduce the following notation.

\begin{notation}
  Let $\n_+=\set{0,1,\dots,n}$. The map
  \[
    \morphism*{I}{\d{I}\coloneqq\set{0}\cup I}
  \]
  identifies $\mychoose{n}{d}$ with the subset of $\mychoose{n_+}{d+1}$
  consisting of those $(d+1)$-element subsets $I\subseteq\mathbf{n_+}$ such that
  $0\in I$. For $I=\set{i_0<i_1<\cdots<i_d}$ in $\mychoose{n_+}{d+1}$ we
  introduce the Lagrangian
  \[
    \L{I}=\prod_{a=1}^d\L{i_0,i_a}
  \]
  in $\Sym{d}{\theDisk\setminus\Stops[d][n]}$. Thus, the set
  $\set{\L{I}}[I\in\mychoose{n}{d}]$ of generators of $\WF{n}{d}$can be
  identified with the set
  \[
    \set{\L{I}}[I\in\mychoose{n_+}{d+1}:0\in I].
  \]
  Finally, for a subset $I\in\mychoose{n_+}{d+1}$ such that $0\not\in I$, we
  introduce the subsets
  \[
    \d{I}[a]=\set{0}\cup(I\setminus\set{i_a}),\qquad 0\leq a\leq d
  \]
  together with the corresponding Lagrangians $\PL{I}[a]$ which, by
  construction, belong to the above set of generators of $\WF{n}{d}$.
\end{notation}

\begin{corollary}
  \label{lemma:standard_res:parts}
  Let $I\in\mychoose{n_+}{d+1}$ be such that $0\not\in I$. %
  For each $0\leq a\leq d$, there is an exact triangle
  \[
    \begin{tikzcd}
      \displaystyle L_{i_{a-1},i_{a}}\times\prod_{b\neq
        a,a+1}L_{i_{b-1},i_b}\rar&L_{i_{a-1},i_{a+1}}\times\displaystyle \prod_{b\neq
        a,a+1}L_{i_{b-1},i_b}\dar\\&L_{i_{a},i_{a+1}}\times\displaystyle \prod_{b\neq
        a,a+1}L_{i_{b-1},i_b}\ar[out=180,in=300]{ul}[description]{1}
    \end{tikzcd}
  \]
  in the partially wrapped Fukaya category $\WF{n}{d}$, where the middle term is
  quasi-isomorphic to $\PL{I}[k]$.
\end{corollary}
\begin{proof}
  This is a special case of \Cref{lemma:Auroux}\eqref{it:Auroux:triangle} for
  the triple $0\leq i_{a-1}<i_a<i_{a+1}\leq n$, where $i_{-1}=0$ by convention.
\end{proof}

\begin{corollary}
  Let $I\in\mychoose{n_+}{d+1}$. There is a quasi-isomorphism
  \[
    \morphism{\L{I}}{L_{i_0,i_1}\times L_{i_1,i_2}\times\cdots\times
      L_{i_{d-1},i_d}}[\simeq]
  \]
  in the partially wrapped Fukaya category $\WF{n}{d}$. In particular, if
  $0\not\in I$, there is a quasi-isomorphism
  \[
    \morphism{\PL{I}[a]}{L_{i_{a-1},i_{a+1}}\times\displaystyle\prod_{b\neq
        k,k+1}L_{i_{b-1},i_b}}[\simeq]
  \]
  in $\WF{n}{d}$ for each $0\leq a\leq d$.
\end{corollary}
\begin{proof}
  The existence of the claimed quasi-isomorphism follows by iterative
  application of the quasi-isomorphisms in
  \Cref{lemma:Auroux}\eqref{it:Auroux:iso}.
\end{proof}

\begin{corollary}
  \label{prop:resolutions}
  Fix grading structures on the standard generators of $\WF{n}{d}$ as in
  \Cref{prop:grading_structures:LI}. Let $I\in\mychoose{n_+}{d+1}$ be such that
  $0\not\in I$. Then, the object $\L{I}$ of $\WF{n}{d}$ admits a (triangulated)
  resolution of the form
  \[
    \begin{tikzcd}[column sep=small,row sep=small]
      &\PL{I}[d-1]\drar&&\cdots\drar&&\PL{I}[1]\drar&&\PL{I}[0]\drar\\
      \PL{I}[d]\urar&&X_{d-1}\ar{ll}[description]{1}\urar&\cdots&X_2\urar&&X_{1}\ar{ll}[description]{1}\urar&&\L{I}\ar{ll}[description]{1}
    \end{tikzcd}
  \]
  where the grading structure on $\L{I}$ is uniquely determined by the
  requirement that the morphism $\PL{I}[0]\to \L{I}$ is of degree $0$.
\end{corollary}
\begin{proof}
  The required resolution is obtained by splicing together the exact triangles
  from \Cref{lemma:standard_res:parts} in the apparent manner. The existence of
  a unique grading structure on the Lagrangian $\L{I}$ with the desired property
  is clear since the grading structure on the Lagrangian $\PL{I}[0]$ is already
  fixed.
\end{proof}

\subsubsection{Koszul duality for $\WF{n}{d}$}

\begin{lemma}
  \label{lemma:smooth_and_proper}
  The higher Auslander algebra $\Aus{n}{d}$ is proper and homologically smooth over $\kk$.
\end{lemma}
\begin{proof}
  The higher Auslander algebra $\Aus{n}{d}$ is clearly proper over $\kk$ since
  its underlying $\kk$-module is free of finite rank. To prove that
  $\Aus{n}{d}$ is homologically smooth over $\kk$ one can proceed as follows.
  Firstly, $\Aus{n}{n}\cong\kk$ is certainly homologically smooth over
  $\kk$. Secondly, we observe that there exists a recollement
  \[
    \begin{tikzcd}
      \perf{\Aus{n-1}{d-1}}\rar&\perf{\Aus{n}{d}}\rar\lar[shift left=0.5em]{i_R}\lar[shift right=0.5em]&%
      \perf{\Aus{n-1}{d}},\lar[shift left=0.5em]\lar[shift right=0.5em,swap]{p_L}%
    \end{tikzcd}
  \]
  which expresses $\perf{\Aus{n}{d}}$ as the upper-triangular gluing of
  $\perf{\Aus{n-1}{d-1}}$ and $\perf{\Aus{n-1}{d}}$, see for example
  Proposition~2.50 in \cite{DJW19b}. Inductively, we may assume that the
  $\kk$-algebras $\Aus{n-1}{d-1}$ and $\Aus{n-1}{d}$ are already known to be
  homologically smooth over $\kk$. The claim follows from Theorem~3.24 in
  \cite{LS14}, which shows that the gluing of two differential graded
  $\kk$-categories which are homologically smooth over $\kk$ is again
  homologically smooth over $\kk$, provided that the gluing bimodule is perfect.
  To see that the gluing functor
  \[
    \functor[F\coloneqq i_R\circ p_L]{\perf{\Aus{n-1}{d}}}{\perf{\Aus{n-1}{d-1}}}
  \]
  in the above recollement is indeed given by a perfect bimodule, consider the
  induced (homotopy) colimit-preserving functor
  \[
    \functor[\mathbb{L}F_!]{\operatorname{D}(\Aus{n-1}{d})}{\operatorname{D}(\Aus{n-1}{d-1})}
  \]
  obtained by passing to the $\operatorname{Ind}$-completions of
  $\perf{\Aus{n-1}{d}}$ and $\perf{\Aus{n-1}{d-1}}$, respectively. By
  construction, the functor $\mathbb{L}F_!$ preserves compact objects and
  therefore its underlying bimodule is perfect by Lemma~2.8 in \cite{TV07}
  (keeping in mind that the $\kk$-algebra $\Aus{n-1}{d}$ is homologically
  smooth over $\kk$ by the inductive hypothesis).
\end{proof}

\Cref{thm:WF-Auslander:Koszul} is an immediate consequence of the following
statement.

\begin{proposition}
  \label{prop:Koszul}
  The generating collections
  \[
    \set{\L{I}}[I\in\mychoose{n}{d}]\qquand\set{\L*{I}}[I\in\mychoose{n}{d}]
  \]
  of the partially wrapped Fukaya category $\WF{n}{d}$ are Koszul dual of each
  other.
\end{proposition}
\begin{proof}
  Firstly, note that the $\kk$-algebra $\Aus{n}{d}$ has a natural augmentation
  $\Aus{n}{d}\to\bigoplus_{I}\kk\cdot f_{II}$ obtained by taking the quotient by
  its two-sided ideal generated by $\{f_{JI}, \; I<J\}$.
  \Cref{lemma:Koszul:orthogonal,lemma:homIJ:rk} below show that the collection
  $\set{\L*{I}}[I\in\mychoose{n}{d}]$ satisfy the homological conditions which
  characterise the $\Aus{n}{d}$-module $\bigoplus_{I}\kk\cdot f_{II}$ in the
  perfect derived category $\perf{\Aus{n}{d}}$, noting that the higher Auslander
  algebra $\Aus{n}{d}$ is proper and homologically smooth over $\kk$, see
  \Cref{lemma:smooth_and_proper}.
\end{proof}

\begin{lemma}
  \label{lemma:Koszul:orthogonal}
  Let $I,J\in\mychoose{n}{d}$. There are isomorphisms of graded $\kk$-modules
  \[
    \hom[\L{I}][\L*{J}]\cong\begin{cases}
      \kk(0)&\text{if }I=J,\\
      0&\text{otherwise}.
    \end{cases}\qedhere
  \]
\end{lemma}
\begin{proof}
  Recall that, as (ungraded) $\kk$-modules,
  \[
    \hom[\L{I}][\L*{J}]=\bigoplus_{\pi\in\SymGrp{d}}\hom[\L{I}][\L*{J}]^\pi
  \]
  where
  \[
    \hom[\L{I}][\L*{J}]^\pi=\hom[\L{i_1}][\L*{j_{\pi(1)}}]\otimes\hom[\L{i_2}][\L*{j_{\pi(2)}}]\otimes\cdots\otimes\hom[\L{i_d}][\L*{j_{\pi(d)}}].
  \]
  Moreover, there are isomorphisms of (ungraded) $\kk$-modules
  \[
    \hom[\L{i}][\L*{j}]\cong\begin{cases}
      \kk&\text{if }i=j,\\
      0&\text{otherwise}.
    \end{cases}
  \]
  Indeed, since by definition $\L{i}=L_{0,i}$ and $\L*{j}=\L{j-1,j}$, there is
  an isomorphism $\L{i}\to\L*{j}$ if and only if $i=j$ (note that if $j-1=0$ and
  $i\leq j$ we must have $i=j=1$). It follows readily from the latter
  isomorphisms that the graded $\kk$-module $\hom[\L{I}][\L*{J}]^\pi$ is
  non-zero if and only if $I=J$ and $\pi=e$ is the trivial permutation, see
  \Cref{fig:orthogonality} for an illustration in the case $n=5$.
  \begin{figure}[!h]
    \includegraphics{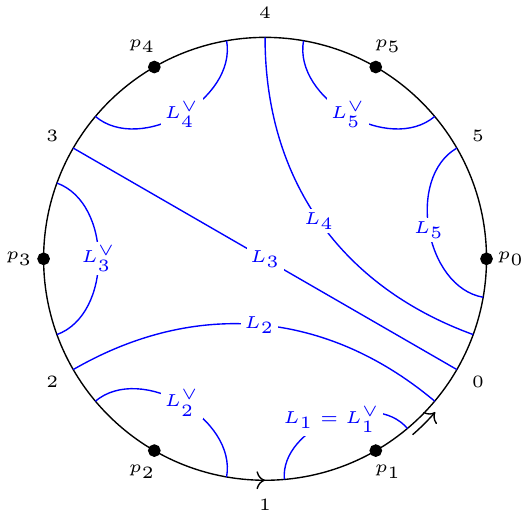}
    \caption{Orthogonality between the generators $\set{\L{I}}$ and
      $\set{\L*{J}}$ in the case $n=5$.}
    \label{fig:orthogonality}
  \end{figure}
  Finally, the grading structure on the Lagrangian $\L*{I}$ is, by definition,
  the unique grading structure such that the graded $\kk$-module
  \[
    \hom[\L{I}][\L*{I}]\cong\kk
  \]
  is concentrated in degree $0$. The claim follows.
\end{proof}

We introduce the following auxiliary notation.

\begin{notation}
  The function $I\mapsto\sum_{a=1}^d i_a$ endows the poset
  $\mychoose{n}{d}\subset\NN^d$ with the structure of a graded poset. For our
  purposes it is more convenient to consider the (normalised) rank function
  \[
    \rk{I}\coloneqq\sum_{a=1}^d(i_a-a),
  \]
  obtained from the usual rank function by subtracting to it the rank of the
  minimal element $\set{1,\dots,d}$ of $\mychoose{n}{d}$.
\end{notation}

\begin{lemma}
  \label{lemma:homIJ:rk}
  Let $I,J\in\mychoose{n}{d}$. There are isomorphisms of graded $\kk$-modules
  \[
    \hom[\L*{J}][\L*{I}]\cong\begin{cases}
      \kk(\rk{I}-\rk{J})&\text{if }\forall a:0\leq j_a-i_a\leq 1,\\
      0&\text{otherwise}.
    \end{cases}
  \]
  In particular, the graded $\kk$-algebra
  \[
    \dgAus*{n}{d}=\bigoplus_{J,I}\hom[\L*{J}][\L*{I}]
  \]
  is generated in cohomological degrees $0$ and $1$.
\end{lemma}
\begin{proof}
  Recall that, as (ungraded) $\kk$-modules,
  \[
    \hom[\L*{J}][\L*{I}]=\bigoplus_{\pi\in\SymGrp{d}}\hom[\L*{J}][\L*{I}]^\pi
  \]
  where
  \[
    \hom[\L*{J}][\L*{I}]^\pi=\hom[\L*{j_1}][\L*{i_{\pi(1)}}]\otimes\hom[\L*{j_2}][\L*{i_{\pi(2)}}]\otimes\cdots\otimes\hom[\L*{j_d}][\L*{i_{\pi(d)}}].
  \]
  Moreover, by \eqref{eq:hom_vee-alt} there are isomorphisms of (ungraded)
  $\kk$-modules
  \[
    \hom[\L*{j}][\L*{i}]\cong\begin{cases}
      \kk&\text{if }0\leq j-i\leq 1,\\
      0&\text{otherwise}.
    \end{cases}
  \]
  It readily follows that there are isomorphisms of (ungraded)
  $\kk$-modules
  \[
    \hom[\L*{J}][\L*{I}]\cong\begin{cases}
      \kk&\text{if }\forall a:0\leq j_a-i_a\leq 1,\\
      0&\text{otherwise}.
    \end{cases}
  \]
  It remains to show that, if $0\leq j_a-i_a\leq 1$, then the graded
  $\kk$-module $\hom[\L*{J}][\L*{I}]$ is in fact concentrated in degree
  $\rk{J}-\rk{I}\geq0$. We consider first the following special case: There
  exists an index $i\in I$ such that $J=(I\setminus\set{i})\cup\set{i+1}$. For
  simplicity, we set
  \[
    L=\prod_{k\in I\setminus\set{i}}\L{k}=\prod_{k\in J\setminus\set{i+1}}\L{k}
  \]
  and
  \[
    L^\vee=\prod_{k\in I\setminus\set{i}}\L*{k}=\prod_{k\in J\setminus\set{i+1}}\L*{k}.
  \]
  Consider now the morphism of exact triangles
  \[
    \begin{tikzcd}[column sep=small]
      \L{0,i}\rar\dar&\L{0,i+1}\rar\dar&\L{i,i+1}\dar[equals]\rar&\L{0,i}[1]\dar\\
      \L{i-1,i}\rar&\L{i-1,i+1}\rar&\L{i,i+1}\rar&\L{i-1,i}[1]      
    \end{tikzcd}    
  \]
  in the partially wrapped Fukaya category $\WF{n}{1}$, see
  \Cref{fig:morhphism} for an illustration depicting the relevant morphisms.
  \begin{figure}
    \includegraphics{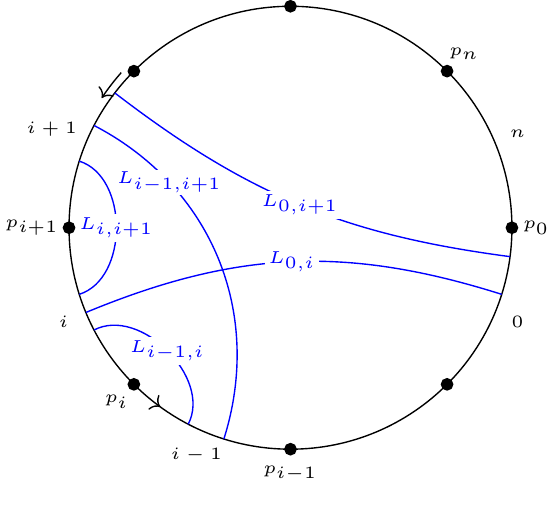}
    \caption{Auxiliary morphisms in the proof of \Cref{lemma:homIJ:rk}.}
    \label{fig:morhphism}
  \end{figure}
  We analyse the induced commutative diagram
  \[
    \begin{tikzcd}[column sep=small]
      \L{i}\times L\rar\dar&\L{i+1}\times L\rar\dar&\L{i+1}^\vee\times L\rar\dar&(\L{i}\times L)[1]\dar\\
      \L*{i}\times L^\vee\rar&\L{i-1,i+1}\times L^\vee\rar&\L*{i+1}\times L^\vee\rar&(\L*{i}\times L^\vee)[1]
    \end{tikzcd}
  \]
  which takes place in the partially wrapped Fukaya category $\WF{n}{d}$. We
  make the following observations:
  \begin{itemize}
  \item By assumption, the morphism $\L{i}\times L\to\L{i+1}\times L$ is of
    degree $0$ (see \Cref{prop:grading_structures:LI}). In particular, by
    \Cref{lemma:Auroux} there exists a grading structure on the Lagrangian
    $\L*{i+1}\times L$ such that the top row of the latter diagram is an exact
    triangle (with all morphisms of degree $0$).
  \item The
  morphisms
  \[
    \begin{tikzcd}[column sep=small]
      (\L{i}\times L)[1]\rar&(\L*{i}\times L^\vee)[1]
    \end{tikzcd}%
    \qquand%
    \begin{tikzcd}[column sep=small]
      \L{i+1}\times L\rar&\L*{i+1}\times L\rar&\L*{i+1}\times L^\vee
    \end{tikzcd}
  \]
  are of degree $0$ by our choice of grading structures on the Lagrangians
  \[
    \L*{I}=\L*{i}\times L^\vee\qquand\L*{J}=\L*{i+1}\times L^\vee.
  \]
  In particular the morphism
  \[
    \morphism{\L*{i+1}\times L\rar}{\L*{i+1}\times L^\vee}
  \]
  is of degree $0$ as well.
\item We conclude that the bottom row in the commutative square
  \[
    \begin{tikzcd}[column sep=small]
      \L{i+1}\times L\rar\dar&(\L{i}\times L)[1]\dar\\
      \L*{i+1}\times L^\vee\rar&(\L*{i}\times L^\vee)[1]
    \end{tikzcd}
  \]
  must be given by a morphism of degree $0$ and, consequently, the morphism
  $\L*{J}\to\L*{I}$ must be a morphism of degree $1=\rk{J}-\rk{I}$.
\end{itemize}
This proves the claim in this special case.

We now return to the general case of two subsets $I,J\in\mychoose{n}{d}$ such
that the inequality ${0\leq j_a-i_a\leq 1}$ is satisfied for all $1\leq a\leq d$.
Write
\[
  J=(I\cap J)\cup\set{j_{a_1},\dots,j_{a_r}}
\]
where $i_{a_t}\in I$ and $j_{a_t}-i_{a_t}=1$; note that $r=\rk{J}-\rk{I}$. The
apparent morphism $\L*{J}\to\L*{I}$ factorises as the composite of the $r$ morphisms
\[
  \begin{tikzcd}
    \displaystyle\prod_{t=1}^s\L*{i_{a_t}}\times\prod_{t=s+1}^r\L*{j_{a_t}}\times\prod_{k\in I\cap J}\L*{k}\rar&\displaystyle\prod_{t=1}^{s-1}\L*{i_{a_t}}\times\prod_{t=s}^{r}\L*{j_{a_t}}\times\prod_{k\in I\cap J}\L*{k},
  \end{tikzcd}
\]
indexed by $r\geq s\geq 1$. Since each of these morphisms has degree $1$ by the
previous argument, the morphism $\L*{J}\to\L*{I}$ has degree $r=\rk{J}-\rk{I}$,
which is what we needed to prove.
\end{proof}

\subsection{The quasi-equivalence $\WF{n}{d}\simeq\perf{\Aus{n}{n-d}}$}

\begin{notation}
  For $I\in\mychoose{n}{d}$ we introduce the (graded) Lagrangian
  \[
    \L{I}^\sharp\coloneqq\L*{I}[-\rk{I}]
  \]
  in $\Sym{d}{\theDisk\setminus\Stops[][n]}$ and define the differential graded $\kk$-algebra
  \[
    \dgAus{n}{d}^\sharp\coloneqq\bigoplus_{J,I}\hom[\L{J}^\sharp][\L{I}^\sharp].\qedhere
  \]
\end{notation}

The following proposition is an immediate consequence of \Cref{thm:Auroux}.

\begin{proposition}
  \label{prop:LIsharp:generate}
  The collection $\set{\L{I}^\sharp}[I\in\mychoose{n}{d}]$ generates the partially
  wrapped Fukaya category $\WF{n}{d}$ as an idempotent-complete triangulated
  $A_\infty$-category. Thus, there exists a quasi-equivalence of triangulated
  $A_\infty$-categories
  \[
    \functor{\perf{\dgAus{n}{d}^\sharp}}{\WF{n}{d},}[\simeq]\qquad\functor*{\Aus{n}{n-d}}{\bigoplus_I\L{I}^\sharp.}\qedhere
  \]
\end{proposition}

In this section we establish the following quasi-equivalence.

\begin{theorem}
  \label{thm:dgAus-sharp:AusShiftedKoszul}
  Let $n\geq d\geq 1$. There is an isomorphism of differential graded
  $\kk$-algebras $\dgAus{n}{d}^\sharp\cong\Aus{n}{n-d}$. Thus, there is a
  quasi-equivalence of triangulated $A_\infty$-categories
  \[
    \functor{\perf{\Aus{n}{n-d}}}{\WF{n}{d},}[\simeq]\qquad\functor*{\Aus{n}{n-d}}{\bigoplus_I\L{I}^\sharp.}\qedhere
  \]
\end{theorem}

\begin{lemma}
  \label{lemma:Aus-sharp:deg0}
  The graded $\kk$-algebra $\dgAus{n}{d}^\sharp$ is concentrated in degree $0$.
\end{lemma}
\begin{proof}
  Let $I,J\in\mychoose{n}{d}$. By \Cref{lemma:homIJ:rk}, the graded $\kk$-module
  $\hom[\L*{J}][\L*{I}]$ either vanishes or is concentrated in degree
  $\rk{J}-\rk{I}$. Consequently, the graded $\kk$-module
  \begin{align*}
    \hom[\L{J}^\sharp][\L{I}^\sharp] &=\hom[\L*{J}[-\rk{J}]][\L*{I}[-\rk{I}]]\\
                                     &\cong\hom[\L*{J}][\L*{I}[\rk{J}-\rk{I}]]\\
                                     &\cong\hom[\L*{J}][\L*{I}](\rk{J}-\rk{I})
  \end{align*}
  is concentrated in degree $0$. The claim follows.
\end{proof}

The proof of the following combinatorial statement is straightforward and is
left to the reader.

\begin{lemma}
  \label{lemma:veers-coveers}
  Let $n>d\geq 1$ and $I,J\in\mychoose{n}{d}$. We let
  \[
    I^\circ=\n\setminus I=\set{u_1<\cdots<u_{n-d}}\qquand J^\circ=\n\setminus
    J=\set{v_1<\cdots<v_{n-d}}.
  \]
  The inequalities $0\leq j_a-i_a\leq 1$ are satisfied for all $1\leq a\leq d$
  if and only if $J^\circ\leq I^\circ$ and $v_b<u_{b+1}$ for all $1\leq
  b<n-d$.
\end{lemma}

We are ready to prove \Cref{thm:dgAus-sharp:AusShiftedKoszul}.

\begin{proof}[Proof of \Cref{thm:dgAus-sharp:AusShiftedKoszul}]
  Firstly, by \Cref{lemma:homIJ:rk} and \Cref{lemma:Aus-sharp:deg0} there are
  isomorphisms of $\kk$-modules
  \[
    \hom[\L{J}^\sharp][\L{I}^\sharp]\cong\begin{cases}
      \kk\cdot g_{IJ}&\text{if }\forall\ a:0\leq j_a-i_a\leq 1,\\
      0&\text{otherwise}.
    \end{cases}
  \]
  Let
  \[
    I^\circ=\n\setminus I=\set{u_1<\cdots<u_{n-d}}\qquand J^\circ=\n\setminus
    J=\set{v_1<\cdots<v_{n-d}}.
  \]
  By \Cref{lemma:veers-coveers} the condition that $0\leq j_a-i_a\leq 1$ for all
  $1\leq a\leq d$ is equivalent to the condition that $J^\circ\leq I^\circ$ and
  $v_b<u_{b+1}$ for all $1\leq b<n-d$. Keeping in mind the poset
  anti-isomorphism $\mychoose{n}{d}\to\mychoose{n}{n-d}$, $I\mapsto I^\circ$, we
  see that the underlying $\kk$-modules of the $\kk$-algebras
  $\dgAus{n}{d}^\sharp$ and $\Aus{n}{n-d}$ can be identified. Comparing the
  multiplication laws on both sides, we conclude that there is an isomorphism of
  $\kk$-algebras
  \[
    \dgAus{n}{d}^\sharp=\bigoplus_{J,I}\hom[\L{J}^\sharp][\L{I}^\sharp]\cong\Aus{n}{n-d}.
  \]
  The existence of the required quasi-equivalence now follows from
  \Cref{prop:LIsharp:generate}.
\end{proof}

\subsection{The Serre functor and Iyama's cluster tilting subcategory of $\WF{n}{d}$}

In this section we give a simple geometric description of the Serre functor on
the partially wrapped Fukaya category $\WF{n}{d}$ and use it to describe a
distinguished subcategory first investigated by Iyama in the context of higher
Auslander--Reiten theory. Throughout this subsection we assume that $\kk$ is a
field.

\subsubsection{Geometric description of the Serre functor on $\WF{n}{d}$}

For definiteness, let $\theDisk$ be the $2$-dimensional unit disk and
$\Stops[][n]$ the set of $(n+1)$-st roots of unity. Observe that there is a
symplectomorphism $\theDisk\to\theDisk$, given by counter-clockwise rotation by
$\frac{2\pi}{n+1}$, which cyclically permutes the set $\Stops[][n]$ of stops.
Passing to symmetric products we obtain a symplectomorphism
\[
  \morphism[r]{\Sym{d}{\theDisk}}{\Sym{d}{\theDisk}}[\cong]
\]
which preserves $\Stops[d][n]$. To extract an autoequivalence of $\WF{n}{d}$ from $r$ we need to
lift it to a \emph{graded} symplectomorphism; since $H^1(\Sym{d}{\theDisk})=0$, there is no
obstruction to the existence of such a graded lift and all possible graded lifts of $r$ form
$\ZZ$-torsor (see Lemma~2.4 in \cite{Sei00}). Each graded lift of $r$ induces an autoequivalence
\[
  \functor[\mathfrak{r}]{\WF{n}{d}}{\WF{n}{d};}[\simeq]
\]
different choices of graded lifts of $r$ induce autoequivalences that differ
only by a power of the shift functor on $\WF{n}{d}$. The next proposition shows
that a particular graded lift of $r$ gives rise to an autoequivalence
$\mathfrak{r}$ of $\WF{n}{d}$ that agrees with its Serre functor on $\WF{n}{d}$.

\begin{proposition}
  \label{prop:Serre}
  Let $\kk$ be a field. Let $\theDisk$ be the $2$-dimensional unit disk and
  $\Stops[][n]$ the set of $(n+1)$-st roots of unity. Let
  $r\colon\theDisk\to\theDisk$ by the symplectomorphism given by
  counter-clockwise rotation by $\frac{2\pi}{n+1}$. There exists a graded lift
  of the induced symplectomorphism
  \[
    \morphism[r]{\Sym{d}{\theDisk}}{\Sym{d}{\theDisk}}[\cong]
  \]
  such that the induced diagram of $A_\infty$-functors
  \[
    \begin{tikzcd}
      \WF{n}{d}\rar{\simeq}\dar[dotted]{\mathfrak{r}}&\perf{\Aus{n}{d}}\dar{\Serre}\\
      \WF{n}{d}\rar{\simeq}&\perf{\Aus{n}{d}}
    \end{tikzcd}
  \]
  commutes, where
  \[
    \functor[\Serre=-\otimes_{\Aus{n}{d}}^{\mathbb{L}}\Hom[\Aus{n}{d}][\kk][\kk]]{\perf{\Aus{n}{d}}}{\perf{\Aus{n}{d}}}[\simeq]
  \]
  is the Serre functor on $\perf{\Aus{n}{d}}$ and both horizontal
  quasi-equivalences are the quasi-equivalence from \Cref{thm:WF-Auslander}.
\end{proposition}
\begin{proof}
  
  Let $I\in\mychoose{n+1}{d+1}$ be such that $0\in I$ and $f_{II}$ the
  corresponding idempotent in $\Aus{n}{d}$. The image of the right module
  $f_{II}\Aus{n}{d}$ under the composite
  \[
    \begin{tikzcd}
      \perf{\Aus{n}{d}}\rar{\simeq}&\WF{n}{d}\rar{\mathfrak{r}}&\WF{n}{d}
    \end{tikzcd}
  \]
  is the object
  \[
    \mathfrak{r}\L{I}=\prod_{a=1}^drL_{0,i_a}=\prod_{a=1}^dL_{i_a-1,n}
\]
Iterated application of \Cref{lemma:Auroux} yields a quasi-isomorphism $\mathfrak{r}\L{I}\simeq\L{rI}$ where
\[
  rI=\set{i_1-1,\dots,i_d-1,n}\in\mychoose{n+1}{d+1}.
\]
It is straightforward to verify that the standard resolution of the
object $\L{rI}$ provided by \Cref{prop:resolutions} corresponds to the minimal
projective resolution of the injective right $\Aus{n}{d}$-module
$D(\Aus{n}{d}f_{II})$ as described for example in Proposition 3.17 in
\cite{OT12} or Proposition 2.7 in \cite{JK19a}. It follows that the restriction
of the composite
\[
  \begin{tikzcd}
    \perf{\Aus{n}{d}}\rar{\simeq}&\WF{n}{d}\rar{\mathfrak{r}}&\WF{n}{d}\rar{\simeq}&\perf{\Aus{n}{d}}
  \end{tikzcd}
\]
to the full subcategory of $\perf{\Aus{n}{d}}$ spanned by the regular
representation $\Aus{n}{d}$ agrees with the Serre functor on $\perf{\Aus{n}{d}}$
up to a power of the shift, corresponding to a choice of a graded lift of the
symplectomorphism $r\colon\Sym{d}{\theDisk}\to\Sym{d}{\theDisk}$. The claim
follows.
\end{proof}

\begin{remark}
  The geometric description of the Serre functor on $\WF{n}{d}$ given in
  \Cref{prop:Serre} makes it apparent that its
  $(n+1)$-iteration must be a power of the shift functor; indeed, the
  symplectomorphism $r^{n+1}\colon\theDisk\to\theDisk$ is the identity. More
  precisely, there is a quasi-isomorphism
  \[
    \Serre^{n+1}\simeq[d(n-d)]
  \]
  of functors $\WF{n}{d}\to \WF{n}{d}$ expressing the known
  \emph{fractionally} Calabi--Yau property of the quasi-equivalent category
  $\perf{\Aus{n}{d}}$, see \cite{HI11a} for details. Note that the above power
  of the shift is invariant under the passage $d\mapsto n-d$, in agreement with
  the quasi-equivalence $\WF{n}{d}\simeq\WF{n}{n-d}$ from
  \Cref{thm:WF-Auslander:Koszul}.
\end{remark}

\subsubsection{Geometric description of Iyama's $d$-cluster-tilting subcategory
  of $\WF{n}{d}$}

\begin{notation}
  For an idempotent complete $A_\infty$-category $\A$ admitting finite direct
  sums and a collection $\X$ of objects of $\A$, we let $\operatorname{add}\X$
  be the smallest full subcategory of $\A$ which is idempotent complete, is
  closed under finite direct sums, and contains $\X$.
\end{notation}

As explained in the introduction, one of the main results
in \cite{Iya11}, Theorem~1.18 therein, shows that the higher Auslander algebra
$\Aus{n}{d}$ is a $d$-Auslander algebra in the sense that it satisfies the
inequalities
\[
  \gldim\Aus{n}{d}\leq d\leq\domdim\Aus{n}{d}.
\]
As a consequence of the Auslander--Iyama correspondence~\cite{Iya07a} and
Theorem~1.23 in \cite{Iya11}, the subcategory
\[
  \U(\Aus{n}{d})\coloneqq\operatorname{add}\set{\mathfrak{\Serre}_d^{k}(\Aus{n}{d})}[k\in\ZZ]
\]
is a so-called \emph{$d\ZZ$-cluster tilting subcategory} \cite{Iya07,IJ17} of
the triangulated category $H^0(\perf{\Aus{n}{d}})$, where
\[
  \functor[\Serre=-\otimes_{\Aus{n}{d}}^\mathbb{L}\Hom[\Aus{n}{d}][\kk][\kk]]%
  {\perf{\Aus{n}{d}}}{\perf{\Aus{n}{d}}}[\simeq]
\]
is the Serre functor and $\Serre[d]=\Serre{[-d]}$. By definition, this means
that a perfect complex $X$ lies in $\U(\Aus{n}{d})$ if and
only if
\[
  \forall k\not\in d\ZZ\text{ and }\forall M\in\U(\Aus{n}{d}):\Hom[X][M[k]]=0
\]
if and only if
\[
  \forall k\not\in d\ZZ\text{ and }\forall M\in\U(\Aus{n}{d}):\Hom[M][X[k]]=0.
\]
This kind of subcategory plays an important role in higher Auslander--Reiten
theory and in a higher-dimensional version of homological algebra. For example,
as shown by Gei\ss, Keller and Oppermann \cite{GKO13}, the additive category
$H^0(\U(\Aus{n}{d}))$, equipped with the $d$-fold shift functor $[d]$, has the
structure of a $(d+2)$-angulated category. Similarly. the subcategory
\[
  \M(\Aus{n}{d})\coloneqq\U(\Aus{n}{d})\cap\mmod{\Aus{n}{d}}\subseteq\mmod{\Aus{n}{d}}
\]
is a $d\ZZ$-cluster tilting subcategory of the abelian category
$\mmod{\Aus{n}{d}}$ of finite-dimensional $\Aus{n}{d}$-modules, which implies
that $\M(\Aus{n}{d})$ is a $d$-abelian category in the sense of \cite{Jas16}. In
some sense, the homological properties of the above subcategories can be thought of
as `witnesses' of the higher-dimensional nature of the $\kk$-algebra
$\Aus{n}{d}$.

Under the quasi-equivalence $\perf{\Aus{n}{d}}\simeq\WF{n}{d}$ from
\Cref{thm:WF-Auslander}, the subcategory $\U(\Aus{n}{d})$ of $\perf{\Aus{n}{d}}$
induces a $d\ZZ$-cluster tilting subcategory of the partially wrapped Fukaya
category $\WF{n}{d}$ which we describe below.

\begin{proposition}
  \label{prop:d-CT}
  Let $\kk$ be a field. Fix grading structures on the Lagrangians
  $\set{\L{I}}[I\in\mychoose{n}{d}]$ as in \Cref{prop:grading_structures:LI}.
  Fix the grading structure on the symplectomorphism
  $r\colon\Sym{d}{\theDisk}\to\Sym{d}{\theDisk}$ from \Cref{prop:Serre}, so that
  the induced autoequivalence
  \[
    \functor[\mathfrak{r}]{\WF{n}{d}}{\WF{n}{d}}[\simeq]
  \]
  is the Serre functor, and define $\mathfrak{r}_d\coloneqq\mathfrak{r}[-d]$.
  The full subcategory
  \[
    \UF{n}{d}\coloneqq\operatorname{add}\set{\mathfrak{r}_d^{k}(\L{I})}[I\in\mychoose{n}{d},\
    k\in\ZZ]
  \]
  of $\WF{n}{d}$ is a $d\ZZ$-cluster tilting subcategory of $H^0(\WF{n}{d})$.
\end{proposition}
\begin{proof}
  Note that $\Aus{n}{d}\mapsto\bigoplus_I \L{I}$ under the quasi-equivalence
  $\perf{\Aus{n}{d}}\simeq\WF{n}{d}$ from \Cref{thm:WF-Auslander}. Moreover,
  \Cref{prop:Serre} implies the existence of a commutative diagram
  \[
    \begin{tikzcd}
      \perf{\Aus{n}{d}}\rar{\simeq}\dar{\Serre[d]}&\WF{n}{d}\dar{\mathfrak{r}_d}\\
      \perf{\Aus{n}{d}}\rar{\simeq}&\WF{n}{d}
    \end{tikzcd}
  \]
  Therefore $\U(\Aus{d}n)$ corresponds to $\UF{n}{d}$ under the
  quasi-equivalence ${\perf{\Aus{n}{d}}\simeq\WF{n}{d}}$. The claim follows.
\end{proof}

\begin{remark}
  The indecomposable objects in the $d\ZZ$-cluster tilting subcategory
  $\UF{n}{d}$ of $\WF{n}{d}$ are represented by the $d$-fold shifts of the
  Lagrangians $\set{\L{I}}[I\in\mychoose{n_+}{d+1}]$ (equipped with grading
  structures as in \Cref{prop:resolutions}). Since $\UF{n}{d}$ is a
  $d\ZZ$-cluster tilting subcategory of the partially wrapped Fukaya category
  $\WF{n}{d}$, an object $X$ lies in the $d\ZZ$-cluster tilting subcategory
  $\UF{n}{d}$ if and only if
  \[
    \forall k\not\in d\ZZ\text{ and }\forall L\in\UF{n}{d}:\Hom[X][L[k]]=0
  \]
  if and only if
  \[
    \forall k\not\in d\ZZ\text{ and }\forall L\in\UF{n}{d}:\Hom[L][X[k]]=0.
  \]
  In particular, given $I,J\in\mychoose{n+1}{d+1}$, the extension space
  $\Hom[\L{I}][\L{J}[k]]$ vanishes for all $k\not\in d\ZZ$.
\end{remark}

\begin{remark}
  Oppermann and Thomas \cite{OT12} provide a beautiful classification of the
  Miyashita tilting modules \cite{Miy86} which belong to the $d$-cluster tilting
  subcategory
  \[
    \M(\Aus{n}{d})=\U(\Aus{n}{d})\cap\mmod{\Aus{n}{d}}
  \]
  of $\mmod{\Aus{n}{d}}$ in terms of triangulations of a $2d$-dimensional cyclic
  polytope with $n+d+1$ vertices (we remind the reader that $\Aus{n}{d}$ is
  associated with the Dynkin type $\AA_{n-d+1}$); under this bijection, the
  so-called `mutation' of tilting modules corresponds to the bistellar flip of
  triangulations (see \cite{Ram97} or Section 6.1 in \cite{DLRS10} for further
  information on cyclic polytopes and their triangulations). Under the
  quasi-equivalence $\WF{n}{d}\simeq\perf{\Aus{n}{d}}$ from
  \Cref{thm:WF-Auslander}, this representation-theoretic procedure provides
  numerous generators for the partially wrapped Fukaya category $\WF{n}{d}$
  whose derived endomorphism algebra has its cohomology concentrated in degree
  $0$.
\end{remark}

\begin{remark}
  In the above discussion it is essential that we work over a field rather than
  over an arbitrary commutative ring. Recall that, for each $d\geq 1$, there is
  a quasi-equivalence
  \[
    \perf{\kk}\simeq\WF{d}{d}.
  \]
  If $\kk$ is a field, the triangulated category $H^0(\perf{\kk})$ has a
  $d\ZZ$-cluster tilting subcategory
  \[
    \operatorname{add}\set{\kk[-dk]}[k\in\ZZ]\subseteq\perf{\kk}
  \]
  for each $d\geq 1$, where we use that the Serre functor on $\perf{\kk}$ is the
  identity functor. In contrast, if $\kk=\ZZ$ it is easy to see that
  $H^0(\perf{\ZZ})$ does not admit a $d\ZZ$-cluster tilting subcategory for
  $d>1$ (a $1\ZZ$-cluster tilting subcategory always exists
  and coincides with the ambient category).
\end{remark}

\subsection{Examples}

We conclude this section with some examples to illustrate our results.

\subsubsection{The quasi-equivalence $\WF{3}{1}\simeq\WF{3}{2}$}

For simplicity, we let $\kk$ be a field. We illustrate the quasi-equivalence
$\WF{n}{d}\simeq\WF{n}{n-d}$ in the simplest non-trivial case $n=3$ and $d=1$.

The following diagram depicts the so-called Auslander--Reiten quiver
\cite{Hap88} of the triangulated category $H^0(\WF{3}{1})$ where, for
simplicity, we write $ij$ in place of the Lagrangian $L_{ij}$:
\begin{center}
  \begin{tikzpicture}[rotate=-45,scale=1.5]
    \node (m3m1) at (-2.75,-0.75) {$\cdots$};%
    \node (m30) at (-3,0) {${01}[-1]$};%

    \node (m2m1) at (-2,-1) {${23}[-2]$};%
    \node (m20) at (-2,0) {${02}[-1]$};%
    \node (m21) at (-2,1) {${12}[-1]$};%

    \node (m10) at (-1,0) {${03}[-1]$};%
    \node (m11) at (-1,1) {${13}[-1]$};%
    \node (m12) at (-1,2) {${23}[-1]$};%

    \node[draw,rectangle,inner sep=2] (01) at (0,1) {${01}$};%
    \node[draw,rectangle,inner sep=2] (02) at (0,2) {${02}$};%
    \node[draw,rectangle,inner sep=2] (03) at (0,3) {${03}$};%

    \node[draw,rectangle,inner sep=2] (12) at (1,2) {${12}$};%
    \node[draw,rectangle,inner sep=2] (13) at (1,3) {${13}$};%
    \node (14) at (1,4) {${01}[1]$};%

    \node[draw,rectangle,inner sep=2] (23) at (2,3) {${23}$};%
    \node (24) at (2,4) {${02}[1]$};%
    \node (25) at (2,5) {${12}[1]$};%
    
    \node (34) at (3,4) {${03}[1]$};%
    \node (35) at (3,5) {${13}[1]$};%
    \node (36) at (3,6) {${23}[1]$};%

    \node (45) at (4,5) {${01}[2]$};%
    \node (46) at (3.75,5.75) {$\cdots$};%

    \path[commutative diagrams/.cd,every arrow] (m2m1)--(m20);%
    \path[commutative diagrams/.cd,every arrow] (m20)--(m21);%
    
    \path[commutative diagrams/.cd,every arrow] (m10)--(m11);%
    \path[commutative diagrams/.cd,every arrow] (m11)--(m12);%

    \path[commutative diagrams/.cd,every arrow] (01)--(02);%
    \path[commutative diagrams/.cd,every arrow] (02)--(03);%

    \path[commutative diagrams/.cd,every arrow] (12)--(13);%
    \path[commutative diagrams/.cd,every arrow] (13)--(14);%

    \path[commutative diagrams/.cd,every arrow] (23)--(24);%
    \path[commutative diagrams/.cd,every arrow] (24)--(25);%

    \path[commutative diagrams/.cd,every arrow] (34)--(35);%
    \path[commutative diagrams/.cd,every arrow] (35)--(36);%

    \path[commutative diagrams/.cd,every arrow] (m30)--(m20);%
    \path[commutative diagrams/.cd,every arrow] (m20)--(m10);%

    \path[commutative diagrams/.cd,every arrow] (m21)--(m11);%
    \path[commutative diagrams/.cd,every arrow] (m11)--(01);%

    \path[commutative diagrams/.cd,every arrow] (m12)--(02);%
    \path[commutative diagrams/.cd,every arrow] (02)--(12);%

    \path[commutative diagrams/.cd,every arrow] (03)--(13);%
    \path[commutative diagrams/.cd,every arrow] (13)--(23);%

    \path[commutative diagrams/.cd,every arrow] (14)--(24);%
    \path[commutative diagrams/.cd,every arrow] (24)--(34);%

    \path[commutative diagrams/.cd,every arrow] (25)--(35);%
    \path[commutative diagrams/.cd,every arrow] (35)--(45);%

    \path[dotted] (m2m1)--(m10);%

    \draw[dotted](m2m1)--(m10)--(01)--(12)--(23)--(34)--(45);%
    \draw[dotted](m3m1)--(m20)--(m11)--(02)--(13)--(24)--(35)--(46);%
    \draw[dotted](m30)--(m21)--(m12)--(03)--(14)--(25)--(36);%
  \end{tikzpicture}
\end{center}
The vertices of quiver correspond to the indecomposable objects in
$H^0(\WF{3}{1})$; the arrows correspond to a $\kk$-basis of the space of
irreducible morphisms, that is the non-isomorphisms which cannot be expressed as
a non-trivial composite of non-isomorphisms. The dotted lines indicate the
apparent commutativity and zero relations. We have chosen grading structures on
the above Lagrangians so that all depicted morphisms have degree $0$. The action
of the derived Auslander--Reiten translation $\Serre[1]=\Serre{[-1]}$ is given
by left horizontal translation. Finally, we have encircled the six
indecomposable objects which belong to the heart of the $t$-structure induced by
the quasi-equivalence $\WF{3}{1}\simeq\perf{\Aus{3}{1}}$; these objects form a
complete set of representative of the isomorphism classes of indecomposable
objects in $\WF{3}{1}$ up to the action of the shift functor.

We remind the reader that, for $0\leq i<j<k\leq n$, there is an equivalence
$L_{ij}\times L_{ik}\simeq L_{ij}\times L_{jk}$ in $\WF{n}{2}$ induced by the
exact triangle
\[
  \begin{tikzcd}[column sep=small]
    L_{ij}\rar&L_{ik}\rar&L_{jk}\rar&L_{ij}[1]
  \end{tikzcd}
\]
in the partially wrapped Fukaya category $\WF{n}{1}$, see \Cref{lemma:Auroux}.
The following diagram depicts the Auslander--Reiten quiver of the triangulated
category $H^0(\WF{3}{2})$ where, for simplicity, we write $ij,k\ell$ in place of
the Lagrangian $L_{ij}\times L_{k\ell}$:
\begin{center}
  \begin{tikzpicture}[rotate=-45,scale=1.5]
    \node (m3m1) at (-2.75,-0.75) {$\cdots$};%
    \node (m30) at (-3,0) {$(02,23)[-1]$};%

    \node[draw,rectangle,inner sep=2,fill=gray,fill opacity=0.25,text opacity=1] (m2m1) at (-2,-1) {$01,12$};%
    \node (m20) at (-2,0) {$(03,12)[-1]$};%
    \node[draw,rectangle,inner sep=2,fill=gray,fill opacity=0.25,text opacity=1] (m21) at (-2,1) {$01,13$};%

    \node (m10) at (-1,0) {$(12,23)[-1]$};%
    \node[draw,rectangle,inner sep=2] (m11) at (-1,1) {$01,23$};%
    \node (m12) at (-1,2) {$(01,12)[1]$};%

    \node[draw,rectangle,inner sep=2,fill=gray,fill opacity=0.25,text opacity=1] (01) at (0,1) {${02,23}$};%
    \node (02) at (0,2) {$03,12$};%
    \node[draw,rectangle,inner sep=2,fill=gray,fill opacity=0.25,text opacity=1] (03) at (0,3) {$12,23$};%

    \node (12) at (1,2) {$(01,13)[1]$};%
    \node (13) at (1,3) {$(01,23)[1]$};%
    \node (14) at (1,4) {$(02,23)[1]$};%

    \node[rectangle,inner sep=2,fill=gray,fill opacity=0.25,text opacity=1] (23) at (2,3) {$(01,12)[2]$};%
    \node (24) at (2,4) {$(03,12)[1]$};%
    \node[rectangle,inner sep=2,fill=gray,fill opacity=0.25,text opacity=1] (25) at (2,5) {$(01,13)[2]$};%
    
    \node (34) at (3,4) {$(12,23)[1]$};%
    \node (35) at (3,5) {$(01,23)[2]$};%
    \node (36) at (3,6) {$(01,12)[3]$};%

    \node[rectangle,inner sep=2,fill=gray,fill opacity=0.25,text opacity=1] (45) at (4,5) {$(02,23)[2]$};%
    \node (46) at (3.75,5.75) {$\cdots$};%

    \path[commutative diagrams/.cd,every arrow](m2m1)--(m20);%
    \path[commutative diagrams/.cd,every arrow](m20)--(m21);%
    
    \path[commutative diagrams/.cd,every arrow](m10)--(m11);%
    \path[commutative diagrams/.cd,every arrow](m11)--(m12);%

    \path[commutative diagrams/.cd,every arrow](01)--(02);%
    \path[commutative diagrams/.cd,every arrow](02)--(03);%

    \path[commutative diagrams/.cd,every arrow](12)--(13);%
    \path[commutative diagrams/.cd,every arrow](13)--(14);%

    \path[commutative diagrams/.cd,every arrow](23)--(24);%
    \path[commutative diagrams/.cd,every arrow](24)--(25);%

    \path[commutative diagrams/.cd,every arrow](34)--(35);%
    \path[commutative diagrams/.cd,every arrow](35)--(36);%

    \path[commutative diagrams/.cd,every arrow](m30)--(m20);%
    \path[commutative diagrams/.cd,every arrow](m20)--(m10);%

    \path[commutative diagrams/.cd,every arrow](m21)--(m11);%
    \path[commutative diagrams/.cd,every arrow](m11)--(01);%

    \path[commutative diagrams/.cd,every arrow](m12)--(02);%
    \path[commutative diagrams/.cd,every arrow](02)--(12);%

    \path[commutative diagrams/.cd,every arrow](03)--(13);%
    \path[commutative diagrams/.cd,every arrow](13)--(23);%

    \path[commutative diagrams/.cd,every arrow](14)--(24);%
    \path[commutative diagrams/.cd,every arrow](24)--(34);%

    \path[commutative diagrams/.cd,every arrow](25)--(35);%
    \path[commutative diagrams/.cd,every arrow](35)--(45);%

    \path[dotted](m2m1)--(m10);%

    \draw[dotted](m2m1)--(m10)--(01)--(12)--(23)--(34)--(45);%
    \draw[dotted](m3m1)--(m20)--(m11)--(02)--(13)--(24)--(35)--(46);%
    \draw[dotted](m30)--(m21)--(m12)--(03)--(14)--(25)--(36);%
  \end{tikzpicture}
\end{center}
Again, we have chosen grading structures on the above Lagrangians so that all
depicted morphisms have degree $0$. Finally, we have encircled the five
indecomposable objects which belong to the heart of the $t$-structure induced by
the quasi-equivalence $\WF{3}{2}\simeq\perf{\Aus{3}{2}}$; together with the
Lagrangian $L_{03}\times L_{12}$, these objects form a complete set of
representative of the isomorphism classes of indecomposable objects in
$\WF{3}{1}$ up to the action of the shift functor. The indecomposable objects in
the $2\ZZ$-cluster tilting subcategory $\U_3^{(2)}$ of the triangulated category
$H^0(\WF{3}{2})$ are highlighted; notice that this subcategory is spanned by the
(finite direct sums of) even shifts of the four highlighted objects in the
heart.

\subsubsection{The $2\ZZ$-cluster tilting subcategory of $H^0(\WF{4}{2})$}

The diagram below depicts a complete set of representatives of the isomorphism
classes of indecomposable objects in the $2\ZZ$-cluster tilting subcategory
$\UF{4}{2}$ of the triangulated category $H^0(\WF{4}{2})$, see \Cref{prop:d-CT}
(all other objects in this subcategory are obtained as even shifts of those
below):
\begin{center}
  \tdplotsetmaincoords{-20}{0}%
  \begin{tikzpicture}[scale=3,tdplot_main_coords,every node/.style={scale=0.75}]
    \tdplotsetrotatedcoords{0}{60}{180}
    \begin{scope}[tdplot_rotated_coords]
      \node[draw,rectangle] (123) at (1,2,3) {$\L{01}\times\L{02}$};%
      \node[draw,rectangle] (124) at (1,2,4) {$\L{01}\times\L{03}$};%
      \node[draw,rectangle] (125) at (1,2,5) {$\L{01}\times\L{04}$};%

      \node (133) at (1,3,3) {$0$};%
      \node[draw,rectangle] (134) at (1,3,4) {$\L{02}\times\L{03}$};%
      \node[draw,rectangle] (135) at (1,3,5) {$\L{02}\times\L{04}$};%

      \node (144) at (1,4,4) {$0$};%
      \node[draw,rectangle] (145) at (1,4,5) {$\L{03}\times\L{04}$};%

      \node (223) at (2,2,3) {$0$};%
      \node (224) at (2,2,4) {$0$};%
      \node (225) at (2,2,5) {$0$};%

      \node (233) at (2,3,3) {$0$};%
      \node (234) at (2,3,4) {$\L{12}\times\L{13}$};%
      \node (235) at (2,3,5) {$\L{12}\times\L{14}$};%

      \node (244) at (2,4,4) {$0$};%
      \node (245) at (2,4,5) {$\L{13}\times\L{14}$};%

      \node (334) at (3,3,4) {$0$};%
      \node (344) at (3,4,4) {$0$};%
      \node (335) at (3,3,5) {$0$};%

      \node (345) at (3,4,5) {$\L{23}\times\L{24}$};%
    \end{scope}

    \path[commutative diagrams/.cd,every arrow](123)--(124);%
    \path[commutative diagrams/.cd,every arrow](123)--(133);%
    \path[commutative diagrams/.cd,every arrow](124)--(125);%
    \path[commutative diagrams/.cd,every arrow](124)--(134);%
    \path[commutative diagrams/.cd,every arrow](125)--(135);%
    \path[commutative diagrams/.cd,every arrow](133)--(134);%
    \path[commutative diagrams/.cd,every arrow](134)--(135);%
    \path[commutative diagrams/.cd,every arrow](134)--(144);%
    \path[commutative diagrams/.cd,every arrow](135)--(145);%
    \path[commutative diagrams/.cd,every arrow](144)--(145);%

    \path[commutative diagrams/.cd,every arrow](124)--(224);%
    \path[commutative diagrams/.cd,every arrow](125)--(225);%
    \path[commutative diagrams/.cd,every arrow](134)--(234);%
    \path[commutative diagrams/.cd,every arrow](135)--(235);%
    \path[commutative diagrams/.cd,every arrow](144)--(244);%
    \path[commutative diagrams/.cd,every arrow](145)--(245);%
    \path[commutative diagrams/.cd,every arrow](245)--(345);%

    \path[commutative diagrams/.cd,every arrow](224)--(234);%
    \path[commutative diagrams/.cd,every arrow](225)--(235);%
    \path[commutative diagrams/.cd,every arrow](234)--(244);%
    \path[commutative diagrams/.cd,every arrow](235)--(245);%

    \path[commutative diagrams/.cd,every arrow](224)--(225);%
    \path[commutative diagrams/.cd,every arrow](234)--(235);%
    \path[commutative diagrams/.cd,every arrow](244)--(245);%

    \path[commutative diagrams/.cd,every arrow](235)--(335);%
    \path[commutative diagrams/.cd,every arrow](335)--(345);%

    \path[commutative diagrams/.cd,every arrow](123)--(223);%
    \path[commutative diagrams/.cd,every arrow](133)--(233);%
    \path[commutative diagrams/.cd,every arrow](223)--(233);%

    \path[commutative diagrams/.cd,every arrow](223)--(224);%
    \path[commutative diagrams/.cd,every arrow](233)--(234);%

    \path[commutative diagrams/.cd,every arrow](234)--(334);%
    \path[commutative diagrams/.cd,every arrow](334)--(344);%
    \path[commutative diagrams/.cd,every arrow](334)--(335);%

    \path[commutative diagrams/.cd,every arrow](244)--(344);%
    \path[commutative diagrams/.cd,every arrow](344)--(345);%
  \end{tikzpicture}
\end{center}
The generators $\set{\L{I}}$ of the partially wrapped Fukaya category
$\WF{4}{2}$ are encircled. In the above diagram, all squares commute; note,
however, that in the $A_\infty$-category $\WF{4}{2}$ there are further higher
operations which witness the fact the apparent rectilinear cubes are
bicartesian.

\subsubsection{Explicit examples in the case $n=5$}

We conclude this section by displaying the endomorphism algebras
\[
  H^*(\dgAus{n}{d})\cong\Aus{n}{d}\qquand\dgAus{n}{d}^\sharp\cong\Aus{n}{n-d}
\]
in the case $n=5$ for all $1\leq d<n$.

We begin with the $\kk$-algebras $H^*(\dgAus{n}{d})\cong\Aus{n}{d}$. For $d=1$
we obtain the $\kk$-algebra with generators
\[
  \begin{tikzcd}[column sep=small]
    L_{1}\rar&L_{2}\rar&L_{3}\rar&L_{4}\rar&L_{5}
  \end{tikzcd}
\]
with no relations between the above morphisms. For $d=2,3$ we obtain the
$\kk$-algebras with generators
\[
  \begin{tikzcd}[column sep=small]
    L_{1}\times L_{2}\rar\dar&L_{1}\times L_{3}\rar\dar&L_{1}\times
    L_{4}\rar\dar&L_{1}\times L_{5}\dar\\
    0\rar&L_{2}\times L_{3}\rar\dar&L_{2}\times L_{4}\rar\dar&L_{2}\times
    L_{5}\dar\\
    &0\rar&L_{3}\times L_{4}\rar\dar&L_{3}\times L_{5}\dar\\
    &&0\rar&L_{4}\times L_{5}
  \end{tikzcd}
\]
and
\begin{center}
  \tdplotsetmaincoords{-20}{0}%
  \begin{tikzpicture}[scale=3,tdplot_main_coords,every node/.style={scale=0.75}]
    \tdplotsetrotatedcoords{0}{60}{180}
    \begin{scope}[tdplot_rotated_coords]
      \node (123) at (1,2,3) {$\L{1}\times\L{2}\times\L{3}$};%
      \node (124) at (1,2,4) {$\L{1}\times\L{2}\times\L{4}$};%
      \node (125) at (1,2,5) {$\L{1}\times\L{2}\times\L{5}$};%

      \node (133) at (1,3,3) {$0$};%
      \node (134) at (1,3,4) {$\L{1}\times\L{3}\times\L{4}$};%
      \node (135) at (1,3,5) {$\L{1}\times\L{3}\times\L{5}$};%

      \node (144) at (1,4,4) {$0$};%
      \node (145) at (1,4,5) {$\L{1}\times\L{4}\times\L{5}$};%

      \node (224) at (2,2,4) {$0$};%
      \node (225) at (2,2,5) {$0$};%

      \node (234) at (2,3,4) {$\L{2}\times\L{3}\times\L{4}$};%
      \node (235) at (2,3,5) {$\L{2}\times\L{3}\times\L{5}$};%

      \node (244) at (2,4,4) {$0$};%
      \node (245) at (2,4,5) {$\L{2}\times\L{4}\times\L{5}$};%

      \node (335) at (3,3,5) {$0$};%

      \node (345) at (3,4,5) {$\L{3}\times\L{4}\times\L{5}$};%
    \end{scope}

    \path[commutative diagrams/.cd,every arrow](123)--(124);%
    \path[commutative diagrams/.cd,every arrow](123)--(133);%
    \path[commutative diagrams/.cd,every arrow](124)--(125);%
    \path[commutative diagrams/.cd,every arrow](124)--(134);%
    \path[commutative diagrams/.cd,every arrow](125)--(135);%
    \path[commutative diagrams/.cd,every arrow](133)--(134);%
    \path[commutative diagrams/.cd,every arrow](134)--(135);%
    \path[commutative diagrams/.cd,every arrow](134)--(144);%
    \path[commutative diagrams/.cd,every arrow](135)--(145);%
    \path[commutative diagrams/.cd,every arrow](144)--(145);%

    \path[commutative diagrams/.cd,every arrow](124)--(224);%
    \path[commutative diagrams/.cd,every arrow](125)--(225);%
    \path[commutative diagrams/.cd,every arrow](134)--(234);%
    \path[commutative diagrams/.cd,every arrow](135)--(235);%
    \path[commutative diagrams/.cd,every arrow](144)--(244);%
    \path[commutative diagrams/.cd,every arrow](145)--(245);%
    \path[commutative diagrams/.cd,every arrow](245)--(345);%

    \path[commutative diagrams/.cd,every arrow](224)--(234);%
    \path[commutative diagrams/.cd,every arrow](225)--(235);%
    \path[commutative diagrams/.cd,every arrow](234)--(244);%
    \path[commutative diagrams/.cd,every arrow](235)--(245);%

    \path[commutative diagrams/.cd,every arrow](224)--(225);%
    \path[commutative diagrams/.cd,every arrow](234)--(235);%
    \path[commutative diagrams/.cd,every arrow](244)--(245);%

    \path[commutative diagrams/.cd,every arrow](235)--(335);%
    \path[commutative diagrams/.cd,every arrow](335)--(345);%

  \end{tikzpicture}
\end{center}
equipped with all possible commutativity relations. Finally, for $d=4$, we
obtain the $\kk$-algebra with generators
\[
  \begin{tikzcd}[column sep=small]
    \displaystyle\prod_{i\neq5}L_{i}\rar&\displaystyle\prod_{i\neq4}L_{i}\rar&\displaystyle\prod_{i\neq3}L_{i}\rar&\displaystyle\prod_{i\neq2}L_{i}\rar&\displaystyle\prod_{i\neq1}L_{i}
  \end{tikzcd}
\]
and such that all consecutive composites vanish. Note that this last quiver is
better drawn as a maximal path in a $4$-dimensional hypercube.

We continue with the $\kk$-algebras $\dgAus{n}{d}^\sharp\cong\Aus{n}{n-d}$; for
simplicity, we omit the shifts on the generators $\set{\L{I}^\sharp=\L*{I}[-\rk{I}]}$.
For $d=4$ we obtain the $\kk$-algebra with generators
\[
  \begin{tikzcd}[column sep=small]
    \displaystyle\prod_{i\neq1}L_{i-1,i}\rar&\displaystyle\prod_{i\neq2}L_{i-1,i}\rar&\displaystyle\prod_{i\neq3}L_{i-1,i}\rar&\displaystyle\prod_{i\neq4}L_{i-1,i}\rar&\displaystyle\prod_{i\neq5}L_{i-1,i}
  \end{tikzcd}
\]
and no relations between the above morphisms. For $d=3,2$ we obtain the
$\kk$-algebras with generators
\[
  \begin{tikzcd}[column sep=small]
    L_{23}\times L_{34}\times L_{45}\rar\dar&L_{12}\times L_{34}\times
    L_{45}\rar\dar&L_{12}\times
    L_{23}\times L_{45}\rar\dar&L_{12}\times L_{23}\times L_{34}\dar\\
    0\rar&L_{01}\times L_{34}\times L_{45}\rar\dar&L_{01}\times L_{23}\times
    L_{45}\rar\dar&L_{01}\times
    L_{23}\times L_{34}\dar\\
    &0\rar&L_{01}\times L_{12}\times L_{45}\rar\dar&L_{01}\times L_{12}\times L_{45}\dar\\
    &&0\rar&L_{01}\times L_{12}\times L_{23}\\
  \end{tikzcd}
\]
and
\begin{center}
  \tdplotsetmaincoords{-20}{0}%
  \begin{tikzpicture}[scale=3,tdplot_main_coords,every node/.style={scale=0.75}]
    \tdplotsetrotatedcoords{0}{60}{180}
    \begin{scope}[tdplot_rotated_coords]
      \node (123) at (1,2,3) {$\L{34}\times\L{45}$};%
      \node (124) at (1,2,4) {$\L{23}\times\L{45}$};%
      \node (125) at (1,2,5) {$\L{23}\times\L{34}$};%

      \node (133) at (1,3,3) {$0$};%
      \node (134) at (1,3,4) {$\L{12}\times\L{45}$};%
      \node (135) at (1,3,5) {$\L{12}\times\L{34}$};%

      \node (144) at (1,4,4) {$0$};%
      \node (145) at (1,4,5) {$\L{12}\times\L{34}$};%

      \node (224) at (2,2,4) {$0$};%
      \node (225) at (2,2,5) {$0$};%

      \node (234) at (2,3,4) {$\L{01}\times\L{45}$};%
      \node (235) at (2,3,5) {$\L{01}\times\L{34}$};%

      \node (244) at (2,4,4) {$0$};%
      \node (245) at (2,4,5) {$\L{01}\times\L{23}$};%

      \node (335) at (3,3,5) {$0$};%

      \node (345) at (3,4,5) {$\L{01}\times\L{12}$};%
    \end{scope}

    \path[commutative diagrams/.cd,every arrow](123)--(124);%
    \path[commutative diagrams/.cd,every arrow](123)--(133);%
    \path[commutative diagrams/.cd,every arrow](124)--(125);%
    \path[commutative diagrams/.cd,every arrow](124)--(134);%
    \path[commutative diagrams/.cd,every arrow](125)--(135);%
    \path[commutative diagrams/.cd,every arrow](133)--(134);%
    \path[commutative diagrams/.cd,every arrow](134)--(135);%
    \path[commutative diagrams/.cd,every arrow](134)--(144);%
    \path[commutative diagrams/.cd,every arrow](135)--(145);%
    \path[commutative diagrams/.cd,every arrow](144)--(145);%

    \path[commutative diagrams/.cd,every arrow](124)--(224);%
    \path[commutative diagrams/.cd,every arrow](125)--(225);%
    \path[commutative diagrams/.cd,every arrow](134)--(234);%
    \path[commutative diagrams/.cd,every arrow](135)--(235);%
    \path[commutative diagrams/.cd,every arrow](144)--(244);%
    \path[commutative diagrams/.cd,every arrow](145)--(245);%
    \path[commutative diagrams/.cd,every arrow](245)--(345);%

    \path[commutative diagrams/.cd,every arrow](224)--(234);%
    \path[commutative diagrams/.cd,every arrow](225)--(235);%
    \path[commutative diagrams/.cd,every arrow](234)--(244);%
    \path[commutative diagrams/.cd,every arrow](235)--(245);%

    \path[commutative diagrams/.cd,every arrow](224)--(225);%
    \path[commutative diagrams/.cd,every arrow](234)--(235);%
    \path[commutative diagrams/.cd,every arrow](244)--(245);%

    \path[commutative diagrams/.cd,every arrow](235)--(335);%
    \path[commutative diagrams/.cd,every arrow](335)--(345);%

  \end{tikzpicture}
\end{center}
with all possible commutativity relations. Finally, for $d=1$ we obtain the
$\kk$-algebra with generators
\[
  \begin{tikzcd}[column sep=small]
    L_{45}\rar&L_{34}\rar&L_{23}\rar&L_{12}\rar&L_{01}
  \end{tikzcd}
\]
and such that all consecutive composites vanish. Again, this last quiver is
better drawn as a maximal path in a $4$-dimensional hypercube.

\section{Partially wrapped Fukaya categories and models for Waldhausen $K$-theory}

In this section we provide an interpretation of the partially wrapped Fukaya categories
\[
  \WF{n}{d}=\W(\Sym{d}{\theDisk},\Stops[d][n]),\qquad d\geq1,\ n\geq0
\]
as the cells of a simplicial model for Waldhausen $K$-theory. This interpretation arises as an
immediate consequence of the results of this work combined with the results of \cite{DJW19b} which
provide a relation between the $d$-dimensional $S$-construction and the $d$-dimensional Auslander
algebras of type $\AA$. We use freely the language of $\infty$-categories \cite{Lur09} as well as
basic aspects of the theory of stable $\infty$-categories \cite{Lur17}.

\subsection{The $d$-dimensional Waldhausen $\wS$-construction}

\subsubsection{Stable $\infty$-categories v.s.~differential graded
  $\kk$-categories}

Recall that a pointed $\infty$-category $\A$ is \emph{stable} if it admits all finite (homotopy)
limits, all finite (homotopy) colimits, and the suspension functor
\[
  \functor[\Sigma]{\A}{\A,}\qquad\functor*{a}{0\amalg_a0}
\]
is an equivalence. If $\A$ is a stable $\infty$-category, then the homotopy
category $\Ho[\A]$, equipped with the suspension autoequivalence $\Ho[\Sigma]$,
is additive and can be endowed with a canonical triangulation. For this reason,
stable $\infty$-categories can be regarded as a refinement of Verdier's
triangulated categories. We recall from \cite{Coh13} that Lurie's differential
graded nerve \cite[Construction~1.3.1.6]{Lur17} yields a Quillen equivalence
between
\begin{itemize}
\item the homotopy theory of (small) idempotent-complete pre-triangulated
  differential graded $\kk$-categories (up to quasi-equivalence) and
\item the homotopy theory of (small) idempotent-complete \emph{$\kk$-linear}
  stable $\infty$-categories.
\end{itemize}
Thus, the theory of stable $\infty$-categories can be regarded as an extension
of the theory of differential graded categories which encompasses higher
categories which are not linear over any commutative ring, such as the stable
$\infty$-category of spectra. In addition, the language of $\infty$-categories
affords a powerful calculus of (homotopy) Kan extensions on which many of the
statements below are reliant upon.

Below, we identify differential graded $\kk$-categories with their differential
graded nerves without further mention. In this process, we implicitly replace
$\WF{n}{d}$ with a quasi-equivalent differential graded model.

\subsubsection{Waldhausen $K$-theory}

\begin{notation}
  For an $\infty$-category $\A$, we let $\A^\simeq\subseteq\A$ be the largest
  $\infty$-groupoid (=Kan complex) contained in $\A$. The passage
  $\A\mapsto\A^\simeq$ yields a right adjoint to the inclusion of
  $\infty$-groupoids into $\infty$-categories.
\end{notation}

An important invariant associated to a stable $\infty$-category $\A$ is its
Waldhausen $K$-theory space $K(\A)$. For example, if $R$ is a ring, then the
Waldhausen $K$-theory space $K(R)=K(\perf{R})$ of the stable $\infty$-category
$\perf{R}$ of perfect $R$-modules is homotopy equivalent to the algebraic
$K$-theory space of $R$ defined by Quillen \cite{Qui73} in terms of the
$Q$-construction of the (split-exact) category of finitely generated projective
$R$-modules.

Waldhausen's definition \cite{Wal85,BGT13} of the space $K(\A)$ involves the
construction of a simplicial $\infty$-groupoid $\wS[1][\bullet][\A]^\simeq$
whose geometric realisation is then the delooping of the Waldhausen $K$-theory
space of $\A$, that is
\[
  K(\A)\coloneqq\Omega^1|\wS[1][\bullet][\A]^\simeq|.
\]
More generally, for each integer $d\geq1$, Dyckerhoff~\cite{Dyc17} and
Poguntke~\cite{Pog17} introduce an analogous simplicial object
$\wS[d][\bullet][\A]$ such that
\[
  K(\A)\simeq\Omega^d|\wS[d][\bullet][\A]^\simeq|.
\]
For each $d\geq 1$, the \emph{$d$-dimensional Waldhausen $\wS$-construction
  $\wS[d][\bullet][\A]$ of $\A$} exhibits $K(\A)$ as the $d$-fold loop space of
a $(d-1)$-connected space or, equivalently, as a connective spectrum.
We remark that the simplicial object
\[
  \functor*[\wS[d][\bullet][\A]]{n}{\wS[d][n][\A]}
\]
takes its values in the $\infty$-category $\St$ of stable $\infty$-categories
and exact functors between them, see Section~1.4 in \cite{DJW19b} for details.

\begin{remark}
  The $d$-dimensional Waldhausen $\wS$-construction of a stable
  $\infty$-category is \emph{not} the $d$-fold iteration of the
  $\wS$-construction, which is a multi-simplicial object rather than a
  simplicial object.
\end{remark}

\subsection{The equivalence $\WF{n}{d}\simeq\wS[d][n]$}

The following theorem is an immediate consequence of Corollary~2.26 in
\cite{DJW19b} in the case ${\A=\perf{\kk}}$ and \Cref{thm:WF-Auslander} above. 

\begin{theorem}
  \label{thm:paracyclic_object}
  Let $d\geq1$. For each $n\geq0$, there are equivalences of $\infty$-categories
  \[
    \begin{tikzcd}
      \wS[d][n][\perf{\kk}]&\perf{\Aus{n}{d}}\lar[swap]{\simeq}\rar{\simeq}&\WF{n}{d}.
    \end{tikzcd}\qedhere
  \]
\end{theorem}

\begin{remark}
  Let $d\geq1$. The equivalent stable $\infty$-categories
  \[
    \wS[d][n][\perf{\kk}]\simeq\perf{\Aus{n}{d}}\simeq\WF{n}{d}
  \]
  vanish if $n<d$. This vanishing is closely related to the fact that the space
  $|\wS[d][\bullet][\perf{\kk}]^{\simeq}|$ is $(d-1)$-connected, which is to say
  that its homotopy groups vanish in degrees less than or equal to $d$.
\end{remark}

In fact, the $d$-dimensional Waldhausen $\wS$-construction of a stable
$\infty$-category can be canonically extended from a simplicial object to a
\emph{paracyclic} object, see Proposition~2.47 in \cite{DJW19b}. In view of
\Cref{thm:paracyclic_object}, this implies that the partially wrapped Fukaya
categories $\WF{n}{d}$, $n\geq0$, themselves arrange into a paracyclic object
\[
  \functor[\WF{\bullet}{d}]{\ParacyclicCat^\op}{\St^\kk}
\]
with values in the $\infty$-category of $\kk$-linear stable $\infty$-categories
and exact functors between them, where $\ParacyclicCat$ is the paracyclic
category \cite{Nis90,FL91,GJ93}. Passing to homotopy categories yields a
paracylic object
\[
  \functor[\Ho[\WF{\bullet}{d}]]{\ParacyclicCat^\op}{\Ho[\St^\kk]}
\]
with values in the homotopy category of the $\infty$-category $\St^\kk$. Such a
paracylic object $\Ho[\WF{\bullet}{d}]$ amounts to the data of exact functors
\[
  \functor[d_0]{\WF{n+1}{d}}{\WF{n}{d}}\qquand\functor[s_0]{\WF{n}{d}}{\WF{n+1}{d}}
\]
and an (exact) autoequivalence
\[
  \functor[t=t_{n+1}]{\WF{n}{d}}{\WF{n}{d}}[\simeq]
\]
for each $n\geq 0$; these functors induce additional distinguished functors
\[
  \functor[d_i=t^id_0t^{-i}]{\WF{n+1}{d}}{\WF{n}{d}}\qquand\functor[s_i=t^is_0t^{-i}]{\WF{n}{d}}{\WF{n+1}{d}}
\]
for each $n\geq0$ and each $1\leq i\leq n$. In the homotopy category
$\Ho[\St^\kk]$, the above functors must satisfy the simplicial identities
\begin{equation*}
  \begin{split}
    d_0s_0&=1,\\
    d_1s_0&=1,
  \end{split}\qquad\qquad
  \begin{split}
    d_0d_i&=d_{i-1}d_0,\quad 1\leq i\leq n,\\
    d_0s_i&=s_{i-1}d_0,\quad 1\leq i<n,\\
  \end{split}\qquad\qquad
  \begin{split}
    s_0s_i&=s_{i+1}s_0,\quad 0\leq i\leq n,\\
    d_is_0&=s_0d_{i-1},\quad 1<i\leq n,
  \end{split}
\end{equation*}
as well as the paracyclic identities
\begin{equation*}
  d_0t^{n+1}=t^nd_0\qquand s_0t^{n+1}=t^{n+2}s_0.
\end{equation*}
By definition, the above identities encode the mere existence of natural
isomorphism between the corresponding functors. In contrast, the paracyclic
object $\WF{\bullet}{d}$ is a functor of $\infty$-categories and hence it
contains an infinite amount of coherence data (including explicit natural
isomorphisms witnessing the validity of the identities above).

\begin{remark}
  Let $d\geq1$. It follows from our previous discussion that there is a homotopy
  equivalence
  \[
    K(\WF{d}{d})\simeq K(\perf{\kk})\simeq\Omega^d|\WF{\bullet}{d}|,
  \]
  where we remind the reader that $\WF{d}{d}$ is equivalent to $\perf{\kk}$.
  More generally, if $(M,\Stops)$ is an arbitrary Weinstein manifold with stops,
  there are equivalences of stable $\infty$-categories
  \begin{align*}
    \W(M\times\Sym{d}{\theDisk},  (\Stops\times \Sym{d}{\theDisk}) \cup (M \times \Stops[d][n])    ) &\simeq\Fun[\WF{n}{d}][\W(M,\Stops)][][\kk] \\ &\simeq\Fun[\Aus{n}{d}][\W(M,\Stops)][][\kk]\simeq\wS[d][n][\W(M,\Stops)].
  \end{align*}
  Indeed,
  \begin{itemize}
  \item the leftmost equivalence is a consequence of
    \begin{itemize}
    \item the K{\"u}nneth formula \cite[Corollary~1.11]{GPS18}
      \[
        \functor{\W(M,\Stops)\otimes_\kk\W(\Sym{d}{\theDisk},\Stops[d][n])}{\W(M\times\Sym{d}{\theDisk},   (\Stops\times \Sym{d}{\theDisk}) \cup (M \times \Stops[d][n]))}[\simeq]
      \]
    \item the equivalence
      \[
        \functor{\W(M,\Stops)\otimes_\kk\Fun[\WF{n}{d}][\perf{\kk}][][\kk]}{\Fun[\WF{n}{d}][\W(M,\Stops)][][\kk],}[\simeq]
      \]
    \item and the equivalences
      \[
        \Fun[\WF{n}{d}][\perf{\kk}][][\kk]\simeq(\WF{n}{d})^\op\simeq\WF{n}{d}=\W(\Sym{d}{\theDisk},\Stops[d][n]),
      \]
    \end{itemize}
    where the existence of the second and third equivalences is a consequence of
    the fact that $\WF{n}{d}\simeq\perf{\Aus{n}{d}}$ is a dualisable object of
    the symmetric monoidal $\infty$-category $(\St^\kk,\otimes_\kk,\perf{\kk})$
    since the higher Auslander algebra $\Aus{n}{d}$ is proper and homologically
    smooth over $\kk$;
  \item the middle equivalence stems from the fact that
    $\WF{n}{d}\simeq\perf{\Aus{n}{d}}$ is the $\kk$-linear stable hull of the
    $\kk$-algebra $\Aus{n}{d}$;
  \item the rightmost equivalence is a consequence of Propositions~2.10 and 2.24
    in \cite{DJW19b}.
  \end{itemize}
  Consequently, there are a paracyclic object
  \[
    \functor[\W(M\times\Sym{d}{\theDisk},(\Stops\times \Sym{d}{\theDisk})\cup (M \times\Stops[d][\bullet]))]{\ParacyclicCat^\op}{\St^\kk}
  \]
  and homotopy equivalences
  \[
    K(\W(M,\Stops))\simeq\Omega^d|\wS[d][\bullet][\W(M,\Stops)]^\simeq|\simeq\Omega^d|\W(M\times\Sym{d}{\theDisk}, (\Stops\times \Sym{d}{\theDisk}) \cup (M \times \Stops[d][\bullet]) )^\simeq|
  \]
  which describe the $d$-fold delooping of the Waldhausen $K$-theory space of
  the partially wrapped Fukaya category $\W(M,\Stops)$ by means of partially
  wrapped Fukaya categories of symmetric products of marked disks. Compare with
  Section~1.2 in \cite{Tan19} where the case $d=1$ is discussed.
\end{remark}

\subsubsection{Symplecto-geometric description of the structure maps of $\WF{\bullet}{d}$}

The generating structure maps of the paracyclic object $\WF{\bullet}{d}$ admit a natural
symplecto-geometric interpretation: Fix $n\geq 0$. Firstly, the paracyclic shift
\[
  \functor[t]{\WF{n}{d}}{\WF{n}{d}}[\simeq],
\]
can be identified with the autoequivalence $\mathfrak{r}^{-1}[d]$ of $\WF{n}{d}$
which, according to \Cref{prop:Serre}, is induced by rotating the disk clockwise
by an angle of $\frac{2\pi}{n+1}$ (here $\mathfrak{r}$ denotes the Serre functor
of $\WF{n}{d}$).

Secondly, for $0\leq i\leq n+1$, the face functor
\[
  \functor[d_i]{\WF{n+1}{d}}{\WF{n}{d}}[\simeq]
\]
can be identified with the stop-removal functor which removes the stop
$\set{p_i}\times\Sym{d}{\theDisk}$ from $\Stops[d][n+1]$. These functors are
described as follows: Consider the smallest idempotent-complete stable
subcategory of $\WF{n+1}{d}$ containing the objects of the form
\[
  L\times L_{i-1,i}
\]
where $L$ is a product of $d-1$ mutually disjoint arcs in
$\theDisk\setminus\Stops[d][n]$ which are also disjoint from the arc
$L_{i-1,i}$; here, we identify $\Stops[][n]$ with
$\Stops[][n+1]\setminus\set{p_i}$. Equivalently, this subcategory is the
essential image of the Orlov functor
\[
  \functor[\iota_i]{\WF{n}{d-1}}{\WF{n+1}{d},}\qquad\functor*{X}{X\times
    L_{i-1,i}.}
\]
The functor $d_i$ is defined in terms of the Verdier quotient\footnote{Under the
  differential graded nerve, the Verdier quotient of $\kk$-linear stable
  $\infty$-categories corresponds to the Drinfeld quotient \cite{Dri04} of
  differential graded $\kk$-categories. This is a consequence of the fact that
  both quotients are characterised as homotopy cofibres in the corresponding
  $\infty$-categories.}
\[
  \begin{tikzcd}
    \WF{n}{d-1}\rar[hookrightarrow]{\iota_i}&\WF{n+1}{d}\rar[two
    heads]{d_i}&\WF{n+1}{d}/\WF{n}{d-1}
  \end{tikzcd}
\]
of $\WF{n+1}{d}$ by the essential image of $\iota_i$. As a consequence of
general stop-removal theorems \cite{GPS18,Syl19,Syl19a} we can identify the
target $\WF{n+1}{d}/\WF{n}{d-1}$ of the localisation functor $d_i$ with the
partially wrapped Fukaya category $\WF{n}{d}$.

Let $1\leq i\leq n$. Finally, we describe the degeneracy functor
\[
  \functor[s_i]{\WF{n}{d}}{\WF{n+1}{d}.}
\]
For this, we introduce a new stop into $\Stops[d][n]$ by adding positive
push-off of the component $\set{p_i}\times\Sym{d-1}{\theDisk}$ in the direction
of the Reeb flow on the disk. That is, we consider the new set of stops
\[
  \Stops[d][n+1]=\Stops[d][n]\cup\set{\set{p_i^{\varepsilon}}\times\Sym{d-1}{\theDisk}}
\]
where $p_i^{\varepsilon}$ is a point on the boundary of the disk obtained by
rotating $p_i$ by a small angle $\varepsilon>0$ in counter-clockwise direction.
The above construction gives rise to a pushforward functor $s_i$ with the
desired source and target categories. The restriction of the functor $s_i$ to
the full subcategory of $\WF{n}{d}$ spanned by those objects given by products
of disjoint arcs in $\theDisk\setminus\Stops[][n]$ is the identity functor, from
which it readily follows that $s_i$ is fully faithful. This is a consequence of
the fact that, at the geometric level, a product of arcs in
$\theDisk\setminus\Stops[][n]$ can also be seen as a product of arcs in
$\theDisk\setminus(\Stops[][n]\cup\set{p_i^\varepsilon})$; as can be seen from
Auroux's description of morphisms in terms of strand diagrams, the fact that
$\varepsilon$ is chosen to be sufficiently small guarantees that $s_i$ induces
an \emph{isomorphism} between the cochain complexes of morphisms between such
objects. The construction of the functor $s_i$ should be compared with the
forward stopped inclusions introduced in \cite{GPS18}.

\subsubsection{Relation to the categorified Dold--Kan correspondence}

As a consequence of a categorified version of the Dold--Kan correspondence
\cite{Dyc17}, the $d$-dimensional Waldhausen $\wS$-construction of $\perf{\kk}$
is characterised, as a $2$-simplicial object, by the existence of equivalences
of stable $\infty$-categories $\wS[d][0][\A]\simeq0$ and
\[
  \bigcap_{i=1}^{n}\ker(d_i\colon\wS[d][n][\A]\twoheadrightarrow\wS[d][n-1][\A])\simeq\begin{cases}
    \perf{\kk}&\text{if }n=d,\\
    0&\text{otherwise},
  \end{cases}
\]
for $n\geq1$. These vanishing conditions have a natural interpretation in terms
of the stop-removal functors described above.

\begin{proposition}
  Let $d\geq 1$. There are equivalences of stable $\infty$-categories
  $\WF{d}{0}\simeq0$ and
  \[
    \WF*{n}{d}\coloneqq\bigcap_{i=1}^{n}\ker(d_i\colon\WF{n}{d}\twoheadrightarrow\WF{n-1}{d})\simeq\begin{cases}
      \perf{\kk}&\text{if }n=d,\\
      0&\text{otherwise},
    \end{cases}
  \]
  for $n\geq1$.
\end{proposition}
\begin{proof}
  The claim is obvious for $n<d$ since $\WF{n}{d}$ vanishes in this case. For
  $n=d$ there is an equivalence
  \[
    \perf{\kk}\simeq\WF{d}{d}=\WF*{n}{d}=\bigcap_{i=1}^d\ker(d_i\colon\WF{d}{d}\twoheadrightarrow\WF{d-1}{d})
  \]
  since $\WF{d}{d-1}\simeq0$ and the higher Auslander algebra $\Aus{d}{d}$ is
  isomorphic to the base commutative ring $\kk$ in this case. It remains to
  prove that the intersection
  \[
    \WF*{n}{d}=\bigcap_{i=1}^{n}\ker(d_i\colon\WF{n}{d}\twoheadrightarrow\WF{n-1}{d})=\bigcap_{i=1}^{n}\operatorname{im}(\iota_i\colon\WF{n-1}{d-1}\hookrightarrow\WF{n}{d}),
  \]
  vanishes for $n>d$, where $\iota_i$ is the Orlov functor corresponding to the
  point $p_i\in\Stops[][n]$. This is clear since $\WF*{n}{d}$ is generated by
  Lagrangians of the form $\prod_{i=1}^dL_i$ where $L_1,\dots,L_d$ are pairwise
  non-intersecting arcs in $\theDisk\setminus\Stops[][n]$ which must be jointly
  supported near all the stops $p_1,\dots,p_n$; but the assumption that $n>d$
  implies that no such a collection of arcs exists. The claim follows.
\end{proof}

\begin{remark}
  An interesting challenge---which we do not pursue here---is to extend the
  above descriptions to a construction of the paracyclic object
  \[
    \functor[\WF{\bullet}{d}]{\ParacyclicCat^\op}{\St^\kk}
  \]
  carried entirely within the framework of partially wrapped Fukaya categories
  (see~\cite{Tan19} for a related discussion in the case $d=1$). In particular,
  notice that we have not described the higher order components of the
  paracyclic object $\WF{\bullet}{d}$ nor have we given symplectic explanations
  of the fact that these satisfy the required coherence equations.
\end{remark}


\providecommand{\bysame}{\leavevmode\hbox to3em{\hrulefill}\thinspace}
\providecommand{\MR}{\relax\ifhmode\unskip\space\fi MR }
\providecommand{\MRhref}[2]{%
  \href{http://www.ams.org/mathscinet-getitem?mr=#1}{#2}
}
\providecommand{\href}[2]{#2}

\end{document}